\documentclass{amsart}

\usepackage{amsthm,amsmath,amssymb,amsfonts}
\usepackage[utf8]{inputenc}
\usepackage{verbatim}
\usepackage{marvosym}
\usepackage{stackrel}
\usepackage{graphicx}
\usepackage{stmaryrd}
\usepackage[svgnames]{xcolor}
\usepackage{lipsum}
\usepackage{svg}
\usepackage{bbm}
\usepackage{caption}
\usepackage{enumitem}
\usepackage{epigraph}
\usepackage{color}
\usepackage{csquotes}
\usepackage{microtype}
\usepackage{relsize}
\usepackage{amsfonts}
\usepackage{adjustbox}
\usepackage{subfig}
\usepackage{framed}
\usepackage{hyperref}
\hypersetup{
	colorlinks = true,
	linkbordercolor = {red},
	linkcolor ={teal},
	anchorcolor = {pink},
	citecolor =  {orange},
	filecolor = {teal},
	menucolor = {teal},
	runcolor =  {teal},
	urlcolor = {teal},
}
\usepackage[capitalize]{cleveref}
\usepackage{bussproofs} 
\usepackage{tcolorbox} 
\usepackage[draft]{optional}
\usepackage[colorinlistoftodos,prependcaption,textsize=tiny]{todonotes}

\usepackage{mathbbol}
\DeclareSymbolFontAlphabet{\amsmathbb}{AMSb}%
\usepackage{tikz-cd}
\tikzcdset{scale cd/.style={every label/.append style={scale=#1},
		cells={nodes={scale=#1}}}}
\newcounter{nodemaker}
\tikzset{%
	symbol/.style={%
		draw=none,
		every to/.append style={%
			edge node={node [sloped, allow upside down, auto=false]{$#1$}}}
	}
}
\usepackage{quiver}

\newcommand{\hirayyyy}{\text{\usefont{U}{min}{m}{n}\symbol{'110}}}
\DeclareFontFamily{U}{min}{}
\DeclareFontShape{U}{min}{m}{n}{<-> dmjhira}{}

\newcommand{\yo}{\hirayyyy}

\usepackage{mathrsfs}

\makeatletter
\newcommand{\big@doubleop}[1]{%
	\DOTSB\mathop{\mathpalette\big@doubleop@aux{#1}}\slimits@
}
\newcommand*{\doublerightarrow}[2]{\mathrel{
		\settowidth{\@tempdima}{$\scriptstyle#1$}
		\settowidth{\@tempdimb}{$\scriptstyle#2$}
		\ifdim\@tempdimb>\@tempdima \@tempdima=\@tempdimb
		\mathop{\vcenter{
				\offinterlineskip\ialign{\hbox to\dimexpr\@tempdima+1em{##}\cr
					\rightarrowfill\cr\noalign{\kern.5ex}
					\rightarrowfill\cr}}}\limits^{\!#1}_{\!#2}}}
\newcommand*{\triplerightarrow}[1]{\mathrel{
		\settowidth{\@tempdima}{$\scriptstyle#1$}
		\mathop{\vcenter{
				\offinterlineskip\ialign{\hbox to\dimexpr\@tempdima+1em{##}\cr
					\rightarrowfill\cr\noalign{\kern.5ex}
					\rightarrowfill\cr\noalign{\kern.5ex}
					\rightarrowfill\cr}}}\limits^{\!#1}}}
\makeatother

\theoremstyle{definition}
\newtheorem{thm}{Theorem}[subsection]
\newtheorem*{thm*}{Theorem}
\newtheorem{prop}[thm]{Proposition}
\newtheorem{lem}[thm]{Lemma}
\newtheorem{cor}[thm]{Corollary}
\newtheorem{defn}[thm]{Definition}
\newtheorem{rem}[thm]{Remark}
\newtheorem{constr}[thm]{Construction}
\newtheorem{exa}[thm]{Example}

\newtheorem{notat}[thm]{Notation}
\newtheorem{disc}[thm]{}

\newtheorem{disclaim}[thm]{Remark}
\newtheorem{req}[thm]{Requirement}
\usepackage[english]{babel}
\usepackage{ifpdf,ifxetex,ifluatex}
\relax\ifluatex
\usepackage{firamath-otf}
\usepackage[default]{FiraSans}
\else
\usepackage{FiraSans}
\fi
\usepackage{roboto}
\usepackage{graphicx}
\usepackage{amsmath,tikz-cd,amsthm}
\edef\cdrestoreat{
\noexpand\catcode\lq\noexpand\@=\the\catcode\lq\@}\catcode\lq\@=11
\def\@makechapterhead#1{%
  \vspace*{50\p@}%
  {\parindent \z@ \raggedright \robotoslab
    \ifnum \c@secnumdepth >\m@ne
      \if@mainmatter
        \Huge{\bfseries \@chapapp\space \thechapter}
        \par\nobreak
        \vskip 20\p@
      \fi
    \fi
    \interlinepenalty\@M
    \huge #1\par\nobreak
    \vskip 40\p@
  }}
\def\@makeschapterhead#1{%
  \vspace*{50\p@}%
  {\parindent \z@ \raggedright
    \robotoslab
    \interlinepenalty\@M
    \Huge \bfseries  #1\par\nobreak
    \vskip 40\p@
  }}
\cdrestoreat

\newcommand{\noop}[1]{}
\newcommand{\quotmarks}[1]{``#1''}

\newcommand{\ie}{\textit{i.e.}\ }

\newcommand{\catof}[1]{\mathbf{#1}}
\newcommand{\ctg}[1]{\mathcal{#1}}

\newcommand{\cod}{\mathrm{cod}}
\newcommand{\dom}{\mathrm{dom}}
\newcommand{\id}{\mathrm{id}}
\newcommand{\Id}{\mathrm{Id}}

\newcommand{\due}{^{\mathsf{2}}}

\newcommand{\opp}{^{\mathsf{op}}}

\newcommand{\catcat}{\catof{Cat}}

\newcommand{\pr}{\mathrm{pr}}

\newcommand{\tand}{\quad\text{and}\quad}

\newcommand{\disp}{\mathrm{disp}}

\newcommand{\duu}{\dot{\ctg{U}}}
\newcommand{\uu}{\ctg{U}}
\newcommand{\du}{\dot{u}}
\newcommand{\ctx}{\ctg{B}} 
\newcommand{\ff}{\ctg{F}}
\newcommand{\ee}{\ctg{E}}

\newcommand{\bb}{\ctg{B}}

\newcommand{\la}{[}
\newcommand{\ra}{]}

\renewcommand{\ulcorner}{\lrcorner}

\usepackage{afterpage}

\newcommand{\hirayo}{\yo}

\renewcommand{\mathbb}{\ctg}
\newcommand{\classof}[1]{\mathscr{#1}}
\renewcommand{\ctx}{\catof{ctx}}
\renewcommand{\ff}{\ctg{P}}

\begin{document}

\title{Context, judgement, deduction}

\author{Greta Coraglia}
\address{LUCI Lab, Department of Philosophy, University of Milan, University of Milan, Via Festa del Perdono 7, 20122 Milano, Italy}
\email{greta.coraglia@unimi.it}

\author{Ivan Di Liberti}
\address{Department of Mathematics, Stockholm University, Stockholm, Sweden}
\email{diliberti.math@gmail.com}

\date{}


\begin{abstract}
 We introduce judgemental theories and their calculus as a general framework to present and study deductive systems. As an exemplification of their expressivity, we approach dependent type theory and natural deduction as special kinds of judgemental theories. Our analysis sheds light on both the topics, providing a new point of view. In the case of type theory, we provide an abstract definition of extensional type constructor featuring the usual formation, introduction, elimination and computation rules. For natural deduction we offer a deep analysis of structural rules, putting them into context. We finish the paper discussing the internal logic of a topos, a predicative topos, an elementary $2$-topos et similia, and show how these can be organized in judgemental theories.
 
  
\end{abstract}

\maketitle
{  \hypersetup{linkcolor=black}
\setcounter{tocdepth}{1}
\tableofcontents
}

\epigraph{Everything that can be thought at all can be thought clearly. Everything that can be said can be said clearly.}{\cite[4.116]{witt}}

\section*{Introduction}
\subsection*{General discussion}
As the title advertises, this paper is concerned with the notions of \textit{context, judgement} and \textit{deduction}. These three notions belong to Logic, but different communities with different backgrounds and cultures have quite different perspectives on them. The purpose of this work is to present a mathematical and unified approach that accommodates these diverse takes on the topic of \textit{deduction}.
In order to show how tightly this new framework captures the motions of deductive systems, we develop two applications of the theory we introduce, and since our choice is intended to pay tribute to two major cultures concerned with the topic, we look at examples from type theory \cite{martin1984intuitionistic} and from proof theory \cite{troelstra2000basic}. They each stand on different {\it conceptual} grounds, as it is exemplified by the two following rules.

\begin{minipage}{0.49\textwidth}
    \begin{prooftree}\AxiomC{$\Gamma\vdash a:A$}\AxiomC{$\Gamma.A\vdash B\; {\tt Type}$}\LeftLabel{(DTy)}\BinaryInfC{$\Gamma\vdash B \la a \ra \; {\tt Type}$}\end{prooftree}
 \end{minipage}
\begin{minipage}{0.49\textwidth}
     \begin{prooftree}\AxiomC{$x;\Gamma \vdash \phi$}\AxiomC{$x;\Gamma,\phi \vdash \psi$}\LeftLabel{(Cut)}\BinaryInfC{$x;\Gamma \vdash \psi$}\end{prooftree}
 \end{minipage}
 \vspace{.08cm}

\noindent Despite their incredibly similar look, and the somehow parallel development of the theories in the same notational framework, there are some philosophical differences between the interpretation of the symbols above.
\begin{enumerate}
    \item[TT)] In type theories, especially those inspired by the reflections of Martin-Löf, $\Gamma \vdash B\; {\tt Type}$ is intuitively seen as a \textit{judgement}. A judgement is an act of knowledge \cite{martin1996meanings,martin1987truth} bound to a context $(\Gamma)$ and pertinent to an object $(B)$. For example, $\Gamma \vdash B\; {\tt Type}$ could be read \textit{Given} $\Gamma$, $B$ \textit{is a type}. The ontological status of a context and an object is, in principle, very different. Also, and most notably, judgements can be of different kinds, claiming all sorts of possible things about their objects \cite{martin1996meanings}.
    \item[ND)] In natural deduction, $x;\Gamma \vdash \psi$ is intuitively seen as a \textit{consecution}. A consecution is a relation between \textit{structured formulae} ($x;\Gamma$) and \textit{formulae} ($\psi$) \cite{kleenemathlog}. For example, the sequent  $x;\Gamma \vdash \psi$ could be read \textit{The multiset of formulae $\Gamma$, in the variables $x$, entails $\psi$}. Besides the fact that structured formulae are multisets of formulae, there isn't an ontological difference between the glyphs appearing on the left and right side of the entailment.
\end{enumerate}

These differences, though admittedly subtle and not that easy to detect on a technical level, dictate a part of the experts' intuition on the topics (see \Cref{prolegomena}). Of course, one could argue that these different points of view are mostly philosophical, that the oversimplification commanded by the length of this introduction stresses on them in a somewhat artificial way, and that some variations are allowed, for example \cite{negri2008structural} adopts what we would call a more type theoretic perspective on proof theory, and indeed it is always possible to adopt a judgemental perspective on consecutions. In particular, the deep connection between proof theory and type theory has of course been studied for a while, and its development falls under the paradigm that is mostly known as {\it propositions-as-types} \cite{wadler2015propositions}, and how to translate intuitionistic natural deduction into the language of types is beautifully described in \cite{Martin-Lof1996-MAROTM-7}. This work aims at providing a new, perhaps more semantic, argument in the same unifying direction.

Rebooting some ideas from \cite{jacobs1999categorical}, we conciliate the differences in a unified categorical framework that can highlight and clarify in a more precise way the meaning of all these apparently specific phenomena. Going back to the example of (DTy) and (Cut), we intuitively see how they both fit the same paradigm, in the sense that we could read both as instances of the following syntactic string of symbols
\vspace{-0.5cm}
\begin{prooftree}\AxiomC{$\heartsuit \; \vdash \blacksquare $}\AxiomC{$\square \vdash \clubsuit $}\LeftLabel{$(\triangle)$}\BinaryInfC{$\heartsuit \; \vdash \spadesuit$}\end{prooftree}
which we usually parse as: \textit{by} $\triangle$, \textit{given} $\heartsuit \; \vdash \blacksquare$ \textit{and} $\square \vdash \clubsuit$ \textit{we deduce} $\heartsuit \; \vdash \spadesuit$.
Our theory allows for a coherent expression of {\it all} such strings of symbols, and shows how a suitable choice of {\it context} either produces (DTy) or (Cut). As a necessary biproduct of our effort, we get a theory that has both the advantage of being very versatile, spanning much farther than dependent types and natural deduction, and computationally meaningful in the sense that it has a built-in notion of computation. 

\subsection*{Our contribution}

We introduce the notion of judgemental theory using the language of category theory. Judgemental theories are philosophically inspired by Martin-Löf's reflections on the topic of judgement \cite{martin1987truth,martin1996meanings}, and technically grounded on recent developments in the categorical treatment of dependent type theory \cite{awodey_2018,uemura2019general}. We believe that our perspective is also very attuned to \emph{type refinement systems} as described in \cite{10.1145/2676726.2676970}.
\begin{prooftree}\AxiomC{$\Gamma \vdash H \;\; \mathbb{H} \qquad \lambda H \vdash F\;\; \mathbb{F}$}\end{prooftree}
Usually, when looking at the premises of a rule, we are confronted with a list of (nested) judgements as above, which then are transformed into another judgement by the rule.  The technical advantage of our notion is to allow natively for \textit{nested judgements}. That is, for us a nested family of judgements is actually a whole judgement \textit{per se}:
\begin{prooftree}\AxiomC{$\Gamma \vdash H.\lambda F\;\; \mathbb{H}.\lambda\mathbb{F}\;.$}\end{prooftree}
This flexibility allows for an algebraic treatment of extensional type constructors (and of connectives), in a fashion that is somewhat inspired by Awodey's natural models. Judgment classifiers will be categories living over contexts, and functors between them will regulate deduction rules. In the example below, which is the formation rule for the $\Pi$-constructor in dependent type theory, the category $\mathbb{U}.\Delta \mathbb{U}$ classifies the nested judgement in the premise of the rule (on the right).

 \begin{minipage}{0.39\textwidth}
\[\begin{tikzcd}
	{\mathbb{U}.\Delta{\mathbb{U}}} && {\mathbb{U}} \\
	& {\ctx}
	\arrow["\Pi"{description}, color={rgb,255:red,167;green,42;blue,42}, from=1-1, to=1-3]
	\arrow[from=1-1, to=2-2]
	\arrow[from=1-3, to=2-2]
\end{tikzcd}\]
 \end{minipage}
 \begin{minipage}{0.60\textwidth}
     \begin{prooftree}\AxiomC{$\Gamma \vdash A\; {\tt Type}$}\AxiomC{$\Gamma.A \vdash B\; {\tt Type}$}\LeftLabel{($\Pi$\textsc{F})}\BinaryInfC{$\Gamma \vdash \Pi_A B \; {\tt Type}$}\end{prooftree}
 \end{minipage}
 \vspace{.05cm}

 We expand the original approach à la Jacobs, where some of these ideas were evidently hinted at both in the treatment of propositional logic \cite[Chapter 2]{jacobs1999categorical}, and in the treatment of type theories \cite[Chapter 10]{jacobs1999categorical}. It also expands Awodey's natural models \cite{awodey_2018}, taking very seriously his algebraic presentation of some constructors: see for example the  discussion at page 9 and later Prop. 2.4 in loc.\! cit.

 It should be noted that Logical Frameworks \cite{harper1993framework} of Plotkin et al.\! have a similar purpose, and their system is based on $\lambda\Pi$-calculus. Logical Frameworks do rely on the notion of judgement in a substantial way, as we do, but their approach is somewhat much more syntactic. More recently \cite{uemura2019general} has provided recipes to transform Logical Frameworks into more categorical gadgets, based on a generalization of Awodey's natural models. One of the advantages of our approach is to avoid the complexity of Logical Frameworks (and thus of Uemura's recipe), substituting it with a native categorical language.

 In the next subsection we will discuss in detail all the achievements of this structure on a technical level, though we end this very qualitative discussion with a bird's-eye view on a list of advantages of our system.

\begin{enumerate}
    \item We provide an algebraic approach to the notion of rule, which is also suitable for an analysis of the proof theory associated to a deductive system.
    \item In the case of type theory, this provides a clear definition of extensional type constructor, which was actually not available before, if not in a case-by-case form. We put into perspective the usual paradigm of rules (formation, introduction, elimination, $\beta$- and $\eta$-computation).
    \item In a similar spirit, we provide an in depth analysis of what constitutes what in proof theory are called structural rules. We here see that branches in proof trees are cones in our framework.
    \item We introduce the notion of \textit{policy} for a judgemental theory, inspired by the classical Cut of the Gentzen calculus. Surprisingly, type dependency in dependent type theory is precisely a type theoretic form of Cut.
    \item Our proofs are computationally meaningful. In a sense, this is due to the structural rigidity and the algebraicity of the framework. Each of our proofs needs to be as atomized and as transparent as possible, this will be particularly evident in our analysis of the proof theory generated by a dependent type theory with $\Pi$-types. The whole framework feels like a categorical proof assistant when doing proofs. 
\end{enumerate}

While this introduction seems to focus mainly on dependent type theories and natural deduction, the reader will notice that we have only chosen these two specific frameworks as an \textit{exemplum} of the expressive power of this theory. Indeed we could have covered modal logic, infinitary logics and much more.
\subsection*{Structure and main results}

\subsubsection*{\Cref{interpreting}} After the definition of (pre)judgemental theory (\Cref{jt}), we introduce \textit{judgemental theories} (\Cref{mainjt}), these are the mathematical gadgets which the whole paper is built on. In a nutshell, judgemental theories are pre-judgemental theories closed under a family of categorical constructions which will modulate the deductive power of our logical systems.

\subsubsection*{\Cref{calculi}} Each judgemental theory has an associated judgement calculus, which is a graphical bookkeeping of the categorical properties of the judgemental theory. In \Cref{prolegomena} we go through a critical analysis of the calculi of a dependent type theory and of natural deduction, both to highlight their main features and to have an inspirational account on what a calculus \textit{should look like}. Then, in \Cref{syn}, \Cref{judgement} and \Cref{rulesss} we declare the dictionary to translate a judgemental theory into its judgement calculus.

\subsubsection*{\Cref{secdtt}} In this section we show how to recover dependent type theories in our framework. After a definition of (judgemental) dependent type theory (\Cref{jdtt}), we discuss its relation with natural models à la Awodey (\Cref{awodeytous}) and with comprehension categories à la Jacobs (\Cref{tojacobs}). The rest of the section is dedicated to showing that the judgement calculus of a dependent type theory recovers the usual calculus of a dependent type theory (dtt).  In \Cref{dictionnaire} we declare a dictionary to convert our notation in the standard notation of dtts. In \Cref{cetd} we recover context extension and type dependency, in \Cref{pity} $\Pi$-types, in \Cref{idty} $\mathsf{Id}$-types. One can look at these subsections as a translation of the main results of \cite{awodey_2018} in our language. Yet, our proofs are much more synthetic and computationally meaningful, as they need to be incredibly atomized.  \Cref{phity} introduces a general notion of extensional type constructor, featuring the usual \textit{formation, introduction, elimination} and \textit{computation} rules. On a technical level, this is one of the most significant contributions of the paper, and of course the subsections \ref{pity} and \ref{idty}, could be a corollary of this subsection. To clarify this, we recover unit types and $\Sigma$-types as a corollary of \Cref{phity}. Because a judgement calculus is essentially only a representation of the judgemental theory, our technology is much more than \textit{a model of} dependent type theory, or a \textit{categorical semantics} for calculi, it is intrinsically a very syntactic object.

\subsubsection*{\Cref{secnat}} We then provide the correct judgemental infrastructure to sustain {\it natural deduction} for {\it first order logic}. Our exposition follows that in \Cref{secdtt}, meaning that we start from a (pre)judgemental theory (\Cref{jgen}), for readability reasons we pin-point a dictionary (\Cref{dict_gentzen}), and discuss which rules it generates, depending on the axioms we put on the judgemental system. In \Cref{doctrines} and \Cref{weak} we translate logical structures which are traditionally coded via doctrines into the fibrational setting, meaning the treatment of connectives and that of weakening, respectively. In \Cref{structrules} we provide evidence for structural rules, then for connectives in \Cref{connrules}. We add quantifiers in \Cref{fojndt} and show that we have rules regulating them in \Cref{quantrules}.  One of  our novelties emerges in \Cref{proptorules}, in \Cref{simplefib}, where we give a glimpse of the introduction of {\it monads as modalities} in the context of judgemental theories, and in the {\it treatment} of the cut rule (see \Cref{cutelim}). The correspondence between proof trees and (co)cones in the theory is very telling (see \Cref{lovebranches}).

\subsubsection*{\Cref{topos}} This section frames the internal logic of a topos-like category in the language of dependent type theories. Our notion is perfectly suited to present the Mitchell-Bénabou language of a topos (\Cref{elementary}). In \Cref{predicative} we introduce a notion of predicative (elementary) topos and show that it supports an essentially identically expressive internal logic, encoded by a dependent type theory. All infinitary pretopoi that we are aware of fall into our assumptions. \Cref{2topos} brings the previous discussion to the internal logic of an elementary $2$-topos in the sense of Weber, and shows that its internal logic can also be organized via a dependent type theory. Our treatment perfects that of Weber handling size issues in a more precise way. 

\subsection*{Acknowledgements}
The authors are especially grateful to \textit{Nathanael Arkor}, \textit{Jacopo Emmenegger} and \textit{Francesco Dagnino} for their comments and for their guidance through the literature. We are indebted to \textit{Pino Rosolini}, \textit{Milly Maietti} and \textit{Mike Shulman} for inspiring discussions. Both authors are grateful to the anonymous referee for their comments, which improved the presentation of the paper.

For part of this work the first author was supported by the BRIO \quotmarks{Bias, Risk and Opacity in AI} PRIN project (n.2020SSKZ7R) and by the Departments of Excellence 2023-2027 initiative, awarded by the Italian Ministry of Education, Universities and Research (MIUR).  Additionally, the first author would like to thank the University of Genova's PhD programme for supporting them during a large part of this project. The second author was supported by the Swedish Research Council (SRC, Vetenskapsrådet) under Grant No. 2019-04545. The research has received funding from Knut and Alice Wallenbergs Foundation through the Foundation’s program for mathematics.

\section{Judgemental theories} \label{interpreting}

\begin{defn}[Pre-judgemental theory] \label{jt}
A \textit{pre-judgemental theory} $(\ctx, \classof{J}, \classof{R}, \classof{P})$ of (\emph{contexts}, \emph{judgements}, \emph{rules}, \emph{policies}) is specified by the following data:
\begin{itemize}
    \item[($\ctx$)] a category (with terminal object $\diamond$);
    \item[($\classof{J}$)] a set of functors $f: \mathbb{F} \to \ctx$ over the category of contexts;
    \item[($\classof{R}$)] a set of functors $\lambda : \mathbb{F} \to \mathbb{G}$.
    \item[($\classof{P}$)] a set of $2$-dimensional cells filling (some) triangles induced by the rules (functors in $\classof{R})$ and the judgements (functors in $\classof{J})$, as in the diagrams below.
\end{itemize}
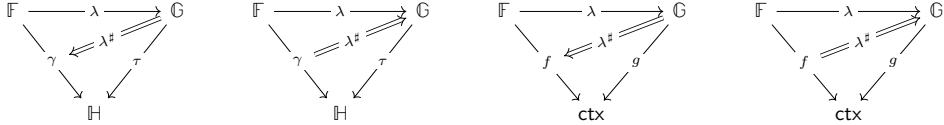
\adjustbox{max width=\textwidth}{
\begin{tikzcd}
	{\mathbb{F}} && {\mathbb{G}} & {\mathbb{F}} && {\mathbb{G}} & {\mathbb{F}} && {\mathbb{G}} & {\mathbb{F}} && {\mathbb{G}} \\
	\\
	& {\mathbb{H}} &&& {\mathbb{H}} &&& {\ctx} &&& {\ctx}
	\arrow[""{name=0, anchor=center, inner sep=0}, "\gamma"{description}, from=1-1, to=3-2]
	\arrow["\tau"{description}, from=1-3, to=3-2]
	\arrow["\lambda"{description}, from=1-1, to=1-3]
	\arrow[""{name=1, anchor=center, inner sep=0}, "\gamma"{description}, from=1-4, to=3-5]
	\arrow["\tau"{description}, from=1-6, to=3-5]
	\arrow["\lambda"{description}, from=1-4, to=1-6]
	\arrow[""{name=2, anchor=center, inner sep=0}, "f"{description}, from=1-7, to=3-8]
	\arrow["g"{description}, from=1-9, to=3-8]
	\arrow["\lambda"{description}, from=1-7, to=1-9]
	\arrow[""{name=3, anchor=center, inner sep=0}, "f"{description}, from=1-10, to=3-11]
	\arrow["g"{description}, from=1-12, to=3-11]
	\arrow["\lambda"{description}, from=1-10, to=1-12]
	\arrow["{\lambda^\sharp}"{description}, shorten >=8pt, Rightarrow, from=1-3, to=0]
	\arrow["{\lambda^\sharp}"{description}, shorten <=8pt, Rightarrow, from=1, to=1-6]
	\arrow["{\lambda^\sharp}"{description}, shorten <=8pt, Rightarrow, from=3, to=1-12]
	\arrow["{\lambda^\sharp}"{description}, shorten >=8pt, Rightarrow, from=1-9, to=2]
\end{tikzcd}
}
\end{defn}

\begin{notat} \label{mainnotation}
Let us introduce a bit of terminology:
\begin{itemize}
\item  \textit{contexts} $\Gamma, \Theta$ are objects of $\ctx$, morphisms $\sigma:\Theta\to\Gamma$ are {\it substitutions};
\item the element $g: \mathbb{G} \to \ctx$ of $\classof{J}$ is the \textit{classifier of the judgement} $\mathbb{G}$. We will often blur the distinction between the classifier and its judgement. In general we use letters such as $\mathbb{F, G,H}$;
\item objects $G$ in $\mathbb{G}$ are usually named after corresponding letter;
\item a \textit{rule} $\lambda$ is an element of $\classof{R}$;
\item a \textit{policy} $\lambda^\sharp$ is an element of $\classof{P}$;
\end{itemize}
and for special judgemental theories that happen to have an established notation we declare a switch of notation in the appropriate section.
\end{notat}

This is all the syntactic data needed to describe deduction: judgement classifiers prescribe the status of objects with respect to contexts; rules transform objects into other objects, with the context changing accordingly; and policies allow for the possibility that the context of the premise of the rule and that of the consequent are somehow naturally related, either covariantly (i.e. $\lambda^{\sharp}:f\Rightarrow g\lambda$), contravariantly (i.e. $\lambda^{\sharp}:g\lambda\Rightarrow f$), or constantly (i.e. when the triangle is strictly commutative) with respect to the direction of the rule. A bird's eye view of this first definition and what it might have been can be found in \Cref{futurejt}.

\begin{exa}[Toy Martin-Löf type theory]\label{toytt}
In order to get acquainted with the definition, let us introduce the categorical syntax to present a toy type theory. Consider a category $\ctx$ of contexts and substitutions, $\uu$ a category (universe) of types and $\duu$ a category (universe) of terms. For simplicity, we imagine that a term is always registered together with its type, so that objects of $\duu$ are of the form $(a,A)$ with $A$ an object in $\uu$. Define the pre-judgemental theory having $\mathcal{J}=\{u,\du\}$, $\mathcal{R}=\{\Sigma\}$, $\mathcal{P}=\{\Id: u\circ \Sigma\Rightarrow\du \}$ as below.
\[\begin{tikzcd}[ampersand replacement=\&]
	\duu \&\& \uu \\
	\& \ctx
	\arrow["u", to=2-2, from=1-3]
	\arrow[""{name=0, anchor=center, inner sep=0}, "\du"', from=1-1, to=2-2]
	\arrow["\Sigma", from=1-1, to=1-3]
	\arrow["\Id"{description}, shorten <=8pt, shorten >=8pt, Rightarrow, from=1-3, to=0]
\end{tikzcd}\]
Intuitively, $\du$ classifies terms with their context, $u$ does the same for types, $\Sigma$ performs typing, meaning it is the second projection, and $\Id$ shows that such an operation preserves the context.

We know that this all looks very unorthodox. We will use this toy example to get a first small impression on how judgemental theories work, see \ref{toytt.j}, \ref{toytt.r}, and above all \Cref{secdtt}.
\end{exa}

On the data expressed by a pre-judgemental theory we wish to impress some deductive power. This is achieved using some 2-categorical constructions and properties.

\begin{defn}[Judgemental theory] \label{mainjt}
A \textit{judgemental theory} $(\ctx, \classof{J}, \classof{R}, \classof{P})$ is a pre-judgemental theory such that
\begin{enumerate}
    \item $\classof{R}$ and $\classof{P}$ are closed under composition;
    \item the judgements are precisely those rules whose codomain is $\ctx$;
    \item $\classof{R}$ and $\classof{P}$ are closed under \textit{finite limits} (see \Cref{pull}, \Cref{equal} and \Cref{powers}), \textit{$\sharp$-liftings} (see \Cref{lifting}) and \textit{whiskering} (see \Cref{whiskering}).
\end{enumerate}
\end{defn}

The rest of this section is dedicated to clarifying the technical aspects of this definition. In the next section we will see that these properties influence the inference power of our logical systems. The more we put, the more we infer.

\begin{rem}
The condition (2) in \Cref{mainjt} is actually not needed, yet it is not harmful for the theory and it allows a cleaner axiomatization of (3), which otherwise would not look as pretty.
\end{rem}

\begin{rem}[Infinitary judgemental theories] \label{infinitary}
We could have allowed $\lambda$-small limits for $\lambda$ a (regular) cardinal, so that we are actually studying \textit{finitary} judgemental theories. In the present work we stick to this choice.
\end{rem}

\begin{rem}[Economical presentations of judgemental theories] \label{economical}
In the majority of concrete instances, a judgemental theory is presented by a pre-judgemental theory $(\ctx, \classof{J}, \classof{R}, \classof{P})$, in the sense that we close the data of judgements, rules and policies under finite limits and $\sharp$-liftings and whiskering. This produces the smallest judgemental theory containing $(\ctx, \classof{J}, \classof{R}, \classof{P})$.
\end{rem}

\begin{notat}\label{thenotation}
When a classifier $\mathbb{X}$ is obtained by iterated pullback of classifiers along rules, we try to use a notation that keeps in mind this special property of the classifier. Consider thus the diagram below.
\[\begin{tikzcd}
	{(\mathbb{H}.\lambda\mathbb{F})\gamma.\mathbb{V}} &&&& {\mathbb{V}} \\
	\\
	{\mathbb{H}.\lambda\mathbb{F}} && {\mathbb{F}\times\mathbb{G}} && {\mathbb{F}} \\
	\\
	{\mathbb{H}} && {\mathbb{G}} && {\ctx}
	\arrow["f"{description}, from=3-5, to=5-5]
	\arrow["g"{description}, from=5-3, to=5-5]
	\arrow["\lambda"{description}, color={rgb,255:red,214;green,92;blue,92}, from=5-1, to=5-3]
	\arrow[dashed, from=3-1, to=5-1]
	\arrow[from=3-3, to=5-3]
	\arrow[dashed, from=3-3, to=3-5]
	\arrow["\ulcorner"{anchor=center, pos=0.125}, draw=none, from=3-3, to=5-5]
	\arrow["\gamma"{description}, color={rgb,255:red,214;green,92;blue,92}, from=1-5, to=3-5]
	\arrow[dashed, from=1-1, to=3-1]
	\arrow[dashed, from=1-1, to=1-5]
	\arrow[curve={height=-30pt}, dashed, from=3-1, to=3-5]
	\arrow["\ulcorner"{anchor=center, pos=0.125}, draw=none, from=1-1, to=5-5]
	\arrow["\ulcorner"{anchor=center, pos=0.125}, draw=none, from=3-1, to=5-3]
\end{tikzcd}\]
\begin{itemize}
\item We use the notation $\mathbb{F} \times \mathbb{G}$ when we pullback classifiers along classifiers.
\item When we pullback a classifier along a rule, we use the notation $\mathbb{H}.\lambda\mathbb{F}$. We can make sense of this as if we put an additional bound on $\mathbb{F}$, and this is induced from $\mathbb{H}$ via $\lambda$. The reader will find more about this in \Cref{complexjcp}.
\item When we iterate this procedure, for example as in the diagram, we use the notation  $(\mathbb{H}.\lambda\mathbb{F})\gamma.\mathbb{V}$. When $g=f$ we write it $\mathbb{H}\gamma.\lambda\mathbb{V}$.
\end{itemize}
We are aware that this notation is not entirely economical, nor uniquely determined, but in the practical circumstances of this paper, it will be very useful.
\end{notat}

\begin{req}[Pullbacks] \label{pull}
$\classof{R}$ is closed under pullbacks in the sense that, given solid (black) spans and cones in $\classof{R}$ as below, we have that all the colored arrows belong to $\classof{R}$.

\[\begin{tikzcd}
	{\mathbb{X}} \\
	& {\mathbb{F}\classof{R}.\classof{R}\mathbb{H}} && {\mathbb{H}} \\
	\\
	& {\mathbb{F}} && {\mathbb{G}}
	\arrow["{\classof{R}}"{description}, from=4-2, to=4-4]
	\arrow["{\classof{R}}"{description}, from=2-4, to=4-4]
	\arrow["{\classof{R}}"{description}, color={rgb,255:red,214;green,92;blue,92}, dashed, from=2-2, to=4-2]
	\arrow["{\classof{R}}"{description}, color={rgb,255:red,214;green,92;blue,92}, dashed, from=2-2, to=2-4]
	\arrow["\ulcorner"{anchor=center, pos=0.125}, draw=none, from=2-2, to=4-4]
	\arrow["{\classof{R}}"{description}, curve={height=-12pt}, from=1-1, to=2-4]
	\arrow["{\classof{R}}"{description}, curve={height=12pt}, from=1-1, to=4-2]
	\arrow["{\classof{R}}"{description}, color={rgb,255:red,214;green,92;blue,92}, dotted, from=1-1, to=2-2]
\end{tikzcd}\]

In this definition we see the advantage of including $\classof{J}$ in $\classof{R}$, otherwise we would have to specify another axiom for the case in which the span is made of judgements. This could have been done without major differences, but would lead to an incredible proliferation of diagrams.
\end{req}

\begin{req}[Equalizers] \label{equal}
Similarly to the case of pullbacks, we require that the equalizer $\mathbb{E}$, together with its limiting maps, belongs to the rules.

\[\begin{tikzcd}
	{\mathbb{X}} \\
	{\mathbb{E}} & {\mathbb{F}} & {\mathbb{G}}
	\arrow["{\classof{R}}"{description}, shift left=2, from=2-2, to=2-3]
	\arrow["{\classof{R}}"{description}, shift right=2, from=2-2, to=2-3]
	\arrow["{\classof{R}}"', color={rgb,255:red,214;green,92;blue,92}, dashed, from=2-1, to=2-2]
	\arrow["{\classof{R}}"{description}, from=1-1, to=2-2]
	\arrow["{\classof{R}}"', color={rgb,255:red,214;green,92;blue,92}, dashed, from=1-1, to=2-1]
\end{tikzcd}\]
\end{req}

\begin{req}[Powers] \label{powers}
We also require that, for all rules $\mathbb{X} \to \mathbb{Y}$, we can form the finite powers below in $\classof{R}$ and that, as in \Cref{equal} and \Cref{pull}, all the arrows induced by their universal properties by cones made of rules, are rules too.
\[\begin{tikzcd}
	{\phantom{i}\mathbb{X}\due} && {\phantom{i}\mathbb{X}^n} \\
	{\mathbb{X}} && {\mathbb{X}}
	\arrow["{\dom}", shift left=1, from=1-1, to=2-1]
	\arrow["{\cod}"', shift right=1, from=1-1, to=2-1]
	\arrow["{\pi_1}"', shift right=4, from=1-3, to=2-3]
	\arrow["{\pi_n}", shift left=4, from=1-3, to=2-3]
	\arrow["{...}"{description}, from=1-3, to=2-3]
\end{tikzcd}\]
Regarding our meta-theory, this only requires that the finite product of sets (or classes, or $\kappa$-sets for some inaccessible cardinal $\kappa$, depending on the meta-theory of choice) is again a set (or class, or $\kappa$-set).
\end{req}

Finally, we complete the discussion of \Cref{mainjt} by explaining what we mean by closure under $\sharp$-liftings -- and what a $\sharp$-lifting is.

\begin{defn}[$\sharp$-lifting]\label{defnsharplift}
Consider a functor $f\colon\ctg{A}\to\ctg{B}$ and a 2-cell $\alpha\colon c' \Rightarrow c\colon \ctg{C} \to \ctg{B}$, and compute the pullback of $c$ and $c'$ along $f$. A \emph{sharp lifting} or \emph{$\sharp$-lifting} of $\alpha$ along $f$ is a pair $(\id\times_{\ctg{B}}f, \alpha\times_{\ctg{B}}f )$ of a functor and a natural transformation as below,
\[\begin{tikzcd}[ampersand replacement=\&]
	{c\times_{\ctg{B}}f} \&\& {\ctg{A}} \\
	\& {c'\times_{\ctg{B}}f} \\
	\\
	{\ctg{C}} \&\& {\ctg{B}} \\
	\& {\ctg{C}}
	\arrow["f", from=1-3, to=4-3]
	\arrow[""{name=0, anchor=center, inner sep=0}, "c", shift left=2, from=4-1, to=4-3]
	\arrow["{c'}"', from=5-2, to=4-3]
	\arrow["\id"'{pos=0.3}, curve={height=12pt}, from=4-1, to=5-2]
	\arrow["{f^*c'}"', from=2-2, to=1-3]
	\arrow[""{name=1, anchor=center, inner sep=0}, "{f^*c}", from=1-1, to=1-3]
	\arrow["{\id\times_{\ctg{B}}f}"'{pos=0.3}, curve={height=12pt}, dashed, from=1-1, to=2-2]
	\arrow["\alpha"{pos=0.4}, shorten >=4pt, Rightarrow, from=5-2, to=0]
	\arrow["{\alpha\times_{\ctg{B}}f}"{pos=0.4}, shorten >=3pt, Rightarrow, dashed, from=2-2, to=1]
\end{tikzcd}\]
so that all ``vertical'' squares commute.
\end{defn}
\begin{rem}
We could not find any precise instance of $\sharp$-liftings in the literature, our construction seems original. The most comparable results seems to be contained in  \cite{gray_fibered}, see for example Thm.\! 2.10 in loc.\! cit.
\end{rem}

This construction is clearly less known that the others appearing in the previous sections, but it will be of fundamental importance in using judgemental theories, as it is actually quite closely related to the process of computing substitution. We detail its technical features in \Cref{notions_subs} and its consequences for the logic in \Cref{onsubst}, but for the moment the reader only needs to know that whenever $f$ is a(n) (op)fibration, such a $\sharp$-lifting exists.

\begin{req}[$\sharp$-lifting] \label{lifting}

Consider a policy  $\lambda^\sharp$ as in the diagram below, and a rule $\classof{R}^*$ which is a fibration. Then, by \Cref{charsharplift}, there is a pair $(\classof{R}^*\lambda, \classof{R}^*\lambda^\sharp)$ as below. Closure for $\sharp$-liftings amounts to ask that $\classof{R}^*\lambda$ and $\classof{R}^*\lambda^\sharp$ belong to $\classof{R}$ and $\classof{P}$, respectively. Similarly, for $\classof{R}_*$ an opfibration, we get the op-diagram on the right. Notice that in both cases the square containing $\lambda$ and $\classof{R}^*\lambda$ or $\classof{R}_*\lambda$ commutes strictly.

\adjustbox{max width=\textwidth}{%
\begin{tikzcd}
	{\mathbb{F}\classof{R}^*.\classof{R}\mathbb{H}} &&& {\mathbb{H}} & {\mathbb{F}\classof{R}_*.\classof{R}\mathbb{H}} &&& {\mathbb{H}} \\
	& {\mathbb{G}\classof{R}^*.\classof{R}\mathbb{H}} &&&& {\mathbb{G}\classof{R}_*.\classof{R}\mathbb{H}} \\
	{\mathbb{F}} &&& {\mathbb{X}} & {\mathbb{F}} &&& {\mathbb{X}} \\
	& {\mathbb{G}} &&&& {\mathbb{G}}
	\arrow["{\classof{R}_*}", from=1-8, to=3-8]
	\arrow["{\classof{R}}"{description}, from=4-6, to=3-8]
	\arrow[from=2-6, to=4-6]
	\arrow["{\classof{R}}"{description}, from=2-6, to=1-8]
	\arrow[""{name=0, anchor=center, inner sep=0}, "{\classof{R}}"{description}, from=1-5, to=1-8]
	\arrow["{\classof{R}_*\lambda}"'{pos=0.6}, color={rgb,255:red,214;green,92;blue,92}, curve={height=6pt}, dashed, from=1-5, to=2-6]
	\arrow["\lambda"', curve={height=6pt}, from=3-5, to=4-6]
	\arrow[from=1-5, to=3-5]
	\arrow[""{name=1, anchor=center, inner sep=0}, "{\classof{R}}"{description}, from=3-5, to=3-8]
	\arrow["{\classof{R}}"{description}, from=4-2, to=3-4]
	\arrow[""{name=2, anchor=center, inner sep=0}, "{\classof{R}}"{description}, from=3-1, to=3-4]
	\arrow["\lambda"', curve={height=6pt}, from=3-1, to=4-2]
	\arrow["{\classof{R}}"{description}, from=2-2, to=1-4]
	\arrow[""{name=3, anchor=center, inner sep=0}, "{\classof{R}}"{description}, from=1-1, to=1-4]
	\arrow["{\classof{R}^*\lambda}"'{pos=0.6}, color={rgb,255:red,214;green,92;blue,92}, curve={height=6pt}, dashed, from=1-1, to=2-2]
	\arrow["{\classof{R}^*}", from=1-4, to=3-4]
	\arrow[from=2-2, to=4-2]
	\arrow[from=1-1, to=3-1]
	\arrow["\ulcorner"{anchor=center, pos=0.125}, draw=none, from=2-6, to=3-8]
	\arrow["\ulcorner"{anchor=center, pos=0.125}, draw=none, from=1-5, to=3-8]
	\arrow["\ulcorner"{anchor=center, pos=0.125}, draw=none, from=2-2, to=3-4]
	\arrow["\ulcorner"{anchor=center, pos=0.125}, draw=none, from=1-1, to=3-4]
	\arrow["{\lambda^\sharp}"'{pos=0.6}, shorten <=2pt, shorten >=4pt, Rightarrow, from=4-2, to=2]
	\arrow["{\classof{R}^*\lambda^\sharp}"{pos=0.2}, color={rgb,255:red,214;green,92;blue,92}, shorten <=2pt, shorten >=4pt, Rightarrow, dashed, from=2-2, to=3]
	\arrow["{\classof{R}_*\lambda_\sharp}"'{pos=0.8}, color={rgb,255:red,214;green,92;blue,92}, shorten <=4pt, shorten >=2pt, Rightarrow, dashed, from=0, to=2-6]
	\arrow["{\lambda_\sharp}"{pos=0.4}, shorten <=4pt, shorten >=2pt, Rightarrow, from=1, to=4-6]
\end{tikzcd}
}
\end{req}

\begin{req}[Whiskering] \label{whiskering}
As it is quite frequent in 2-category theory
, one might want to compose 1-cells with 2-cells. As our theory is quite heavily 2-dimensional, it only make sense that we ask that performing such an operation does not bring us out of our logic. We recall the general definition in the 2-category $\catcat$, as it is the one we are interested in now. Consider categories, functors, and natural transformations as below.
\[\begin{tikzcd}[ampersand replacement=\&]
	{\ctg{A}} \& {\ctg{B}} \& {\ctg{C}} \& {\ctg{D}}
	\arrow["F", from=1-1, to=1-2]
	\arrow["H", from=1-3, to=1-4]
	\arrow[""{name=0, anchor=center, inner sep=0}, "G", shift left=3, from=1-2, to=1-3]
	\arrow[""{name=1, anchor=center, inner sep=0}, "{G'}"', shift right=3, from=1-2, to=1-3]
	\arrow["\alpha", shorten <=2pt, shorten >=2pt, Rightarrow, from=0, to=1]
\end{tikzcd}\]
One can always define natural transformations $\alpha \ast F\colon GF\Rightarrow G'F$ and $H\ast \alpha\colon HG\Rightarrow HG'$ that point-wise act as
\[
(\alpha\ast F)_A = \alpha_{FA}\qquad (H\ast \alpha)_B = H(\alpha_B).
\]

Given classifiers $\mathbb{X,Y,Z}$ , rules $\lambda, \lambda', \gamma$ and a policy $\lambda^\sharp$, then, we say that the judgemental theory is closed under whiskering in the sense that the colored natural transformations are policies too.
\[\begin{tikzcd}
	{\mathbb{X}} & {\mathbb{Y}} & {\mathbb{Z}} && {\mathbb{X}} && {\mathbb{Z}} \\
	\\
	{\mathbb{Z}} & {\mathbb{X}} & {\mathbb{Y}} && {\mathbb{Z}} && {\mathbb{X}}
	\arrow[""{name=0, anchor=center, inner sep=0}, "\lambda", shift left=4, from=1-1, to=1-2]
	\arrow[""{name=1, anchor=center, inner sep=0}, "{\lambda'}"', shift right=4, from=1-1, to=1-2]
	\arrow["\gamma", from=1-2, to=1-3]
	\arrow[""{name=2, anchor=center, inner sep=0}, "\gamma\lambda", shift left=4, from=1-5, to=1-7]
	\arrow[""{name=3, anchor=center, inner sep=0}, "{\gamma\lambda'}"', shift right=4, from=1-5, to=1-7]
	\arrow[""{name=4, anchor=center, inner sep=0}, "\lambda", shift left=4, from=3-2, to=3-3]
	\arrow[""{name=5, anchor=center, inner sep=0}, "{\lambda'}"', shift right=4, from=3-2, to=3-3]
	\arrow["\gamma"', from=3-1, to=3-2]
	\arrow[""{name=6, anchor=center, inner sep=0}, "\lambda\gamma", shift left=4, from=3-5, to=3-7]
	\arrow[""{name=7, anchor=center, inner sep=0}, "{\lambda'\gamma}"', shift right=4, from=3-5, to=3-7]
	\arrow["{\gamma(\lambda^\sharp)}"', color={rgb,255:red,214;green,92;blue,92}, shorten <=2pt, shorten >=2pt, Rightarrow, from=2, to=3]
	\arrow["{\lambda^\sharp}", shorten <=2pt, shorten >=2pt, Rightarrow, from=0, to=1]
	\arrow["{\lambda^\sharp}", shorten <=2pt, shorten >=2pt, Rightarrow, from=4, to=5]
	\arrow["{\lambda^\sharp_{\gamma}}"', color={rgb,255:red,214;green,92;blue,92}, shorten <=2pt, shorten >=2pt, Rightarrow, from=6, to=7]
\end{tikzcd}\]
\end{req}

\begin{disclaim}
We understand that up to this point the reader has been faced with many concepts and strange notations that they have no intuition for. Therefore, before we formally describe what it means to define a calculus based on the blocks that are our judgemental theories, we advise the reader to skip to \Cref{dicjud} and see what it is that we are trying to achieve.
\end{disclaim}

\subsection{Notions of substitution}\label{notions_subs}

\begin{defn}\label{substitutional}
A judgement classifier is \emph{(op)substi\-tutional} if it is an (op)fibration. A rule is \emph{(op)cartesian} if it preserves (op)cartesian maps. A policy is \emph{(op)Frobenius} with respect to a given judgement classifier if it has cartesian components.

By extension, we will say that a (pre)judgemental theory is \emph{(op)substi\-tutional} if all judgement classifiers are (op)substitutional, all rules are (op)cartesian, and all policies are (op)Frobenius.
\end{defn}

\begin{thm}[Characterizing fibrations via $\sharp$-lifting]\label{charsharplift}
The following are equivalent for a functor $p \colon \ee \to \bb$:
\begin{enumerate}
\item $p$ is a fibration with a cleavage $S$;
\item each 2-cell $\alpha\colon c' \Rightarrow c\colon \ctg{C} \to \ctg{B}$ admits a terminal $\sharp$-lifting along $p$, meaning that provided another $(g,\beta)$ $\sharp$-lifting of $\alpha$ along $p$, we have a unique vertical $\overline{\beta}$ such that $\beta=\alpha\times_{\ctg{B}}p\ast\overline{\beta}$.

\adjustbox{width=.9\textwidth}{
\begin{tikzcd}[ampersand replacement=\&]
	{c\times_{\ctg{B}}p} \&\& {\ctg{E}} \&\& {c\times_{\ctg{B}}p} \&\& {\ctg{E}} \\
	\& {c'\times_{\ctg{B}}p} \&\& {=} \&\& {c'\times_{\ctg{B}}p}
	\arrow[from=2-6, to=1-7]
	\arrow[""{name=0, anchor=center, inner sep=0}, from=1-5, to=1-7]
	\arrow[""{name=1, anchor=center, inner sep=0}, from=1-5, to=2-6]
	\arrow[""{name=2, anchor=center, inner sep=0}, "g"', curve={height=18pt}, from=1-5, to=2-6]
	\arrow[""{name=3, anchor=center, inner sep=0}, from=1-1, to=1-3]
	\arrow[from=2-2, to=1-3]
	\arrow["g"', curve={height=18pt}, from=1-1, to=2-2]
	\arrow["{\alpha\times_{\ctg{B}}p}"{description, pos=0.4}, shorten >=3pt, Rightarrow, from=2-6, to=0]
	\arrow["\beta"', shorten >=3pt, Rightarrow, from=2-2, to=3]
	\arrow["{\overline{\beta}}"', shorten <=3pt, shorten >=3pt, Rightarrow, from=2, to=1]
\end{tikzcd}
}
\end{enumerate}
\end{thm}
\begin{proof}
If $p$ is a fibration with cleavage $S$, we can define the functor
\[
\id\times_\ctg{B}p \colon c\times_\ctg{B}p \to c'\times_\ctg{B}p,\quad (X,A)\mapsto (X,S(A,\alpha_X))
\]
reindexing $A$ along $\alpha_X\colon c'X\to cX=pA$. All desired squares commute. Moreover, we have a natural transformation $\alpha\times_\ctg{B}p\colon p^* c' \circ \id\times_\ctg{B}p \Rightarrow p^* c$ that on components is defined as follows
\[
(\alpha\times_\ctg{B}p)_{(X,A)} = s_{A,\alpha_X} \,.
\]
In particular, since each $s_{A,\alpha_X}$ is cartesian we have that the pair $(\id\times_\ctg{B}p,\alpha\times_\ctg{B}p)$ enjoys the desired universal property: the induced unique vertical arrows assemble into the necessary $\overline{\beta}$.

Conversely, let $p$ a functor and consider the trivial 2-cell $\dom\Rightarrow\cod$, then there exists a terminal sharp lifting as below,
\[\begin{tikzcd}[ampersand replacement=\&]
	{\cod\times_{\ctg{B}}p} \&\& {\ctg{E}} \\
	\& {\dom\times_{\ctg{B}}p} \\
	\\
	{\ctg{B}\due} \&\& {\ctg{B}} \\
	\& {\ctg{B}\due}
	\arrow["p", from=1-3, to=4-3]
	\arrow[""{name=0, anchor=center, inner sep=0}, "\cod", shift left=2, from=4-1, to=4-3]
	\arrow["\dom"', from=5-2, to=4-3]
	\arrow["\id"'{pos=0.3}, curve={height=12pt}, from=4-1, to=5-2]
	\arrow["{p^*\dom}"'{pos=0.4}, from=2-2, to=1-3]
	\arrow[""{name=1, anchor=center, inner sep=0}, "{p^*\cod}"{pos=0.4}, from=1-1, to=1-3]
	\arrow["{\id\times_{\ctg{B}}p}"'{pos=0.2}, curve={height=12pt}, dashed, from=1-1, to=2-2]
	\arrow["\alpha"{pos=0.4}, shorten >=4pt, Rightarrow, from=5-2, to=0]
	\arrow["{\alpha\times_{\ctg{B}}p}"{pos=0.4}, shorten >=3pt, Rightarrow, dashed, from=2-2, to=1]
\end{tikzcd}\]
therefore $\id\times_{\ctg{B}} p$ maps a pair $(A,\sigma\colon\Theta \to pA)$ to a pair $(B,\sigma\colon\Theta \to pA)$ with $pB=\Theta$, and there is a morphism
\[
(\alpha\times_{\ctg{B}}p)_{(A,\sigma)}\colon B\to A
\]
in $\ctg{E}$ over $\sigma$. We denote $(\id\times_{\ctg{B}} p) (A,\sigma)=S(A,\sigma)$ and $(\alpha\times_{\ctg{B}}p)_{(A,\sigma)}=s_{A,\sigma}$. It is cartesian because any other map over $\sigma\colon \Theta\to pA$ is part of another $\sharp$-lifting $(g,\beta)$ of $\alpha$ along $p$ and since $(\id\times_{\ctg{B}} p, \alpha\times_{\ctg{B}} p)$ is terminal with respect to this property, the unique induced $\overline{\beta}$ produces a suitable  unique vertical map into $S(A,\sigma)$.
\end{proof}

\begin{notat}[Substitution]
For the time being, and to avoid continuous explicit reference to a given cleavage for each fibration involved, we write $A[\sigma]$ for what was called $S(A,\sigma)$ up to this point. 
\end{notat}

See \Cref{onsubst} for what these imply for judgemental theories, for the moment we only prove a couple of technical results.

\begin{lem}[$\sharp$-lifting of cartesian functors]
Consider a fibration $h$ and a 2-cell $\lambda^\sharp$ as follows, and apply the construction in \Cref{lifting}.
\[\begin{tikzcd}[ampersand replacement=\&]
	{\mathbb{F}.\mathbb{H}} \&\&\& {\mathbb{H}} \\
	\& {\mathbb{G}.\mathbb{H}} \\
	{\mathbb{F}} \&\&\& \ctx \\
	\& {\mathbb{G}}
	\arrow["g"', from=4-2, to=3-4]
	\arrow[""{name=0, anchor=center, inner sep=0}, "f", from=3-1, to=3-4]
	\arrow["\lambda"', curve={height=6pt}, from=3-1, to=4-2]
	\arrow["{g.h}"', from=2-2, to=1-4]
	\arrow[""{name=1, anchor=center, inner sep=0}, "{f.h}", from=1-1, to=1-4]
	\arrow["{h^*\lambda}"'{pos=0.6}, curve={height=6pt}, dashed, from=1-1, to=2-2]
	\arrow["h", from=1-4, to=3-4]
	\arrow["{\lambda^\sharp}"'{pos=0.6}, shorten <=2pt, shorten >=4pt, Rightarrow, from=4-2, to=0]
	\arrow["{h^*\lambda^\sharp}"{pos=0.2}, shorten <=2pt, shorten >=4pt, Rightarrow, dashed, from=2-2, to=1]
\end{tikzcd}\]
If $\lambda$ preserves cartesian maps, then so does $h^*\lambda$.
\end{lem}
\begin{proof}
Consider a morphism $a=(a_1,a_2)\colon (F',H')\to(F,H)$ in $\mathbb{F}.\mathbb{H}$, meaning a pair $a_1\colon F'\to F$ in $\mathbb{F}$ and $a_2\colon H'\to H$ in $\mathbb{H}$ such that $f(a_1)=\sigma=h(a_2)$. One can check (see, for example, \cite[Proposition 2.6]{comprehensioncats}) that this is cartesian with respect to $h\circ (f.h)$ if and only if both $a_1$ is $f$-cartesian and $a_2$ is $h$-cartesian. The latter is equivalent to saying that $a_2$ is of the form $a_2=\overline{\sigma}\colon H[\sigma]\to H$. Now consider that the functor $h^*\lambda$ acts as follows
\[
(a_1,a_2)\colon (F',H')\to(F,H) \quad\mapsto\quad (\lambda a_1,a_2^!)\colon (\lambda F',H'[\lambda^\sharp_{F'}])\to (\lambda F,H[\lambda^\sharp_{F}])
\]
with $a_2^!$ the unique map induced by naturality of $\lambda^\sharp$ at $h(a_2)$. Assume that $a$ is cartesian, then we end up having
\[\begin{tikzcd}[ampersand replacement=\&]
	{H[\sigma][\lambda^\sharp_{F'}]} \&\& {H[\sigma]} \\
	\& {H[\lambda^\sharp_F]} \&\& H \\
	{\Theta'} \&\& \Theta \\
	\& {\Gamma'} \&\& \Gamma
	\arrow["{\overline{\sigma}}", from=1-3, to=2-4]
	\arrow["\sigma", from=3-3, to=4-4]
	\arrow["{\lambda^\sharp_F}"', from=4-2, to=4-4]
	\arrow["{\lambda^\sharp_{F'}}", from=3-1, to=3-3]
	\arrow["{\sigma'}"', from=3-1, to=4-2]
	\arrow["{\overline{\lambda^\sharp_{F'}}}", from=1-1, to=1-3]
	\arrow["{\overline{\lambda^\sharp_F}}"', from=2-2, to=2-4]
	\arrow[dashed, from=1-1, to=2-2]
\end{tikzcd}\]
therefore $a_2^!=\overline{\sigma}^!$ is itself cartesian. Hence if $\lambda$ preserves cartesian maps, then so does $h^*\lambda$.
\end{proof}

\begin{rem}[$\sharp$-lifting is cartesian]
The natural transformation $h^*\lambda^\sharp$ has $h$-cartesian components.
\end{rem}
\begin{proof}
This is actually trivial by definition of $h^*\lambda^\sharp$: in fact, it acts as
\[
(h^*\lambda^\sharp)_{(F,H)}=\overline{\lambda^\sharp_F}\colon H[\lambda^\sharp]\to H \,.
\]
\end{proof}

\section{Judgement calculi} \label{calculi}

In the previous section we have introduced judgemental theories, very concrete mathematical objects for which we have presented a suggestive notation referencing some logical intuition. This section is devoted to grounding that intuition and showing that each (pre)judgemental theory is a categorical version of a proof assistant or, more technically, something that supports the categorical semantics for the specification of a type system. We will see how a judgemental theory automatically produces a deductive system via a process of translation. Actually, a judgemental theory is intrinsically a calculus of deduction in a very precise sense.


This section will describe a way to translate the data of a judgemental theory $(\ctx, \classof{J,R,C})$ into a judgement calculus.

 \begin{minipage}{0.39\textwidth}
\[\begin{tikzcd}
	{\mathbb{U}.\Delta{\mathbb{U}}} && {\mathbb{U}} \\
	& {\ctx}
	\arrow["\Pi"{description}, color={rgb,255:red,167;green,42;blue,42}, from=1-1, to=1-3]
	\arrow[from=1-1, to=2-2]
	\arrow[from=1-3, to=2-2]
\end{tikzcd}\]
 \end{minipage}
 \begin{minipage}{0.60\textwidth}
     \begin{prooftree}\AxiomC{$\Gamma \vdash A\; {\tt Type}$}\AxiomC{$\Gamma.A \vdash B\; {\tt Type}$}\LeftLabel{($\Pi$F)}\BinaryInfC{$\Gamma \vdash \Pi_A B \; {\tt Type}$}\end{prooftree}
 \end{minipage}

Of course, such a process of translation requires an almost formal definition of judgement calculus, which must be flexible enough to encode the usual calculi that are used in type theory and in proof theory. For the reasons expressed in the introduction, this is a non-trivial task.

\subsection{Prolegomena} \label{prolegomena}
As we have hinted in the introduction, a very general definition of deductive system or \textit{calculus} is much easier to describe than to actually define.
Of course, one can make reference to \cite{restall2002introduction}, or to \cite{martin1984intuitionistic}, or to \cite{harper1993framework}, or to several variations of this notion, but there is no unified take we find satisfying. In this subsection we go through a critical analysis of the deductive systems of a dependent type theory and of a proof theory to better motivate the choices of the next subsection.

\subsubsection{The deductive system of a DTT} \label{calculusDTT}
We consider the problem of defining the semantics of the underlying signature, judgements and rules defining a formal calculus of a dependent type theory based on Martin-Löf's type theory. There are indeed several approaches in the literature, and the very notion of type theory is somehow (intentionally) fuzzy. We would go as far as to say that a complete agreement on the matter does not exist. Of course, this flexibility is part of the richness of this theory. The informality in the definition of rule and type constructor is one of the reasons for which the topic of (categorical) semantics for dependent type theory is both so popular and so useful in the theoretical research on dependent type theory.
Most sources would probably agree that to declare (the calculus of) a dependent type theory means to specify three boxes of data.
\begin{itemize}
    \item[(S)] Syntax (contexts, types, terms): a theory of dependent types is -informally- a formal system dealing with {\it types} and {\it terms} in {\it context}. From a symbolic point of view, these are a bunch of glyphs that we use as atoms of our language.
    $$\Gamma \quad \quad \quad \quad  A \quad \quad \quad \quad  a:A$$
    \item[(J)] Judgements (about contexts, types, terms): a judgement is a very simple sentence made up of symbols from the syntax, and whose intention is to somehow bound together pieces of atomic data. The most simple type theories of the sort we refer to present three possible (kinds of) judgements,
    $$\vdash\Gamma\,{\tt ctx} \quad\quad \Gamma \vdash A\; {\tt Type} \quad\quad \Gamma \vdash a:A$$
    which are informally interpreted as \textit{$\Gamma$ is a context}, \textit{$A$ is a type in context $\Gamma$}, \textit{$a$ is a term of type $A$ in context $\Gamma$}.
    \item[(R)] Rules (to declare new types, terms, and interact with the syntax): Finally, we should be able to interact with and declare a type. For those that are acquainted with a programming language, this need is completely evident. Indeed we might want to introduce a type which is constructed from other types.
    \vspace{-0.3cm}

    \begin{center}
    \begin{minipage}{0.49\textwidth}
    \begin{prooftree}\AxiomC{$\Gamma\vdash a:A$}\AxiomC{$\Gamma.A\vdash B\; {\tt Type}$}\LeftLabel{(DTy)}\BinaryInfC{$\Gamma\vdash B \la a \ra \; {\tt Type}$}\end{prooftree}
    \end{minipage}
     \end{center}

    Depending on the complexity of the theory, beyond a bunch of basic rules (like type and term dependency (\Cref{cetd})), we find type constructors. Type constructors are packages of rules, labelled by their feature, that allow to construct new types from old ones. Below we list the inescapable labels, for a constructor whose name is - say - $\Phi$.
        \begin{itemize}
        \item[$\Phi$\textsc{F}] Some {\it formation rule(s)}, presenting the type. They specify under which circumstances we can assume it to exist (or, from the point of view of programming languages, we can form it). By circumstances, we usually refer to syntactic data.
        \item[$\Phi$\textsc{I}] Some {\it introduction rule(s)}, producing the canonic terms of a such type. Given a set of syntactic data, they tell how to cook up a term of the new type.
        \item[$\Phi$\textsc{E}] Some {\it elimination rule(s)}, specifying the interaction between a term of the new type and the terms of the types that contributed to the formation of the new type.
        \end{itemize}
	Additionally, based on the computational semantics associated to the theory one wishes to consider, one needs to describe how introduction and elimination interact with one another, meaning to provide suitable {\it conversion rules}.

    These three packages of data reflect the necessities of a type theory: indeed, type theory emerged as a foundational framework, but from a cultural point of view its history is intertwined with that of programming languages. This deep interaction has shaped several aspects of type theory, we will see this especially in the declarative and interactive nature of the Rules box. Let us stress on the fact that, besides these informal distinctions, there is no formal definition of a type constructor, nor of a rule.
\end{itemize}

While the three boxes of Syntax, Judgements, and Rules are definitely there in any type theory, the list of rules, judgments, and the sort of symbols that inhabit them is subject to major choices. Even those that we have listed can be seen as somewhat arbitrary. Still, we believe that in any reasonable type theory the data above will be included. In most of the concrete instances, type theories are even richer than what we have listed above:

\begin{itemize}
    \item  \textit{morphisms of contexts} are usually added to Syntax;
    \item  \textit{definitional equality} is usually added to Judgements;
    \item  \textit{$\beta$- and $\eta$-computation} are almost always added to Rules,  and through definitional equality they determine how introduction and elimination interact with one another. Also, we will see to that our type theory has \textit{context formation}, which stands for a set of rules that form fresh contexts from existing types. Finally, if the syntax is enriched with morphisms of contexts, there might be rules regulating their interaction with judgements (that would be \textit{substitution}).
\end{itemize}

A vast majority of computer scientists and type theorists would probably classify $\beta$- and $\eta$-computation as an essential feature of a type theory.

\subsubsection{Natural deduction} \label{calculusproof}
As it was said in the introduction, natural deduction has already been shown to be fittingly translatable in the language of types \cite{Martin-Lof1996-MAROTM-7}, but we here describe its interpretation separately for multiple reasons:
\begin{enumerate}
\item on one hand, natural deduction for first-order logic has had a greater fortune in being studied and \emph{employed} in our schools and universities, and it is the one that we believe is understood best among most people, therefore
\item we believe that its \quotmarks{familiarity} makes it easier for the reader to connect the categorical syntax for the intuition of what the rules should be, moreover
\item such a familiarity allows us the freedom to describe more rules, and the practice of such an encoding is the main aim of this paper.
\end{enumerate}
The following presentation is mostly inspired by \cite[Def 2.18]{restall2002introduction}, but it is coherent with the treatment of \cite{negri2008structural} and \cite{troelstra2000basic} too.

When specifying a natural deduction calculus we provide three boxes of data:
\begin{itemize}
    \item[(S)] Syntax (variables, formulae): a natural deduction calculus is -informally- a formal system dealing with \textit{variables}, \textit{lists of formulae} and \textit{formulae}. From a symbolic point of view, these are glyphs that will be the atom of our calculus,
    $$x \quad \quad \Gamma \quad \quad  \phi.$$
    Often punctuation symbols as the semicolon \textit{;} are used too to combine the symbols.
    \item[(S)] Sequents: a sequent is a very simple sentence made up of symbols from the syntax, and whose intention is to specify an entailment relation between the data on the left and the data on the right of the entailment symbol.
    $$x; \Gamma \vdash \phi.$$
    For example, the sequent above could be read \textit{the list of formulae in $\Gamma$ entails the formula $\phi$, and they all have (at most) free variables in $x$.}
    \item[(R)] Rules: in natural deduction rules are used to state atomic consequences, they transform a family of sequents into a (family of) sequent(s).
    \begin{center}
              \begin{prooftree}\AxiomC{$x;\Gamma \vdash \phi$}\AxiomC{$x;\Gamma,\phi \vdash \psi$}\LeftLabel{(Cut)}\BinaryInfC{$x;\Gamma \vdash \psi$}\end{prooftree}
    \end{center}
    Traditionally, there is a distinction between \textit{structural rules} \cite[2.23]{restall2002introduction} and \textit{other rules}. Referring to \cite[pag. 26]{restall2002introduction}, the structural rules influence what we can prove. The more structural rules you have, the more you will be able to prove. The other rules are more in the spirit of type constructors and they account for the behavior of the logical operators, like $\wedge, \vee, \forall, \exists$. In modal logics, they can account for the behavior of modal operators too.
\end{itemize}
It is pretty intuitive that we can treat a sequent as a form of judgement. This just amounts to a re-tuning of our intuition with respect to the way we are used to read sequents. On the other hand it is not entirely trivial to find a precise correspondence between the rules of proof theory and the constructors of type theory. For example, for the reason that there is not a precise definition of constructor, nor a classification of them.

\subsubsection{Judgement calculi} \label{calcus}
Given the discussion above, our challenge is pretty clear: how to accommodate extensional type constructors, connectives, and deduction rules in a conceptually unified and technically coherent framework? Provide that we can see sequents as judgements, how do we formally deal with their manipulation from a semantic point of view? Let us dive into the definitions. For us, to declare a judgement calculus means to specify three boxes of data:
\begin{itemize}
    \item[(S)] Syntax (contexts and objects);
    \item[(J)] Judgements (acts of knowledge \textit{bound} to a context and \textit{pertaining to} a (list of) object(s));
    \item[(R)] Rules (transforming judgements into other judgements).
\end{itemize}

\subsection{Syntax} \label{syn}
Let $(\ctx, \classof{J,R,C})$ be a (pre)judgmental theory. Then its corresponding judgement calculus has in its Syntax box letters for each context and each object in a judgement classifier.
\begin{center}
\begin{tabular}{ c c c }
\hline
 $\Gamma$ && $\Gamma\in\ctx$\\
 $F$ && $F\in \mathbb{F}$ \\
\hline
\end{tabular}
\end{center}
 \subsection{Judgements}\label{judgement}
Judgemental calculi include two main kinds of judgement for each $\mathbb{F} \in \classof{J}$.
\begin{itemize}
    \item The first kind of judgement acknowledges  a $\mathbb{F}$-empirical evidence and clarifies the status of an object. One can see it as a kind of \textit{Tarskian snow} for our approach, meaning something that fulfills Tarski's requirement for something to \quotmarks{characterize unambiguously the class of those words and expressions which are to be considered meaningful} \cite{tarski1956concept}. For $F \in f^{-1}(\Gamma)$, then we find in our set of judgement the writing \[\Gamma \vdash F \; \mathbb{F}.\]
    This can be understood as \textit{Given $\Gamma$, $F$ exists} or \textit{Given $\Gamma$, $G$ is green}, or \textit{Given $\Gamma$, $M$ is made of marble}. In the case of the same category $\mathbb{F}$ appearing as the domain of two different judgements, we might use $\Gamma \vdash F\; \mathbb{F}(f)$.
    \item The second kind of judgement is an equality checker for the equality classified by the judgement. We will write  \[\Gamma \vdash F =_{\mathbb{F}} F',\] when $F, F' \in f^{-1}(\Gamma)$ and $F = F'$. This could be read as \textit{Given $\Gamma$, $F$ and $F'$ are indistinguishable by existence}\footnote{This could be actually read \textit{Are identical.}}, or  \textit{Given $\Gamma$, $G$ and $G'$ are indistinguishable by green.} Notice that the interpretation of the notion of equality is relative to the choice of the classifier. If we were to look at something classifying types in a type theory, the equality would (and will, in \Cref{secdtt}) be indistinguishability up to computations.
\end{itemize}

In the table below, we find on the left column the judgement and on the right its translation in terms of the judgemental theory.
\begin{center}
\begin{tabular}{ c c c }
\hline
 $\Gamma_1 \vdash F_1 \; \mathbb{F} \quad ... \quad \Gamma_n \vdash F_n\; \mathbb{F}$ && $(F_1 ... F_n) \in f^{-1}(\Gamma_1)\times ... \times f^{-1}(\Gamma_n)$ \\
  $\Gamma \vdash F =_{\mathbb{F}} F'$ && $F, F' \in f^{-1}(\Gamma)$ and $F = F'$ \\
\hline
\end{tabular}
\end{center}

It might seem that such simple judgements do not guarantee much in terms of expressiveness. This is in fact far from the truth! Recall that a judgemental theory is closed under finite limits and several constructions, thus we obtain an incredible variety of complex judgements.

\begin{rem}[On notions of equality and the relationship between theory and meta-theory]
Notice that all choices made here are to consider as \quotmarks{external}, in the sense that they constitute the building blocks of our calculus. They do not prevent from having, say, an identity judgement \textit{in} a judgemental theory, see for example \Cref{idty}.
\end{rem}

\begin{disc}[Nested judgements: pullbacks]\label{complexjcp}
Let $(\ctx, \classof{J, R, C})$ be a judgemental theory and consider a judgement of the form, \[\Gamma \vdash H.\lambda F \quad \mathbb{H}.\lambda\mathbb{F}.\]
We will see that such a judgement classifies a nested family of judgements, depending on the data of $\mathbb{H,F}$ and $\lambda$.
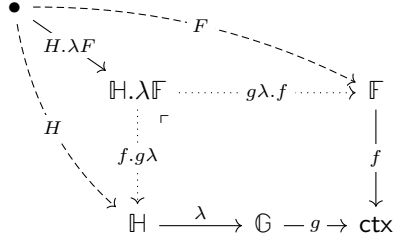
\[\begin{tikzcd}[ampersand replacement=\&]
	\bullet \\
	\& {\mathbb{H}.\lambda\mathbb{F}} \&\& {\mathbb{F}} \\
	\\
	\& {\mathbb{H}} \& {\mathbb{G}} \& {\ctx}
	\arrow["f"{description}, from=2-4, to=4-4]
	\arrow["g"{description}, from=4-3, to=4-4]
	\arrow["\lambda", from=4-2, to=4-3]
	\arrow["{f.g\lambda}"{description}, dotted, from=2-2, to=4-2]
	\arrow["\ulcorner"{anchor=center, pos=0.125}, draw=none, from=2-2, to=4-3]
	\arrow["{g\lambda.f}"{description}, dotted, from=2-2, to=2-4]
	\arrow["{H.\lambda F}"{description}, from=1-1, to=2-2]
	\arrow["H"{description}, curve={height=12pt}, dashed, from=1-1, to=4-2]
	\arrow["F"{description}, curve={height=-12pt}, dashed, from=1-1, to=2-4]
\end{tikzcd}\]
\noindent By inspecting the pullback diagram that defines $\mathbb{H}.\lambda\mathbb{F}$ we can see that judgements of this form are in bijection with pairs of judgements of the form, \[\Gamma \vdash H \; \mathbb{H} \qquad g\lambda H \vdash F\; \mathbb{F}.\]
This means that we are entitled to see the line above, which is composed of two related but separate judgements (in possibly different contexts!), as a single one (in context $\Gamma$). We call judgements of this form {\it nested}.
Notice that, depending on what we want to express, we could say that the relation binding $H$ to $\Gamma$ (hence the judgement classifier with domain $\mathbb{H}$ in the line above) is either $g\lambda$ itself, therefore forcing both $H$ and $F$ to have the same context, or some other $h$. Still, we chose to present the pullback in its most general form.

As a string of symbols, notice that a nested judgement is an informal judgement, in the sense that it is not well defined in our framework. Despite this, we will feel free to use notations as that above because they are a bit easier to parse from a human perspective.
This means that for the rest of the paper we will write

\begin{prooftree}\AxiomC{$\Gamma \vdash H.\lambda F\; \mathbb{H}.\lambda\mathbb{F}$}\doubleLine\UnaryInfC{$\Gamma \vdash H \; \mathbb{H} \qquad g\lambda H \vdash F\; \mathbb{F}$}\end{prooftree}

\noindent to intend that the synthetic judgment below is an \textit{alias} for the nested judgement above, which is defined in our context.

Our notation ($\mathbb{H}.\lambda\mathbb{F}$)  retains almost all the information needed to predict the kind of nested judgements we classify. This also explains why we write composed judgements the way we do: we think of the component in $\mathbb{H}$ to be free, while that in $\mathbb{F}$ is bounded via the map $\lambda$.
\end{disc}

\begin{exa}[Lists]
Let $\mathbb{X}$ be a judgement classifier, and consider $\mathbb{X}^n$, which is given by the (wide) pullback below,

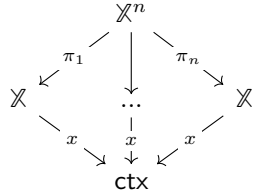
\[\begin{tikzcd}
	& {\mathbb{X}^n} \\
	{\mathbb{X}} & {...} & {\mathbb{X}} \\
	& {\ctx}
	\arrow["x"{description}, from=2-1, to=3-2]
	\arrow["x"{description}, from=2-2, to=3-2]
	\arrow["x"{description}, from=2-3, to=3-2]
	\arrow["{\pi_n}"{description}, from=1-2, to=2-3]
	\arrow["{\pi_1}"{description}, from=1-2, to=2-1]
	\arrow[from=1-2, to=2-2]
\end{tikzcd}\]
then the judgement $\Gamma \vdash X_1.\cdots.X_n \; \mathbb{X}^n$ can be interpreted as a list of judgements, as described below.

\begin{prooftree}\AxiomC{$\Gamma \vdash X_1.\cdots.X_n \; \mathbb{X}^n$}\doubleLine\UnaryInfC{$\Gamma \vdash X_1 \; \mathbb{X} \qquad \cdots \qquad \Gamma \vdash X_n \; \mathbb{X}$}\end{prooftree}

\noindent Notice that the fact that a judgemental theory is by definition closed under finite products implies that these judgements are always available.
\end{exa}

\begin{exa}[Composable arrows]
Let $\mathbb{C}$ be a category and consider the following pullback as on the left. The resulting nested judgement, then, reads as on the right and classifies composable arrows in $\mathbb{C}$.

\begin{minipage}{0.49\textwidth}
\[\begin{tikzcd}
	{\mathbb{C}\due\dom.\cod\mathbb{C}\due} & {\mathbb{C}\due} \\
	{\mathbb{C}\due} & {\mathbb{C}}
	\arrow["{\dom}", from=1-2, to=2-2]
	\arrow["{\cod}"', from=2-1, to=2-2]
	\arrow[dashed, from=1-1, to=2-1]
	\arrow[dashed, from=1-1, to=1-2]
	\arrow["\ulcorner"{anchor=center, pos=0.125}, draw=none, from=1-1, to=2-2]
\end{tikzcd}\]
\end{minipage}
\begin{minipage}{0.49\textwidth}
\begin{prooftree}\AxiomC{$C \vdash (f,g)\;\mathbb{C}\due\dom.\cod\mathbb{C}\due$}\doubleLine\UnaryInfC{$C \vdash f \; \mathbb{C}(\cod)\; \qquad C \vdash g \; \mathbb{C}(\dom)$}\end{prooftree}
\end{minipage}
The middle ground is given by the actual {\it middle} object $f$ and $g$ share. Notice that, though the context (i.e.\! the object) is the same, their bounds to it are very different (that is, respectively $\cod$ and $\dom$).
\end{exa}

\begin{exa}[Toy Martin-Löf type theory]\label{toytt.j}
In the judgemental theory generated by that in \Cref{toytt}, we now have a way of coding, for example, pairs of types in the same context. This is achieved by the pullback $\uu u.u \uu$.

\begin{minipage}{0.49\textwidth}
\[\begin{tikzcd}
	{\uu u.u \uu} & {\uu} \\
	{\uu} & {\ctx}
	\arrow["{u}", from=1-2, to=2-2]
	\arrow["{u}"', from=2-1, to=2-2]
	\arrow[dashed, from=1-1, to=2-1]
	\arrow[dashed, from=1-1, to=1-2]
	\arrow["\ulcorner"{anchor=center, pos=0.125}, draw=none, from=1-1, to=2-2]
\end{tikzcd}\]
\end{minipage}
\begin{minipage}{0.49\textwidth}
\begin{prooftree}\AxiomC{$\Gamma \vdash (A,A')\;\uu u.u \uu$}\doubleLine\UnaryInfC{$\Gamma \vdash A \; \uu\; \qquad \Gamma \vdash A' \; \uu$}\end{prooftree}
\end{minipage}
\end{exa}

The examples above are not particularly interesting, though we believe they give an intuition of the expressive power of nested judgements. We hope \Cref{secdtt} will be definitive proof.

\begin{disc}[Nested judgements: equalizers]\label{complexjce}
Similarly to the previous case, equalizers classify nested judgements of the kind below.

\begin{minipage}{0.39\textwidth}
\[\begin{tikzcd}
	{\mathbb{E}(\lambda,\lambda')} & {\mathbb{F}} & {\mathbb{G}}
	\arrow["\lambda", shift left=1, from=1-2, to=1-3]
	\arrow["{\lambda'}"', shift right=1, from=1-2, to=1-3]
	\arrow["e", dashed, from=1-1, to=1-2]
\end{tikzcd}\]
\end{minipage}
\begin{minipage}{0.59\textwidth}
     \begin{prooftree}\AxiomC{$\Gamma\; \vdash F\; \mathbb{E}(\lambda,\lambda')$}\doubleLine\UnaryInfC{$\Gamma\vdash F\; \mathbb{F} \qquad \Gamma \vdash \lambda F =_\mathbb{G} \lambda' F $}\end{prooftree}
\end{minipage}
\end{disc}

\subsection{Rules} \label{rulesss}

Let $(\ctx, \classof{J}, \classof{R}, \classof{P})$ be a judgemental theory. Consider judgements and a rule as in the diagram below, for each rule $\lambda : \mathbb{F} \to \mathbb{G}$ and each judgement $\Gamma \vdash F \; \mathbb{F}$, we will write as follows (on the right).

\begin{center}
   \begin{minipage}{0.39\textwidth}
\[\begin{tikzcd}
	{\mathbb{F}} && {\mathbb{G}} \\
	\\
	{\ctx} && {\ctx}
	\arrow["\lambda", from=1-1, to=1-3]
	\arrow["f"', from=1-1, to=3-1]
	\arrow["g", from=1-3, to=3-3]
\end{tikzcd}\]
 \end{minipage}
 \begin{minipage}{0.3\textwidth}
            \begin{prooftree}\AxiomC{$\Gamma\; \vdash F\; \mathbb{F}$}\LeftLabel{$(\lambda)$}\UnaryInfC{$g\lambda F \vdash \lambda F \; \mathbb{G}$}\end{prooftree}\label{rule}
 \end{minipage}\hfill
\end{center}

From a technical level, this is just a compact way to organize the data of the functoriality of $\lambda$. Indeed it is true that $\lambda F \in g^{-1}(g\lambda F)$, so that $g\lambda F \vdash \lambda F \; \mathbb{G}$ is actually a judgement in our framework. This is the only kind of rule that we admit in our judgemental calculi, and in a sense all the rules are the same, there are no intrinsic labels like \textit{structural, introduction, elimination}, and so on. Yet, similarly to the case of judgements, the closure under finite limits guarantees an incredible richness of rules, as we will see for the rest of the subsection.

\begin{exa}[Toy Martin-Löf type theory]\label{toytt.r}
The rule $\Sigma$ from \Cref{toytt} now reads as follows.

\begin{prooftree}\AxiomC{$\Gamma\vdash (a,A)\;\duu$}\LeftLabel{($\Sigma$)}\UnaryInfC{$u\Sigma(a,A)\vdash \Sigma (a,A)\;\uu$}\end{prooftree}
\vspace{.01cm}

Now recall that $u\Sigma=\du$, so the behavior of a policy is implied: see \Cref{secdtt} for more on this. We follow the intuition provided for all the data of the judgemental theory in \Cref{toytt} and translate it in the usual type-theoretic strings of symbols. Then it reads as follows

\begin{prooftree}\AxiomC{$\Gamma\vdash a:A$}\LeftLabel{($\Sigma$)}\UnaryInfC{$\Gamma\vdash A\; \texttt{Type}$}\end{prooftree}
\vspace{.01cm}

\noindent and depicts the typing rule. We thoroughly detail this process of translation in \Cref{dictionnaire}.
\end{exa}

\begin{rem}[Rules with many outputs] \label{tonested}
The notion of nested judgement \cref{complexjcp} and of our calculus as a whole have one additional very useful feature, and that is allowing for multiple consequents {\it simultaneously}. In fact, it is very common that one might want to write rules that deduce several judgements from the same (set of) judgement(s), but writing it organically is somewhat frowned upon, so that one usually encounters a proliferation of rules (for example two elimination rules in \Cref{idty} and in \Cref{connrules}). While we mostly follow the tradition with regard to this, the attentive reader will see that in fact they are always the product of the ``break-down'' of a single nested judgement. We make this explicit once in \Cref{triangid}.
\end{rem}

\begin{rem}[Soundness and completeness]
When one looks at this section as a whole, that is the process of producing a graphical/grammatical bookkeeping of the categorical properties of the judgemental theory, organized in the form of a collection of judgements and deductions,  it is natural to raise the question \textit{which deductions are actually produced by a judgemental theory?} There are two possible approaches to this question.

\begin{itemize}
    \item The first approach is to refine our notion of calculus, and give a precise definition of what we mean by deductive system. Our presentation is not that far from a formalization. Then, one would say that the content of this section provides a kind of soundness/correctness theorem for the categorical syntax, and one could try to provide a completeness result that characterize all the possible judgements and deductions.
    \item The second approach is to claim that the question contains an implicit bias towards grammatical/syntactic representations of deductions, and that, in a sense, the categorical language already provides the grammar the reader is looking for, while the calculus in this section only represents a way to make it more digestible to the \textit{grammarian}.
\end{itemize}
Both the approaches are valid, one maybe making a more political statement, and the other being more prone to the classical tradition. Because the author herself does not entirely agree on the path to follow, and because this paper already contains a lot of material, we choose not to invest more on this question in the present work. We will be greatly interested in developing this more along the line.
\end{rem}

\subsection{Policies}\label{syntax_policies}

Recall that a policy is a 2-cell as follows
\[\begin{tikzcd}[ampersand replacement=\&]
	{\mathbb{F}} \&\& {\mathbb{G}} \\
	\& {\ctx}
	\arrow[""{name=0, anchor=center, inner sep=0}, "f"', curve={height=6pt}, from=1-1, to=2-2]
	\arrow["g", curve={height=-6pt}, from=1-3, to=2-2]
	\arrow["\lambda", from=1-1, to=1-3]
	\arrow["{\lambda^\sharp}"{description}, shorten >=9pt, Rightarrow, from=1-3, to=0]
\end{tikzcd}\]
with $f,g$ judgement classifiers and $\lambda$ a rule. This additional data contains that of the rule $\lambda$, hence
\vspace{-0.5cm}
 \begin{prooftree}\AxiomC{$\Gamma \vdash F\; \mathbb{F}$}\LeftLabel{$(\lambda)$}\UnaryInfC{$\Theta \vdash \lambda F \; \mathbb{G}$}\end{prooftree}
but it also establishes a relation between $\Gamma=fF$ and $\Theta=g\lambda F$, namely $\lambda^\sharp_F:\Theta\to \Gamma$. We wish our judgemental calculus to reflect this.

If $f$ is a fibration by \Cref{lifting} we can $\sharp$-lift $\lambda^\sharp$ along $f$,
\[\begin{tikzcd}[ampersand replacement=\&]
	{\mathbb{F}.\mathbb{F}} \&\&\& {\mathbb{F}} \\
	\& {\mathbb{G}.\mathbb{F}} \&\&\&\& {\mathbb{F}} \&\& {\mathbb{F}} \\
	{\mathbb{F}} \&\&\& \ctx \&\&\& \ctx \\
	\& {\mathbb{G}}
	\arrow["g"', from=4-2, to=3-4]
	\arrow[""{name=0, anchor=center, inner sep=0}, "f", from=3-1, to=3-4]
	\arrow["\lambda"', curve={height=6pt}, from=3-1, to=4-2]
	\arrow["{g.f}"', from=2-2, to=1-4]
	\arrow[""{name=1, anchor=center, inner sep=0}, "{f.f}", from=1-1, to=1-4]
	\arrow["{f^*\lambda}"'{pos=0.6}, color={rgb,255:red,214;green,92;blue,92}, curve={height=6pt}, dashed, from=1-1, to=2-2]
	\arrow["f", from=1-4, to=3-4]
	\arrow[from=2-2, to=4-2]
	\arrow[from=1-1, to=3-1]
	\arrow["\ulcorner"{anchor=center, pos=0.125}, draw=none, from=2-2, to=3-4]
	\arrow["\ulcorner"{anchor=center, pos=0.125}, draw=none, from=1-1, to=3-4]
	\arrow[""{name=2, anchor=center, inner sep=0}, "f"', curve={height=6pt}, from=2-6, to=3-7]
	\arrow["f", curve={height=-6pt}, from=2-8, to=3-7]
	\arrow["{g.f\circ f^*\lambda\circ\langle\id,\id\rangle}", from=2-6, to=2-8]
	\arrow["{\lambda^\sharp}"'{pos=0.6}, shorten <=2pt, shorten >=4pt, Rightarrow, from=4-2, to=0]
	\arrow["{f^*\lambda^\sharp}"{pos=0.2}, color={rgb,255:red,214;green,92;blue,92}, shorten <=2pt, shorten >=4pt, Rightarrow, dashed, from=2-2, to=1]
	\arrow[shorten <=9pt, shorten >=9pt, Rightarrow, from=2-8, to=2]
\end{tikzcd}\]
and get a pair $(f^*\lambda,f^*\lambda^\sharp)$ such that on a pair of objects $F,F'$ over the same context
\[
f^*\lambda\colon (F,F') \mapsto (\lambda F, F'[\lambda^\sharp_F]),
\]
hence we can use the universal property of pullbacks to precompose the policy ``on top'' with $\langle\id,\id\rangle$ to get the lax triangle on the right. This now reads as
 \begin{prooftree}\AxiomC{$\Gamma\; \vdash F\; \mathbb{F}$}\LeftLabel{$(g.f\circ f^*\lambda\circ\langle\id,\id\rangle)$}\UnaryInfC{$\Theta \vdash F[\lambda^\sharp_F] \; \mathbb{F}$}\end{prooftree}
where $\lambda^\sharp_F:\Theta\to \Gamma$.

One could detail a similar argument for covariant policies: we do not do so here because in the present work we will only encounter contravariant ones. Still, the reader can easily see how covariant rules are strictly connected to comonads, and comonads have been proven of special interest in logic (see, for example, the treatment of equality in \cite{dagnino2021doctrines}), and this is why we have carried them through all definitions and technical proofs.

\subsection{On substitution}\label{onsubst}
\Cref{syntax_policies} is a first instance of application of substitution-like properties in our setting, in it is worth noticing that the additional data of a policy can be only \emph{externalized} in our setting when the target judgement classifier is a fibration. It seems worth it, then, to describe what information - from the point of view of judgemental theories - lies under the assumption that a functor is a fibration.

Recall from \Cref{charsharplift} that for a fibration $p\colon \ctg{E}\to\ctg{B}$ there are $(p^*\id,p^*\alpha)$ such that for any other $\sharp$-lifting $(g,\beta)$ there is a unique $p$-vertical $\overline{\beta}$ satisfying $\beta=p^*\alpha\ast\overline{\beta}$.

\begin{minipage}{0.49\textwidth}
\[\begin{tikzcd}[ampersand replacement=\&]
	{\ctg{B}\due.\cod\ctg{E}} \&\& {\ctg{E}} \\
	\& {\ctg{B}\due.\dom\ctg{E}} \\
	\\
	{\ctg{B}\due} \&\& {\ctg{B}} \\
	\& {\ctg{B}\due}
	\arrow["p", from=1-3, to=4-3]
	\arrow[""{name=0, anchor=center, inner sep=0}, "\cod", shift left=2, from=4-1, to=4-3]
	\arrow["\dom"', from=5-2, to=4-3]
	\arrow["\id"'{pos=0.3}, curve={height=12pt}, from=4-1, to=5-2]
	\arrow["{\dom.p}"'{pos=0.4}, from=2-2, to=1-3]
	\arrow[""{name=1, anchor=center, inner sep=0}, "{\cod.p}"{pos=0.4}, from=1-1, to=1-3]
	\arrow["g"', curve={height=18pt}, from=1-1, to=2-2]
	\arrow["\alpha"{pos=0.4}, shorten >=4pt, Rightarrow, from=5-2, to=0]
	\arrow["\beta", shorten >=3pt, Rightarrow, from=2-2, to=1]
\end{tikzcd}\]
\end{minipage}
\begin{minipage}{0.49\textwidth}
\[\begin{tikzcd}[ampersand replacement=\&]
	{\ctg{B}\due.\cod\ctg{E}} \&\& {\ctg{E}} \\
	\& {\ctg{B}\due.\dom\ctg{E}} \\
	\\
	{\ctg{B}\due} \&\& {\ctg{B}} \\
	\& {\ctg{B}\due}
	\arrow["p", from=1-3, to=4-3]
	\arrow[""{name=0, anchor=center, inner sep=0}, "\cod", shift left=2, from=4-1, to=4-3]
	\arrow["\dom"', from=5-2, to=4-3]
	\arrow["\id"'{pos=0.3}, curve={height=12pt}, from=4-1, to=5-2]
	\arrow["{\dom.p}"'{pos=0.4}, from=2-2, to=1-3]
	\arrow[""{name=1, anchor=center, inner sep=0}, "{\cod.p}"{pos=0.4}, from=1-1, to=1-3]
	\arrow[""{name=2, anchor=center, inner sep=0}, "{p^*\id}", from=1-1, to=2-2]
	\arrow[""{name=3, anchor=center, inner sep=0}, "g"', curve={height=18pt}, from=1-1, to=2-2]
	\arrow["\alpha"{pos=0.4}, shorten >=4pt, Rightarrow, from=5-2, to=0]
	\arrow["{p^*\alpha}"{pos=0.4}, shorten >=3pt, Rightarrow, from=2-2, to=1]
	\arrow["{\overline{\beta}}"', shift left=1, shorten <=4pt, shorten >=4pt, Rightarrow, from=3, to=2]
\end{tikzcd}\]
\end{minipage}

\noindent Therefore it seems that the peculiarity of fibrations lies in the existence of a unique vertical $\overline{\beta}$. We break down its meaning in the following policy, resulting from whiskering $\overline{\beta}$ with $p\circ\dom.p$,
\[\begin{tikzcd}[ampersand replacement=\&]
	{\ctg{B}\due.\cod\ctg{E}} \&\& {\ctg{B}\due.\dom\ctg{E}} \\
	\\
	\& {\ctg{B}}
	\arrow["g", from=1-1, to=1-3]
	\arrow[""{name=0, anchor=center, inner sep=0}, "{p\circ\dom.p\circ p^*\id}"', curve={height=6pt}, from=1-1, to=3-2]
	\arrow["{p\circ\dom.p}", curve={height=-6pt}, from=1-3, to=3-2]
	\arrow["{(p\circ\dom.p)\ast\overline{\beta}}"'{pos=0.3}, shift left=2, shorten <=15pt, shorten >=15pt, Rightarrow, from=1-3, to=0]
\end{tikzcd}\]

\noindent and it is easy to see that $p\circ\dom.p\circ p^*\id$ is a fibration, therefore we can apply the discussion in \Cref{syntax_policies} and derive the following rule in our judgemental theory. With $\sigma\colon\Theta\to pA$ and $B(A,\sigma)=(\dom.p\circ g )(A,\sigma)$, given that $(A,\sigma)[\overline{\beta_{A,\sigma}}]=B(A,\sigma)$,
\begin{prooftree}\AxiomC{$\Theta \vdash (A,\sigma)\;\ctg{B}\due.\cod\ctg{E}$}\UnaryInfC{$\Theta \vdash (B(A,\sigma),\sigma) \; \ctg{B}\due.\cod\ctg{E}$}
\end{prooftree}

\noindent meaning that initiality translates to the fact that any substitution is derivable from the cartesian one.

\subsection{Logics vs Theories}

The next two sections will focus on modeling dependent type theories and natural deduction in our framework. To be precise, the data of a judgemental theory will be \textit{the same as} a \textit{theory} satisfying the specifics of an intended logic. In its current state, this is a limitation of our framework: we do not offer a modular way to specify a \textit{logic} so that the theories in such a logic are precisely the judgemental theories of a certain shape, and we can only perform this \textit{presentation} via a case by case analysis (which is precisely the content of the next two sections). That said though, this programme is not outside our general scope, and we briefly address this topic in the last paragraph of \Cref{futurejt}.

\section{Plain dependent type theory}\label{secdtt}
In this section we show what features must a judgemental theory have in order to support dependent type theory. We show it produces desired rules, and with respect to this provide some evidence of the computational power of judgemental theories. We then pin-point which rules one needs to add in order to gain typically desirable constructors, for example dependent products and identity. In doing so, we learn something about constructors in general and give a (unifying) definition of extensional type constructor.

\begin{defn}[Plain dependent type theory] \label{jdtt}
A {\it plain dependent type theory} is a substitutional judgemental theory generated by the pre-judgemental theory described by the diagram below.
\[\begin{tikzcd}
	{\dot{\mathbb{U}}} && {\mathbb{U}} \\
	& {\ctx}
	\arrow["{\dot{u}}"', from=1-1, to=2-2]
	\arrow["u", from=1-3, to=2-2]
	\arrow[""{name=0, anchor=center, inner sep=0}, "\Sigma"{description}, from=1-1, to=1-3]
	\arrow[""{name=1, anchor=center, inner sep=0}, "\Delta"{description}, curve={height=18pt}, dashed, from=1-3, to=1-1]
	\arrow["\dashv"{anchor=center, rotate=90}, draw=none, from=0, to=1]
\end{tikzcd}\]
To be precise, we mean that $\classof{J} = \{\dot{u}, u\}$ and that those are fibrations, $\classof{R} = \{\Sigma, \Delta\}$, $\classof{P}$ contains the witness of the commutativity of the solid diagram, and  both the unit and the counit $(\epsilon, \eta)$ of the adjunction $\Sigma \dashv \Delta$. Finally, we require that $(\epsilon, \eta)$ are cartesian natural transformations.  We call this \emph{pDTT}, for short.
\end{defn}

We think of $\uu$ as classifying types, $\duu$ as classifying terms, and the functor $\Sigma$ as the one performing the typing. Its adjoint $\Delta$ will interpret context extension. The choice of the greek letters $\Sigma,\Delta$ is inspired by the notation classically used for polynomials, for example in \cite[p.7]{gambino_kock_2013}.

\begin{rem}[Notational caveats]\label{quelliso}
As we mentioned in \ref{thenotation} our notation, while being very telling, sometimes hides pieces of data. For example one finds that $\duu.\Sigma\duu\cong\duu\times\duu$. This is an instance of the fact that, depending on the choice of maps along which one performs the pullbacks (and depending on the order in which one does so), one gets a classifier that is either nested, or it is not. In general, the {\it nesting degree} is subject to change. Such equations, though unpretty, will be interesting from the point of view of the theory. Each time something of this kind happens, we will state it explicitly.
\end{rem}

\subsection{From natural models to plain dtts}

Recall that a natural model in the sense of \cite{awodey_2018} is the data of
\begin{enumerate}
    \item a category $\ctx$ with terminal object;
    \item an arrow $p: \dot{U} \to U$ in the presheaf category $\catof{Psh}(\ctx)$;
    \item some \textit{representability data}. This means that for all cospans as in the diagram below, we are given an object $\Gamma.A\in\ctg{C}$, a morphism $\delta_A :\Gamma.A \to \Gamma$ in $\ctx$ and an arrow $q_A : \hirayo (\Gamma.A) \to \dot{U}$, such that the square below is a pullback.
\[\begin{tikzcd}
	{\hirayo (\Gamma.A)} && {\dot{U}} \\
	\\
	{\hirayo \Gamma} && U
	\arrow["p"{description}, from=1-3, to=3-3]
	\arrow["A"{description}, from=3-1, to=3-3]
	\arrow["{\hirayo(\delta_A)}"{description}, dashed, from=1-1, to=3-1]
	\arrow["{q_A}"{description}, dashed, from=1-1, to=1-3]
\end{tikzcd}\]

\end{enumerate}

\begin{rem}[Use of the Yoneda lemma]\label{yoneda_extensive}
When working with natural models, the Yoneda lemma is heavily used and, in particular, for a presheaf $X$ over $\ctx$ we tend to identify objects that are in a correspondence under the following (natural) bijection.
\[
X(\Gamma) \cong \catof{NatTr}(\yo\Gamma,X)
\]
When we want to avoid using such an abuse, we denote with $x$ an element of $X(\Gamma)$ and $x^*$ its corresponding natural transformation and, conversely, for a natural transformation $y$ we call $y_*$ its corresponding element. One of the advantages of dealing with judgemental theories is that such an ambiguity will be avoided entirely.
\end{rem}

\begin{thm} \label{awodeytous}
A natural model is the same thing as a plain dependent type theory where the types and terms fibrations are discrete fibrations.
\end{thm}
The greatest part of the theorem relies on the following result. Recall that with respect to a discrete fibration, all maps are cartesian, hence whatever unit and counit we supply, their component will be, too.
\begin{prop}\label{reprreloaded}
Let $p: \dot{U} \to U$ a morphism of presheaves over $\ctx$ and $\Sigma_p$ its image through the Grothendieck biequivalence restricted to presheaves. The following are equivalent.
\begin{enumerate}
    \item We are provided with some representability data for $p$.
    \item The functor $\Sigma_p$ has a right adjoint $\Delta_p$.
\end{enumerate}
\end{prop}
\begin{rem}
It is evident from the discussion between page 245 and 246 of \cite{awodey_2018} that Awodey was aware of this result,  but because he only sketches the proof of the proposition above, we provide it in full.
\end{rem}
\begin{proof}[Proof of \ref{reprreloaded}]
First of all, let us briefly describe $\Sigma_p: \duu\to\uu$ in terms of $p:\dot{U}\to U$, or at least how it acts on the objects. The category $\duu$ has for objects pairs $(\Gamma,a)$ where $\Gamma$ is an object of $\ctx$ and $a\in \dot{U}(\Gamma)$. Similarly, the category $\uu$ has for objects pairs $(\Gamma,A)$ where $\Gamma$ is an object of $\ctx$ and $A\in {U}(\Gamma)$. The presheaf morphism $p$ induces a function $p_\Gamma:\dot{U}(\Gamma)\to U(\Gamma)$, therefore the (discrete) fibration morphism $\Sigma_p$ it induces maps a pair $(\Gamma,a)$ to $(\Gamma,p_\Gamma(a))$. We denote it $\Sigma_p$ in analogy with \Cref{jdtt}.
\begin{itemize}
    \item[($1 \Rightarrow 2$)] We will now construct the functor $\Delta_p$ provided that $p$ is representable.

\[\begin{tikzcd}
	{\dot{\mathbb{U}}} && {\mathbb{U}} \\
	\\
	& {\ctx}
	\arrow["{\dot{u}}"{description}, from=1-1, to=3-2]
	\arrow["u"{description}, from=1-3, to=3-2]
	\arrow[""{name=0, anchor=center, inner sep=0}, "{\Sigma_p}"{description}, from=1-1, to=1-3]
	\arrow[""{name=1, anchor=center, inner sep=0}, "{\Delta_p}"{description}, curve={height=18pt}, dashed, from=1-3, to=1-1]
	\arrow["\dashv"{anchor=center, rotate=90}, draw=none, from=0, to=1]
\end{tikzcd}\]

\noindent Consider an object $A \in \mathbb{U}$, recall that it corresponds by \Cref{yoneda_extensive} to an arrow $A^*: \hirayo uA \to U$. Then, we define $\Delta_p:A \mapsto (q_A)_*$, where the latter is obtained by the representability condition at $A^*$. On a substitution $\sigma:\Theta\to\Gamma$ we take pullbacks as depicted below.
\[\begin{tikzcd}
	& {\hirayo (\Gamma.A)} && {\dot{U}} \\
	{\hirayo(\Theta.B)} \\
	& {\hirayo \Gamma} && U \\
	{\hirayo \Theta}
	\arrow["p", from=1-4, to=3-4]
	\arrow["A"{description, pos=0.4}, from=3-2, to=3-4]
	\arrow["{\hirayo(\delta_A)}"{description, pos=0.6}, from=1-2, to=3-2]
	\arrow["{q_A}"{description, pos=0.4}, from=1-2, to=1-4]
	\arrow["{\hirayo (\sigma)}"{description}, curve={height=-12pt}, from=4-1, to=3-2]
	\arrow["B"{description}, from=4-1, to=3-4]
	\arrow[curve={height=-12pt}, dashed, from=2-1, to=1-2]
	\arrow["{\hirayo (\delta_B)}"{description, pos=0.4}, from=2-1, to=4-1]
	\arrow["{q_B}"{description}, from=2-1, to=1-4]
\end{tikzcd}\]

 We now need to show that $\Sigma_p \dashv \Delta_p$. The easiest thing is to provide the unit and the counit.
\begin{itemize}
    \item[($\epsilon$)] We want to construct an arrow $\epsilon_A : \Sigma_p \Delta_p A \to A$. We define it to be the cartesian lifting of $\delta_A$ at $A$. Now we need to show that this is a natural transformation, but that follows from the universal property of cartesian lifts. In fact, for each $s:B\to A$, the composition $s \circ u^* \delta_B$ is the cartesian lifting of $\sigma\circ \delta_B$, $u^* \delta_A \circ \Sigma_p\Delta_p s$ that of $\delta_A \circ \Delta_p(\sigma)$, and $\delta_A \circ \Delta_p(\sigma)=\sigma\circ \delta_B$, therefore, by uniqueness (up to iso) of the cartesian lifting, $u^* \delta_A \circ \Sigma_p\Delta_p s= s \circ u^* \delta_B$ too.
    \item[$(\eta)$] We want to construct an arrow $a \to \Delta_p \Sigma_p a$. This is also obtained by cartesian lifting, that of $\gamma_a$ induced by the dotted arrow in the diagram below.
\[\begin{tikzcd}
	{\hirayo \Gamma} \\
	& {\hirayo (\Gamma.A)} && {\dot{U}} \\
	\\
	& {\hirayo \Gamma} && U
	\arrow["p"{description}, from=2-4, to=4-4]
	\arrow["A"{description}, from=4-2, to=4-4]
	\arrow["{\hirayo(\delta_A)}"{description}, dashed, from=2-2, to=4-2]
	\arrow["{q_A}"{description}, dashed, from=2-2, to=2-4]
	\arrow["a"{description}, curve={height=-12pt}, from=1-1, to=2-4]
	\arrow["id"{description}, curve={height=18pt}, from=1-1, to=4-2]
	\arrow[dotted, from=1-1, to=2-2]
\end{tikzcd}\]
\noindent Naturality follows as for $\epsilon$.
\end{itemize}
Triangle identities of $(\epsilon,\eta)$ lie above commutative diagrams, for $\Sigma_p$ and $\Delta_p$ respectively
$$ id_{\Gamma}=\delta_A \circ \gamma_a \quad\text{and}\quad id_{\Gamma.A}= \delta_{A\delta_A} \circ \gamma_{q_A},$$
therefore they are satisfied again by uniqueness of the cartesian lifting.
\item[$(2) \Rightarrow (1)$] The diagram below describes the representability data.
\[\begin{tikzcd}
	{\hirayo (u\Sigma_p\Delta_p A)} && {\dot{U}} \\
	\\
	{\hirayo \Gamma} && U
	\arrow["p"{description}, from=1-3, to=3-3]
	\arrow["A"{description}, from=3-1, to=3-3]
	\arrow["{\hirayo(u(\epsilon_A))}"{description}, dashed, from=1-1, to=3-1]
	\arrow["{\Delta_p A}"{description}, dashed, from=1-1, to=1-3]
\end{tikzcd}\]

\noindent It is a pullback by the universal property of $\epsilon$: for each pair $(\sigma,b)$ such that $A\circ\sigma = p\circ b$, there is a map
$$ s: p \circ b = \Sigma_p(b) \to A$$
induced by precomposition with $\sigma$. Therefore there must be a unique $\phi$ such that $\epsilon\circ\Sigma_p\phi=s$. Now $u(\Sigma_p\phi)$ gives the desired map into $u\Sigma_p\Delta_p A$.
\end{itemize}
\end{proof}

\begin{proof}[Proof of \ref{awodeytous}]
There is a clear correspondence between couples $(\ctx, p: \dot{U} \to U)$ and triangles as below, where $\du,u$ are discrete fibrations. The correspondence is given by the Grothendieck construction.

 \begin{minipage}{0.45\textwidth}
$$(\ctx, p: \dot{U} \to U)$$
 \end{minipage}
  \begin{minipage}{0.60\textwidth}
\[\begin{tikzcd}
	{\dot{\mathbb{U}}} && {\mathbb{U}} \\
	& {\ctx}
	\arrow["{\dot{u}}"', from=1-1, to=2-2]
	\arrow["u", from=1-3, to=2-2]
	\arrow["\Sigma_p"{description}, from=1-1, to=1-3]
\end{tikzcd}\]
 \end{minipage}
The additional axioms required on both ends are equivalent because of \Cref{reprreloaded}.
\end{proof}

\subsection{Plain dtts vs comprehension categories}\label{jdttvcompcat}
Another categorical approach to dependent type theories which is historically very meaningful was given by Jacobs in \cite{jacobs1999categorical}. This is the theory of comprehension categories and it is inherently presented in the form of a pre-judgemental theory as below.
\[\begin{tikzcd}
	{\mathbb{U}} && {\ctx\due} \\
	\\
	& {\ctx}
	\arrow["{\disp}"{description}, from=1-1, to=1-3]
	\arrow["u"{description}, curve={height=12pt}, from=1-1, to=3-2]
	\arrow["{\cod}"{description}, curve={height=-12pt}, from=1-3, to=3-2]
\end{tikzcd}\]
Comprehension categories clearly realize some form of context extension, and that is given by display maps.

 \begin{constr}[From pDTTs to comprehension categories] \label{tojacobs}
Each plain dependent type theory produces a comprehension category as described by the steps below.
\[\begin{tikzcd}
	{\mathbb{U}} & {\dot{\mathbb{U}}} & {\mathbb{U}} & {\ctx} & {\mathbb{U}} && {\dot{\mathbb{U}}} \\
	&&&&& {\ctx} \\
	{\mathbb{U}} && {\ctx} && {\mathbb{U}} && {\ctx\due}
	\arrow["\Delta", from=1-1, to=1-2]
	\arrow["\Sigma", from=1-2, to=1-3]
	\arrow["u", from=1-3, to=1-4]
	\arrow[""{name=0, anchor=center, inner sep=0}, "id"', curve={height=18pt}, from=1-1, to=1-3]
	\arrow["{\dot{u}}"{description}, curve={height=-18pt}, from=1-2, to=1-4]
	\arrow[""{name=1, anchor=center, inner sep=0}, "{\dot{u}\Delta}"{description}, shift left=3, from=3-1, to=3-3]
	\arrow[""{name=2, anchor=center, inner sep=0}, "u"{description}, shift right=3, from=3-1, to=3-3]
	\arrow["\Delta", from=1-5, to=1-7]
	\arrow[""{name=3, anchor=center, inner sep=0}, "u"', from=1-5, to=2-6]
	\arrow["{\dot{u}}", from=1-7, to=2-6]
	\arrow["{\disp}", from=3-5, to=3-7]
	\arrow["\epsilon"'{pos=0.1}, shorten >=1pt, Rightarrow, from=1-2, to=0]
	\arrow["\delta"', shorten <=2pt, shorten >=2pt, Rightarrow, from=1, to=2]
	\arrow["\delta"{description}, shorten <=8pt, shorten >=8pt, Rightarrow, from=1-7, to=3]
\end{tikzcd}\]
\vspace{.1cm}

It is enough to follow the picture from left to right (and top to bottom) to see how a plain dependent type theory in our sense produces a \textit{display} functor, which thus specifies a comprehension category.
 \end{constr}

Of course it is a legitimate question to ask whether every comprehension category can be realized via a plain dependent type theory. Turns out that the two are in fact equivalent, and to prove such a thing is the starting point of \cite{biequivCoEm}.

\subsection{Dictionary} \label{dictionnaire}

Dependent type theory has a well established notation, which we switch to in this subsection. The table below declares the dictionary between our framework and the classical notation.

Following the presentation in \Cref{prolegomena}, it will need to take into account Syntax (but there is not much to say there), Judgements, and Rules. What we adopt here is a one-to-one rewriting of (some) components introduced in \Cref{calculi} in order to make the calculations we will see more transparent. Still, each string of symbols will simply represent its categorical backbone.

\subsubsection{Dictionary for judgements} \label{dicjud}
As we mentioned in \Cref{jdtt}, we think of $\uu$ as classifing types, $\duu$ as classifing terms, and of $\Sigma$ as performing the typing. We make this clear with the choices in the translation that follow. (Sometimes we might omit the word ${\tt Type}$ for brevity).
\begin{center}
\begin{tabular}{ c | c }
\hline
 $\Gamma \vdash A \; \mathbb{U}$ & $\Gamma \vdash A\; {\tt Type} $ \\
  $\Gamma \vdash a \; \dot{\mathbb{U}} \quad (\Gamma \vdash A \; \mathbb{U} \quad \Gamma \vdash \Sigma a =_{\mathbb{U}} A)$ & $\Gamma \vdash a : A$\\
   $\Gamma \vdash A=_{\uu}B \quad (\Gamma \vdash A \; \mathbb{U} \quad \Gamma \vdash B \; \mathbb{U})$ & $\Gamma \vdash A=B\; {\tt Type} $ \\
   $\Gamma \vdash a=_{\duu}b \quad (\Gamma \vdash A \; \mathbb{U} \quad \Gamma \vdash \Sigma a =_{\mathbb{U}} A \quad \Gamma \vdash \Sigma b =_{\mathbb{U}} A )$ & $\Gamma \vdash a=b:A $ \\
\hline
\end{tabular}
\end{center}

\begin{rem}[How many types to a term?]
One might see our choice in the treatment of typing as profoundly Church-like, in the sense that to one term we only assign one type via the functor $\Sigma$, and that is far from the practice. The generality of our definition, though, allows for some tweaks, so that if one wishes to have the possibility of assigning different types to the same term (say both $0:{\tt N}$ and $0:{\tt Z}$) one can simply choose $\duu$ as a subcategory of two categories with, respectively, names for terms and for types (hence code the two above as $(0,{\tt N})$ and $(0,{\tt Z})$), and make $\Sigma$ act as a second projection.
\end{rem}

\subsubsection{Dictionary for rules} \label{dicrules}
We also have a dictionary for rules, which we have (at least) two of. The first is implicitly used in \Cref{dicjud}, and it is the typing.

\begin{minipage}{0.35\textwidth}
\[\begin{tikzcd}
	\duu && \uu \\
	& \ctx
	\arrow["\Sigma", from=1-1, to=1-3]
	\arrow["{\dot{u}}"', from=1-1, to=2-2]
	\arrow["u", from=1-3, to=2-2]
\end{tikzcd}\]
\end{minipage}
\begin{minipage}{0.3\textwidth}
\begin{prooftree}\AxiomC{$\Gamma  \vdash a\; \duu$}\LeftLabel{$(\Sigma)$}\UnaryInfC{$\Gamma \vdash \Sigma a \; \uu $}\end{prooftree}
\end{minipage}
\begin{minipage}{0.25\textwidth}
\begin{prooftree}\AxiomC{$\Gamma \vdash a: \Sigma a$}\LeftLabel{$(\Sigma)$}\UnaryInfC{$\Gamma \vdash \Sigma a \;{\tt Type}$}\end{prooftree}
\end{minipage}\hfill

\noindent The second is the policy $\delta$ from \Cref{tojacobs}, which we here denote as follows.
\begin{minipage}{0.35\textwidth}
\[\begin{tikzcd}
	{\mathbb{U}} && {\dot{\mathbb{U}}} \\
	& {\ctx}
	\arrow["\Delta", from=1-1, to=1-3]
	\arrow[""{name=0, anchor=center, inner sep=0}, "u"', from=1-1, to=2-2]
	\arrow["{\dot{u}}", from=1-3, to=2-2]
	\arrow["\delta"{description}, shorten <=8pt, shorten >=8pt, Rightarrow, from=1-3, to=0]
\end{tikzcd}\]
\end{minipage}
\begin{minipage}{0.3\textwidth}
\begin{prooftree}\AxiomC{$\Gamma\; \vdash A\; \mathbb{U}$}\LeftLabel{$(\delta)$}\UnaryInfC{$\dot{u}\Delta A \vdash  \Delta A\; \dot{\mathbb{U}}$}\end{prooftree}
\end{minipage}
\begin{minipage}{0.3\textwidth}
\begin{prooftree}\AxiomC{$\Gamma\; \vdash A\; {\tt Type}$}\LeftLabel{$(\delta)$}\UnaryInfC{$\Gamma.A \vdash  q_A:  \Sigma\Delta A$}\end{prooftree}
\end{minipage}\hfill
\vspace{.08cm}
\noindent Again, such writings are only stand-ins for their categorical counterparts.

\subsection{Context extension and type dependecy} \label{cetd}

In this subsection we compute some rules that are automatically deduced by the finite-limit closure of a plain dependent type theory. As we will see, they correspond to some very well known rules in dependent type theories.

\subsubsection{Context extension in a DTT, explicitly}\label{explicitly}

\begin{notat}
For readibility reasons, and in order to highlight the correspondence between the logic and the categories without trivializing it, we denote $A\sigma$ the result of the cartesian lifting of $A$ along $\sigma$ and $A[\sigma]$ the substitution in the sense of the type theory.
\end{notat}

All of the pieces appearing in the dictionary \ref{dicrules} surely do look familiar to the type-theorist reader, all but one, and that is $\Sigma\Delta A$. In fact one might rightfully ask how to compute such an object.
\begin{prop}[On a formal emergence of substitution] \label{computingsub}
Let $A$ be on object in $\mathbb{U}$. Then, $\Sigma\Delta A = A \delta_{A}$, in the sense of \Cref{syntax_policies}.
\end{prop}
\begin{proof}
 We know that there is an arrow $\epsilon_A: \Sigma\Delta A \to A$. By the discussion in \Cref{syntax_policies}, the thesis is equivalent to the fact that the cartesian lifting of $\delta_{A}$ along $u$ is precisely $\epsilon_A$. Recall, that $\delta_A$ is by definition $u (\epsilon_A)$, therefore it is a lifting. It is cartesian by assumption.
\end{proof}
Notice that this is as trivial as (and in fact it amounts to) proving that the process of computing weakening \emph{can} be simulated in the syntax using substitution, provided that suitable substitution rules do in fact exist. We can re-read the rule hidden in the policy $\delta$ as follows.

\begin{minipage}{0.49\textwidth}
    \begin{prooftree}\AxiomC{$\Gamma \vdash A \;\mathbb{U} $}\LeftLabel{($\delta$)}\UnaryInfC{$\dot{u}\Delta A \vdash\; \Delta A \dot{\mathbb{U}}$}\end{prooftree}
\end{minipage}
\begin{minipage}{0.49\textwidth}
    \begin{prooftree}\AxiomC{$\Gamma \vdash A\; {\tt Type}$}\LeftLabel{($\delta$)}\UnaryInfC{$\Gamma.A \vdash q_A : A\delta_A$}\end{prooftree}
\end{minipage}
\vspace{.08cm}

\noindent Finally, we observe that the deductive rule on the right is a version of {\it context extension} in dependent type theory.

\subsubsection{Type dependency in a DTT, explicitly}\label{typedep}
Similarly to the case of context extension, in a pDTT as in \Cref{jdtt} the most classical instances of type depencency emerge too. Let us produce the following two rules. 

\begin{minipage}{0.49\textwidth}
    \begin{prooftree}\AxiomC{$\Gamma\vdash a:A$}\AxiomC{$\Gamma.A\vdash B\; {\tt Type}$}\LeftLabel{(DTy)}\BinaryInfC{$\Gamma\vdash B \la a \ra \; {\tt Type}$}\end{prooftree}
\end{minipage}
\begin{minipage}{0.49\textwidth}
      \begin{prooftree}\AxiomC{$\Gamma\vdash a:A$}\AxiomC{$\Gamma.A\vdash b:B$}\LeftLabel{(DTm)}\BinaryInfC{$\Gamma\vdash b \la a \ra : B\la a \ra$}\end{prooftree}
\end{minipage}\hfill
\vspace{.03cm}

\noindent In order to do so, we first need (nested) classifiers for the premises. More generally, with an iterated construction we will code composed judgements of the form below.
\begin{align*}
\Gamma\vdash a:A, \Gamma.A\vdash b:B & \quad\quad\quad \Gamma\vdash A,\Gamma.A\vdash b:B \\
\Gamma\vdash a:A, \Gamma.A\vdash B & \quad\quad\quad \Gamma\vdash A, \Gamma.A\vdash B
\end{align*}
This is achieved as follows.
\[\begin{tikzcd}
	{\duu.\Sigma\Delta\duu} & {\mathbb{U}.\Delta\dot{\mathbb{U}}} && {\dot{\mathbb{U}}} \\
	{\duu.\Sigma\Delta\uu} & {\mathbb{U}.\Delta{\mathbb{U}}} && {\mathbb{U}} & {} \\
	\duu & {\mathbb{U}} & {\dot{\mathbb{U}}} & {\ctx}
	\arrow[from=1-2, to=1-4]
	\arrow["\Delta"', from=3-2, to=3-3]
	\arrow["{\dot{u}}"', from=3-3, to=3-4]
	\arrow["u", from=2-4, to=3-4]
	\arrow[""{name=0, anchor=center, inner sep=0}, from=2-2, to=2-4]
	\arrow[from=2-2, to=3-2]
	\arrow["\Sigma", from=1-4, to=2-4]
	\arrow[from=1-2, to=2-2]
	\arrow[""{name=1, anchor=center, inner sep=0}, "\Sigma"', from=3-1, to=3-2]
	\arrow[from=1-1, to=2-1]
	\arrow[from=2-1, to=3-1]
	\arrow[""{name=2, anchor=center, inner sep=0}, from=2-1, to=2-2]
	\arrow[from=1-1, to=1-2]
	\arrow["\ulcorner"{anchor=center, pos=0.125}, draw=none, from=2-2, to=3-3]
	\arrow["\ulcorner"{anchor=center, pos=0.125}, draw=none, from=2-1, to=1]
	\arrow["\ulcorner"{anchor=center, pos=0.125}, draw=none, from=1-1, to=2]
	\arrow["\ulcorner"{anchor=center, pos=0.125}, draw=none, from=1-2, to=0]
\end{tikzcd}\]

\noindent For example, the fibration on $\uu.\Delta\uu$ classifies pairs $(A,B)$ of types such that $u\Sigma\Delta(A)=u(B)$. This is precisely the composed judgement $\Gamma\vdash A,\Gamma.A\vdash B$.

\begin{lem}[Focus on $\mathbb{U}.\Delta\mathbb{U}$] \label{ciccino}
In a plain dtt we have the following rules and policy.
\[\begin{tikzcd}[ampersand replacement=\&]
	{\mathbb{U}.\Delta{\mathbb{U}}} \& {\mathbb{U}} \&\& {\ctx} \\
	{(\Gamma\vdash A,\Gamma.A\vdash B)} \& {(\Gamma\vdash A)} \&\& {(\Gamma.A \to\Gamma)}
	\arrow["{u.\dot{u}\Delta}", from=1-1, to=1-2]
	\arrow[""{name=0, anchor=center, inner sep=0}, "u"{description}, curve={height=12pt}, from=1-2, to=1-4]
	\arrow[""{name=1, anchor=center, inner sep=0}, "{\dot{u} \circ\Delta}"{description}, curve={height=-12pt}, from=1-2, to=1-4]
	\arrow[maps to, from=2-1, to=2-2]
	\arrow[maps to, from=2-2, to=2-4]
	\arrow["\delta", shorten <=3pt, shorten >=3pt, Rightarrow, from=1, to=0]
\end{tikzcd}\]
\end{lem}
\begin{proof}
This is the first  detailed instance of two judgement classifiers supported by the same category, since one could perform the two following compositions
\[\begin{tikzcd}[ampersand replacement=\&]
	{\mathbb{U}.\Delta{\mathbb{U}}} \&\& {\mathbb{U}} \\
	{\mathbb{U}} \&\& {\ctx} \\
	\ctx
	\arrow["u", from=1-3, to=2-3]
	\arrow[from=1-1, to=1-3]
	\arrow[from=1-1, to=2-1]
	\arrow[""{name=0, anchor=center, inner sep=0}, "\du\circ\Delta", from=2-1, to=2-3]
	\arrow["\lrcorner"{anchor=center, pos=0.125}, draw=none, from=1-1, to=2-3]
	\arrow[""{name=1, anchor=center, inner sep=0}, "u"', from=2-1, to=3-1]
	\arrow["\Id", from=2-3, to=3-1]
	\arrow["\delta", shorten <=7pt, shorten >=7pt, Rightarrow, from=0, to=1]
\end{tikzcd}\]
which are related as discussed in \Cref{ctxextdavero}. Such a policy is the symptom of a shift in perspective: on the upper path, one travels along the pullback diagram above, therefore the context which one lands on is $\Gamma.A$; on the lower, one is concerned with the ``original'' context of $A$, therefore getting to $\Gamma$. They are related, as we have thoroughly discussed, by $\delta_A$. Notice that the lower path, being a composition of fibrations, is a fibration as well.
\end{proof}
Since the classifier in the lower part of the diagram in \ref{ciccino} will play an important role in a later discussion, we name it, \[v: \uu.\Delta\uu \to \ctx. \]

In \cite[Prop. 2.2]{awodey_2018} there is the construction of a presheaf $P(\uu)$, with $P$ a polynomial functor, classifying the same nested judgement as $\uu.\Delta\uu$. The polynomial is defined as follows
\[\begin{tikzcd}[ampersand replacement=\&]
	{P=P_p\colon\catof{Psh}(\ctx)} \& {\catof{Psh}(\ctx)_{/\dot{U}}} \& {\catof{Psh}(\ctx)_{/{U}}} \& {\catof{Psh}(\ctx)}
	\arrow["{\dot{U}^*}", from=1-1, to=1-2]
	\arrow["{\Pi_p}", from=1-2, to=1-3]
	\arrow["{\Sigma_U}", from=1-3, to=1-4]
\end{tikzcd}\]
meaning the pullback along the terminal presheaf morphism from $\dot{U}$, followed by the right adjoint to pullback along $p$, followed by composition with the terminal from $U$. We apologize for the ambiguous notation ($-^*$, $\Pi$, $\Sigma$), but we promise this will only be used in the current section. 

\begin{lem}[Classifiers à la Awodey] \label{parricidio}
One can show that the fibration $v: \uu.\Delta\uu \to \ctx$ is precisely the projection $\pi:\hirayo\!\downarrow\!P(U)\to\ctx$.
\end{lem}
\begin{proof}
We sketch the identity fiber-wise. At each $\Gamma$, $(\uu.\Delta\uu)_\Gamma$ is comprised of pairs $(A,B)$ with $A, B$ in $\uu$ such that $u(B)=u(\Sigma\Delta A)$ and $u(A)=\Gamma$. By \Cref{yoneda_extensive}, such $A$ and $B$ correspond to $A^*$ and $B^*$ fitting in the following diagram,
\[\begin{tikzcd}[ampersand replacement=\&]
	U \& {\yo\Gamma.A} \& {\dot{U}} \\
	\& \yo\Gamma \& U
	\arrow["p", from=1-3, to=2-3]
	\arrow["{A^*}"', from=2-2, to=2-3]
	\arrow[from=1-2, to=2-2]
	\arrow[from=1-2, to=1-3]
	\arrow["\lrcorner"{anchor=center, pos=0.125}, draw=none, from=1-2, to=2-3]
	\arrow["{B^*}"', from=1-2, to=1-1]
\end{tikzcd}\]
with the central square being a pullback by representability of $p$. Using a result from \cite{DYCKHOFF1987103}, \cite[Prop. 2.2]{awodey_2018} shows that such diagrams are in a 1-to-1 correspondence with maps of the form $\yo \Gamma\to P(U)$,
hence with elements of $P(U)(\Gamma)$.
\end{proof}

We have shown that there is a very tight connection between our classifier and Awodey's. 
We hope that, though almost tautological, this result can convince the reader about the advantages of our construction, as it makes it much easier to predict the correct pullback that constructs the desired classifier (this will be more and more evident in the following sections), while it might not be always easy to find suitable (polynomial) functors to classify complex judgements. 
Also, we can avoid the complex machinery of polynomial functors (and, in this case, the conflicting notation).

In order to provide the rules (DTm) and (DTy) we build a map out of $\duu.\Sigma\Delta\uu$ (and of $\duu.\Sigma\Delta\duu$), and all we have is $\Sigma,\Delta,\eta,\epsilon$, finite limits closure, composition, substitution, whiskering, and $\sharp$-lifting. A few tries lead us to the following choice.

\vspace{.03cm}
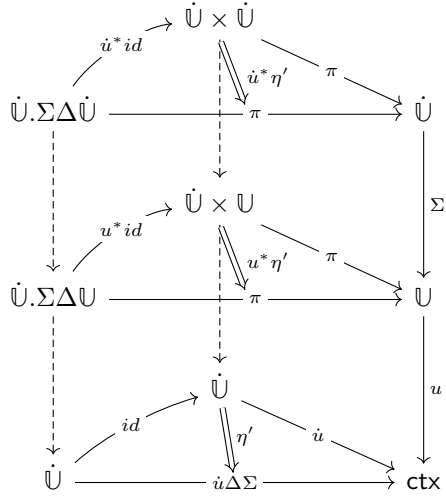
 \begin{minipage}{0.5\textwidth}

\[\begin{tikzcd}
	& \duu\times\duu \\
	{\duu.\Sigma\Delta\duu} &&& \duu \\
	& \duu\times\uu \\
	{\duu.\Sigma\Delta\uu} &&& \uu \\
	& \duu \\
	\duu &&& \ctx
	\arrow[""{name=0, anchor=center, inner sep=0}, "\pi"{description}, from=4-1, to=4-4]
	\arrow["{u^* id}"{description}, curve={height=-12pt}, from=4-1, to=3-2]
	\arrow["\pi"{description}, from=3-2, to=4-4]
	\arrow["id"{description}, curve={height=-6pt}, from=6-1, to=5-2]
	\arrow[""{name=1, anchor=center, inner sep=0}, "{\dot{u}\Delta\Sigma}"{description}, from=6-1, to=6-4]
	\arrow["{\dot u}"{description}, from=5-2, to=6-4]
	\arrow["u", from=4-4, to=6-4]
	\arrow[dashed, from=4-1, to=6-1]
	\arrow[dashed, from=3-2, to=5-2]
	\arrow["\Sigma", from=2-4, to=4-4]
	\arrow[dashed, from=2-1, to=4-1]
	\arrow[dashed, from=1-2, to=3-2]
	\arrow[""{name=2, anchor=center, inner sep=0}, "\pi"{description}, from=2-1, to=2-4]
	\arrow["\pi"{description}, from=1-2, to=2-4]
	\arrow["{\dot{u}^*id}"{description}, curve={height=-12pt}, from=2-1, to=1-2]
	\arrow["{u^*\eta'}"{pos=0.7}, shift right=1, shorten >=4pt, Rightarrow, from=3-2, to=0]
	\arrow["{\eta'}"{pos=0.6}, shorten >=3pt, Rightarrow, from=5-2, to=1]
	\arrow["{\dot{u}^*\eta'}"{pos=0.7}, shift right=1, shorten >=4pt, Rightarrow, from=1-2, to=2]
\end{tikzcd}\]

 \end{minipage}\hfill
 \begin{minipage}{0.45\textwidth}

    We call $\eta':\dot{u}\Rightarrow\dot{u}\Delta\Sigma$ the natural transformation induced by $\eta$ via \cref{whiskering} and apply $\sharp$-lifting (\Cref{lifting}) as on the left. Write $\pi$ for projections.
    
    When we compute each lifting, we see that the policy $(\dot{u}^*\eta')$ computes, starting from a pair $(a,b)$ some new term in context $\Gamma$, while the policy $(u^*\eta')$ matches to a pair $(a,B)$ a new type in context $\Gamma$.

    We give each a meaningful name, that is, extensively:
    $$ \dot{u}^* id (a,b) = (a, b \la a \ra),$$
    $$ u^* id (a,B) = (a, B \la a \ra).$$

    Notice that the typing is appropriate due to the action of the vertical $\Sigma$.
 \end{minipage}
 \vspace{.3cm}

We are now one step away from having (DTy) and (DTm), and in fact the distance between the policies $u^* id$, $\dot{u}^* id$ and the desired rules is extremely subtle, and one could argue - though the author might disagree - a merely technical one: on the premise of, say, dependent typing, we now have the following nested judgement (which we write in our original notation for judgemental theories, so that we can make the difference evident)
$$ \Gamma.A \vdash (a,B) \; \duu.\Sigma\Delta\uu(u \circ \dot{u}\Delta\Sigma.u)$$
while we wish to have the pair stand over $\Gamma$. That is achieved by $v$ (that from \Cref{ciccino}),
$$ \Gamma \vdash (a,B) \; \duu.\Sigma\Delta\uu(v) \, ,$$
therefore we need to adjust the two policies accordingly. We can do that by regular 2-categorical manipulations attaching $v$ (the composition of the colored arrows below) to the diagram above.

 \begin{minipage}{0.50\textwidth}
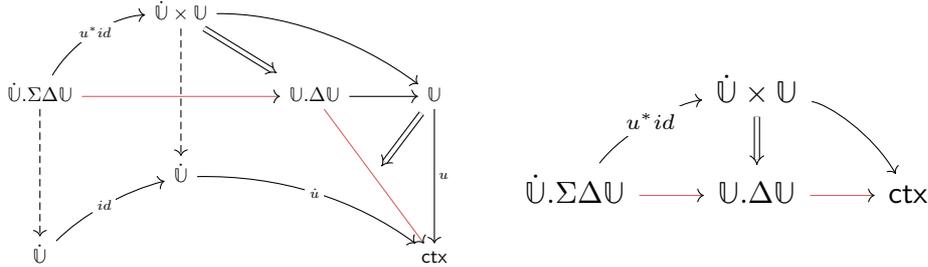
\begin{tikzcd}[scale cd = 0.75]
	& \duu\times\uu \\
	{\duu.\Sigma\Delta\uu} && {\uu.\Delta\uu} & \uu \\
	& \duu \\
	\duu &&& \ctx
	\arrow["{u^* id}"{description}, curve={height=-12pt}, from=2-1, to=1-2]
	\arrow[curve={height=-12pt}, from=1-2, to=2-4]
	\arrow["id"{description}, curve={height=-6pt}, from=4-1, to=3-2]
	\arrow["{\dot u}"{description}, curve={height=-12pt}, from=3-2, to=4-4]
	\arrow["u", from=2-4, to=4-4]
	\arrow[dashed, from=2-1, to=4-1]
	\arrow[dashed, from=1-2, to=3-2]
	\arrow[draw={rgb,255:red,214;green,92;blue,92}, from=2-1, to=2-3]
	\arrow[from=2-3, to=2-4]
	\arrow[shorten >=5pt, Rightarrow, from=1-2, to=2-3]
	\arrow[""{name=0, anchor=center, inner sep=0}, draw={rgb,255:red,214;green,92;blue,92}, from=2-3, to=4-4]
	\arrow[shorten <=2pt, shorten >=4pt, Rightarrow, from=2-4, to=0]
\end{tikzcd}
 \end{minipage}
  \begin{minipage}{0.50\textwidth}
\[\begin{tikzcd}[scale cd = 1.15]
	& \duu\times\uu \\
	{\duu.\Sigma\Delta\uu} & {\uu.\Delta\uu} & \ctx
	\arrow["{u^* id}"{description}, curve={height=-12pt}, from=2-1, to=1-2]
	\arrow[draw={rgb,255:red,214;green,92;blue,92}, from=2-1, to=2-2]
	\arrow[shorten >=2pt, Rightarrow, from=1-2, to=2-2]
	\arrow[draw={rgb,255:red,214;green,92;blue,92}, from=2-2, to=2-3]
	\arrow[curve={height=-12pt}, from=1-2, to=2-3]
\end{tikzcd}\]
\end{minipage}

The policy on the right now is (DTy). One could repeat a similar argument for terms, which again have the correct typing because of the action of $\Sigma$ in the $\sharp$-lifting above.

\begin{rem}[Similarities between DTy and proof theoretic Cut]
In the next section we highlight a remarkable connection between dependent typing and the cut rule from natural deduction: we redirect the reader to \Cref{dtycut} for more information.
\end{rem}

\subsubsection{Substitution along display maps}\label{ctxextdavero}
Of course there are (at least) two interesting natural transformations that we know of insisting on
\[\begin{tikzcd}
	\duu && \uu && \duu && \ctx
	\arrow["\Sigma", from=1-1, to=1-3]
	\arrow["\Delta", from=1-3, to=1-5]
	\arrow["{\dot u}", from=1-5, to=1-7]
\end{tikzcd}\]
that is $\eta$ and $\epsilon$. If $\eta$ is so interesting, one might wonder what repeating the process discussed in \Cref{typedep} with $\epsilon$ might bring. We have a hint about its outcome, and that is given by the $\delta$ from \Cref{tojacobs}, still we compute it precisely.

\begin{minipage}{0.45\textwidth}
\[\begin{tikzcd}
	& {\duu.\Sigma\Delta\duu} \\
	\duu\times\duu &&& \duu \\
	& {\duu.\Sigma\Delta\uu} \\
	\duu\times\uu &&& \uu \\
	& \duu \\
	\duu &&& \ctx
	\arrow[""{name=0, anchor=center, inner sep=0}, "\pi"{description}, from=4-1, to=4-4]
	\arrow["{u^* id}"{description}, curve={height=-12pt}, from=4-1, to=3-2]
	\arrow["\pi"{description}, from=3-2, to=4-4]
	\arrow["id"{description}, curve={height=-6pt}, from=6-1, to=5-2]
	\arrow[""{name=1, anchor=center, inner sep=0}, "{\dot{u}}"{description}, from=6-1, to=6-4]
	\arrow["u\Sigma\Delta\Sigma"{description}, from=5-2, to=6-4]
	\arrow["u", from=4-4, to=6-4]
	\arrow[dashed, from=4-1, to=6-1]
	\arrow[dashed, from=3-2, to=5-2]
	\arrow["\Sigma", from=2-4, to=4-4]
	\arrow[dashed, from=2-1, to=4-1]
	\arrow[dashed, from=1-2, to=3-2]
	\arrow[""{name=2, anchor=center, inner sep=0}, "\pi"{description}, from=2-1, to=2-4]
	\arrow["\pi"{description}, from=1-2, to=2-4]
	\arrow["{\dot{u}^*id}"{description}, curve={height=-12pt}, from=2-1, to=1-2]
	\arrow["{u^*\epsilon'}"{pos=0.7}, shift right=1, shorten >=4pt, Rightarrow, from=3-2, to=0]
	\arrow["{\epsilon'}"{pos=0.6}, shorten >=3pt, Rightarrow, from=5-2, to=1]
	\arrow["{\dot{u}^*\epsilon'}"{pos=0.7}, shift right=1, shorten >=4pt, Rightarrow, from=1-2, to=2]
\end{tikzcd}\]
 \end{minipage}\hfill
 \begin{minipage}{0.45\textwidth}
  We call $\epsilon':\dot{u}\Delta\Sigma\Rightarrow\dot{u}$. The construction detailed here, when explicitly computed, induces the two following rules involving $\delta_A:\Gamma.A\to \Gamma$,
    $$ \dot{u}^* id (a,a') = (a, a'\delta_A)$$
    $$ u^* id (a,A') = (a, A'\delta_A)$$
    meaning we can transport terms and types along arbitrary display maps, given that they insist on the same context.

    \begin{prooftree} \AxiomC{$\Gamma\vdash a:A$}\AxiomC{$\Gamma\vdash a':A'$}\BinaryInfC{$\Gamma\vdash a:A \qquad \Gamma.A \vdash a'\delta_A:A'\delta_A $}\end{prooftree}
 \end{minipage}

\begin{rem}[Rules for free]\label{triangid}
Since we now have rules involving the unit and rules involving the counit of an adjunction, we can exploit their relation to one another and show once again the computational power of judgemental theories. In particular, the (bases of the) constructions in \Cref{typedep} and \Cref{ctxextdavero} are related by the triangle identities:

\[\begin{tikzcd}
	\duu && \uu && \ctx && \duu \\
	\\
	\duu & \uu & \duu & \uu & \ctx && \ctx
	\arrow["\Sigma", from=3-1, to=3-2]
	\arrow["\Delta", from=3-2, to=3-3]
	\arrow["\Sigma", from=3-3, to=3-4]
	\arrow["u", from=3-4, to=3-5]
	\arrow[""{name=0, anchor=center, inner sep=0}, curve={height=-24pt}, from=3-1, to=3-3]
	\arrow[""{name=1, anchor=center, inner sep=0}, curve={height=24pt}, from=3-2, to=3-4]
	\arrow["\Sigma", from=1-1, to=1-3]
	\arrow["u", from=1-3, to=1-5]
	\arrow[""{name=2, anchor=center, inner sep=0}, from=1-7, to=3-7]
	\arrow[""{name=3, anchor=center, inner sep=0}, curve={height=-24pt}, from=1-7, to=3-7]
	\arrow[""{name=4, anchor=center, inner sep=0}, curve={height=24pt}, from=1-7, to=3-7]
	\arrow["\eta", shorten <=2pt, Rightarrow, from=0, to=3-2]
	\arrow["\epsilon"{pos=0.4}, shorten >=2pt, Rightarrow, from=3-3, to=1]
	\arrow["{\eta'}", shorten <=5pt, shorten >=5pt, Rightarrow, from=3, to=2]
	\arrow["{\epsilon'}", shorten <=5pt, shorten >=5pt, Rightarrow, from=2, to=4]
\end{tikzcd}\]
so that whiskering the two $\sharp$-liftings above to compute $\epsilon$ after $\eta$ yields the functor $\dot{u}$. Then at each level we have the same relation. Therefore $$ (A'\delta_A)\la a \ra = A'\quad\text{and}\quad (a'\delta_A)\la a \ra = a', $$
or, explicitly, we have the following rule

\begin{minipage}{0.35\textwidth}
\[\begin{tikzcd}
	\duu\times\duu & {\duu.\Sigma\Delta\duu} & \duu\times\duu
	\arrow["{\dot{u}^* id}", from=1-1, to=1-2]
	\arrow["{\dot{u}^* id}", from=1-2, to=1-3]
	\arrow["id"', curve={height=12pt}, from=1-1, to=1-3]
\end{tikzcd}\]
 \end{minipage}\hfill
 \begin{minipage}{0.61\textwidth}
 \begin{prooftree} \AxiomC{$\Gamma\vdash a:A$}\AxiomC{$\Gamma\vdash a':A'$}\BinaryInfC{$\Gamma\vdash a=a:A \qquad \Gamma \vdash (a'\delta_A)\la a \ra =a':A' $}\end{prooftree}
 \end{minipage}
which we did not know before. Such a rule is an instance of the discussion in \Cref{tonested}. Of course we cannot say the same for the opposite composition, but that is telling all in itself.
\end{rem}

\subsection{Dependent type theories with $\Pi$-types} \label{pity}

\begin{defn}[$\Pi$-types]\label{wpitypes}
A plain dependent type theory {\it with $\Pi$-types} is a pDTT as in \Cref{jdtt} having two additional rules $\Pi$, $\lambda$ such that the diagram below is commutative and the upper square is a pullback.
\[\begin{tikzcd}
	{\mathbb{U}.\Delta\dot{\mathbb{U}}} && {\dot{\mathbb{U}}} \\
	{\mathbb{U}.\Delta{\mathbb{U}}} && {\mathbb{U}} \\
	& {\ctx}
	\arrow["{\Sigma.(\dot{u}\Delta.u)}"', from=1-1, to=2-1]
	\arrow["\lambda"{description}, color={rgb,255:red,167;green,42;blue,42}, from=1-1, to=1-3]
	\arrow["\Pi"{description}, color={rgb,255:red,167;green,42;blue,42}, from=2-1, to=2-3]
	\arrow["\Sigma", from=1-3, to=2-3]
	\arrow["v"{description}, from=2-1, to=3-2]
	\arrow[from=2-3, to=3-2]
\end{tikzcd}\]
\end{defn}

Recall that $v$ is that from \Cref{ciccino}. The rest of this subsection is devoted to showing that the proof theory generated by such a judgemental theory actually meets our intuition for having $\Pi$-types.

\subsubsection{À la Martin-Löf}\label{mlpi}
Having $\Pi$-types in the sense of \cite{MARTINLOF197573} means to implement the following rules,
\vspace{.5em}

 \begin{minipage}{0.5\textwidth}
     \begin{prooftree}\AxiomC{$\Gamma \vdash A\; {\tt Type}$}\AxiomC{$\Gamma.A \vdash B\; {\tt Type}$}\LeftLabel{($\Pi$\textsc{F})}\BinaryInfC{$\Gamma \vdash \Pi_A B \; {\tt Type}$}\end{prooftree}
 \end{minipage}
 \begin{minipage}{0.5\textwidth}
      \begin{prooftree}\LeftLabel{($\Pi$\textsc{I})}\AxiomC{$\Gamma \vdash A\; {\tt Type}$}\AxiomC{$\Gamma.A \vdash b:B$}\BinaryInfC{$\Gamma \vdash \lambda_A b : \Pi_A B$}\end{prooftree}
 \end{minipage}

  \begin{minipage}{0.5\textwidth}
     \begin{prooftree}\AxiomC{$\Gamma \vdash f:\Pi_A B$}\AxiomC{$\Gamma \vdash a:A$}\LeftLabel{($\Pi$\textsc{E})}\BinaryInfC{$\Gamma \vdash f(a):B\la a\ra$}\end{prooftree}
 \end{minipage}
 \begin{minipage}{0.5\textwidth}
      \begin{prooftree}\AxiomC{$\Gamma.A \vdash b:B$}\AxiomC{$\Gamma \vdash a:A$}\LeftLabel{($\Pi$$\beta$)}\BinaryInfC{$\Gamma \vdash (\lambda_A b)(a)=b\la a\ra:B\la a \ra$}\end{prooftree}
 \end{minipage}
\vspace{.5em}

\noindent plus their congruence with definitional equality.
\begin{center}
 \begin{minipage}{0.5\textwidth}
     \begin{prooftree}\AxiomC{$\Gamma \vdash A=A'\; {\tt Type}$}\AxiomC{$\Gamma.A \vdash B=B'\; {\tt Type}$}\LeftLabel{($\Pi$\textsc{F}=)}\BinaryInfC{$\Gamma \vdash \Pi_A B = \Pi_{A'} B' \; {\tt Type}$}\end{prooftree}
 \end{minipage}

 \begin{minipage}{0.5\textwidth}
      \begin{prooftree}\LeftLabel{($\Pi$\textsc{I}=)}\AxiomC{$\Gamma \vdash A=A\; {\tt Type}$}\AxiomC{$\Gamma.A \vdash b=b':B$}\BinaryInfC{$\Gamma \vdash \lambda_A b = \lambda_A b' : \Pi_A B$}\end{prooftree}
 \end{minipage}

  \begin{minipage}{0.5\textwidth}
     \begin{prooftree}\AxiomC{$\Gamma \vdash f=f':\Pi_A B$}\AxiomC{$\Gamma \vdash a=a':A$}\LeftLabel{($\Pi$\textsc{E}=)}\BinaryInfC{$\Gamma \vdash f(a)=f'(a'):B\la a\ra$}\end{prooftree}
 \end{minipage}
\end{center}

\noindent The first two rules are almost evident in the very definition of dependent type theory with $\Pi$-types, while the other rules will be derived by the limit closure of the class of judgements and rules.

\begin{itemize}
    \item[($\Pi$\textsc{F})] Type formation is precisely the rule $(\Pi)$ in the sense of \Cref{rule} and \Cref{dicjud}, indeed $\mathbb{U}.\Delta \mathbb{U}$ classifies precisely the premises of ($\Pi$\textsc{F}).
    \item[($\Pi$\textsc{I})] Similarly, the introduction rule is precisely the rule $(\lambda)$ in the sense of \Cref{rule} and \Cref{dicrules}, where the commutativity of the diagram ensures the correct typing for the term.
\end{itemize}

In order to express the elimination rule, we first need to code its premise, that is the nested judgement
$$ \Gamma\vdash f:\Pi_A B\quad\;\Gamma\vdash a:A. $$
Notice that, because of ($\Pi$F), this is actually silent of two judgements, meaning it should read
$$ \Gamma\vdash A\quad\;\Gamma.A\vdash B\quad\;\Gamma\vdash f:\Pi_A B\quad\;\Gamma\vdash a:A, $$
instead, so that this is really the judgement we need to give a classification of. One can check that $\duu.\Sigma(\uu.\Delta\uu)$ classifies the first, second, and fourth judgement appearing above. Also, we know from \Cref{typedep} that $\duu.\Sigma(\uu.\Delta\uu)\cong\duu.\Sigma\Delta\uu$.
This is an instance of \Cref{quelliso}, and it just expresses the fact that, whenever we have a term $a:A$, we really have its type in our code already.
\[\begin{tikzcd}
	{\duu.\Sigma\Delta\duu} & {\mathbb{U}.\Delta\dot{\mathbb{U}}} && {\dot{\mathbb{U}}} \\
	{\duu.\Sigma\Delta\uu} & {\mathbb{U}.\Delta{\mathbb{U}}} && {\mathbb{U}} & {} \\
	\duu & {\mathbb{U}} & {\dot{\mathbb{U}}} & {\ctx}
	\arrow[from=1-2, to=1-4]
	\arrow["\Delta"', from=3-2, to=3-3]
	\arrow["{\dot{u}}"', from=3-3, to=3-4]
	\arrow["u", from=2-4, to=3-4]
	\arrow[""{name=0, anchor=center, inner sep=0}, from=2-2, to=2-4]
	\arrow[from=2-2, to=3-2]
	\arrow["\Sigma", from=1-4, to=2-4]
	\arrow[from=1-2, to=2-2]
	\arrow[""{name=1, anchor=center, inner sep=0}, "\Sigma"', from=3-1, to=3-2]
	\arrow[from=1-1, to=2-1]
	\arrow[from=2-1, to=3-1]
	\arrow[""{name=2, anchor=center, inner sep=0}, from=2-1, to=2-2]
	\arrow[from=1-1, to=1-2]
	\arrow["\ulcorner"{anchor=center, pos=0.125}, draw=none, from=2-2, to=3-3]
	\arrow["\ulcorner"{anchor=center, pos=0.125}, draw=none, from=2-1, to=1]
	\arrow["\ulcorner"{anchor=center, pos=0.125}, draw=none, from=1-1, to=2]
	\arrow["\ulcorner"{anchor=center, pos=0.125}, draw=none, from=1-2, to=0]
\end{tikzcd}\]
To now introduce the term $f$, we need to perform one more pullback. We attach the diagram above to that in \Cref{wpitypes}. We are entitled to do so because, by hypothesis, the square that $\Pi$ and $\lambda$ fit in has the correct map on its left. For brevity, and since it should not cause much trouble, for the remainder of the proof we call all ``horizontal'' projections $\pi$, and all ``vertical'' ones $\overline{\Sigma}$.
\[\begin{tikzcd}
	{\duu.\Sigma\Delta\duu} && {\mathbb{U}.\Delta\dot{\mathbb{U}}} && {\dot{\mathbb{U}}} \\
	{\duu.\Sigma\Delta\uu} && {\mathbb{U}.\Delta{\mathbb{U}}} && {\mathbb{U}}
	\arrow["{\overline{\Sigma}}", from=1-3, to=2-3]
	\arrow["\lambda"{description}, from=1-3, to=1-5]
	\arrow["\Pi"{description}, from=2-3, to=2-5]
	\arrow["\Sigma", from=1-5, to=2-5]
	\arrow["\ulcorner"{anchor=center, pos=0.125}, draw=none, from=1-3, to=2-5]
	\arrow[""{name=0, anchor=center, inner sep=0}, "\pi"{description}, from=2-1, to=2-3]
	\arrow["\pi"{description}, from=1-1, to=1-3]
	\arrow["{\overline{\Sigma}}", from=1-1, to=2-1]
	\arrow["\ulcorner"{anchor=center, pos=0.125}, draw=none, from=1-1, to=0]
\end{tikzcd}\]
To express the classifier for the whole premise, then, is to compute the pullback against $\Sigma$ of the composition of $\Pi$ and $\pi$ in the lower part of the diagram. Call $\Pi'=\Pi\circ\pi$. The premise of (E) is then classified by $(\duu.\Sigma\Delta\uu)\Sigma.\Pi'\duu$. We can see how it all builds up in the following suggestive writing
$$ (\duu.\Sigma\Delta\uu)\Sigma.\Pi'\duu \qquad (a \, .\, (A\,.\,B))\, .\, f$$
which is fibered over $\Gamma$: though not all of its components are types or terms specifically in context $\Gamma$, every judgement appearing in this nested one is built out of a construction performed entirely in $\Gamma$.

From now on, we will write all $n$-uples as above as traditional $n$-uples, since all pullbacks are subcategories of a product after all.
\[\begin{tikzcd}
	{(\duu.\Sigma\Delta\uu)\Sigma.\Pi'\duu} \\
	& {\duu.\Sigma\Delta\duu} && {\mathbb{U}.\Delta\dot{\mathbb{U}}} && {\dot{\mathbb{U}}} \\
	& {\duu.\Sigma\Delta\uu} && {\mathbb{U}.\Delta{\mathbb{U}}} && {\mathbb{U}}
	\arrow["{\overline{\Sigma}}", from=2-4, to=3-4]
	\arrow["\lambda"{description}, from=2-4, to=2-6]
	\arrow["\Pi"{description}, color={rgb,255:red,214;green,92;blue,92}, from=3-4, to=3-6]
	\arrow["\Sigma", color={rgb,255:red,214;green,92;blue,92}, from=2-6, to=3-6]
	\arrow["\ulcorner"{anchor=center, pos=0.125}, draw=none, from=2-4, to=3-6]
	\arrow[""{name=0, anchor=center, inner sep=0}, "\pi"{description}, color={rgb,255:red,214;green,92;blue,92}, from=3-2, to=3-4]
	\arrow["\pi"{description}, from=2-2, to=2-4]
	\arrow["{\overline{\Sigma}}", from=2-2, to=3-2]
	\arrow["{\overline{\Sigma}}"', curve={height=6pt}, dashed, from=1-1, to=3-2]
	\arrow["\pi"{description}, curve={height=-6pt}, dashed, from=1-1, to=2-6]
	\arrow["\ulcorner"{anchor=center, pos=0.125}, draw=none, from=2-2, to=0]
\end{tikzcd}\]
We now have two pullbacks insisting on the same cospan, then necessarily it is
\begin{equation}\label{pipb}
    (\duu.\Sigma\Delta\uu)\Sigma.\Pi'\duu\cong \duu.\Sigma\Delta\duu.
\end{equation}
This in {\it not} an instance of \Cref{quelliso}, though, and the isomorphism above actually turns out to contain all the information needed to provide rules (E) and ($\beta$), and then some.

Clearly there is always a map going from right to left, just consider:
$$ (a,b)\mapsto (A,B,\lambda_A b,a), $$
but \cref{pipb} is adding three more pieces of information, meaning
\begin{enumerate}[label=(\roman*)]
    \item there is also a map going from left to right, (though we can always {\it expand} information, only this tells us we can {\it compact} it);
    \item starting from the left, going right, and back left again, yields the identity;
    \item starting from the right, going left, and back right again, yields the identity.
\end{enumerate}
Of these, (i) will induce elimination and (iii) $\beta$-computation. The additional piece in (ii) will tell us something about what is generally called the $\eta$-rule, which is much more controversial. We will discuss it in detail in \Cref{damnueta}.

Call $\zeta$ and $\theta$ the inverse maps. A little calculation shows that they act as follows:
 $$\theta:(a,b)\mapsto (A,B,\lambda_A b,a),\quad \zeta: (A,B,f,a)\mapsto (a,f_B),$$
where we write $f_B$ for the term of type $B$ in the second component of $\zeta$. Broadly speaking, $\theta$ computes introduction (this is evident by $\lambda\circ\pi=\pi\circ\theta$) and $\zeta$ elimination (both because of its typing and because {\it we say so}).

Before we can provide an explicit representation for the missing rules, we shall be able to account for writings $b\la a\ra$ and $B\la a\ra$. In order to do that, we need to use the diagram in \Cref{typedep}. We paste it to the previous one as follows, calling
$$\gamma = \pi \circ u^* id \quad\text{and}\quad \dot{\gamma} = \pi \circ \dot{u}^* id.$$
Notice that the map $\overline{\Sigma}:\duu.\Sigma\Delta\duu\to\duu.\Sigma\Delta\uu$ is precisely that appearing in \Cref{typedep} and \Cref{ctxextdavero}, so that both ``rectangles'' insist on the same functor. All solid squares are pullbacks, the dashed one is only commutative. 
\[\begin{tikzcd}
	{(\duu.\Sigma\Delta\uu)\Sigma.\Pi'\duu} \\
	& {\duu.\Sigma\Delta\duu} && {\mathbb{U}.\Delta\dot{\mathbb{U}}} && {\dot{\mathbb{U}}} \\
	&&&& \duu \\
	& {\duu.\Sigma\Delta\uu} && {\mathbb{U}.\Delta{\mathbb{U}}} && {\mathbb{U}} \\
	&&&& \uu
	\arrow["{\overline{\Sigma}}"{pos=0.7}, from=2-4, to=4-4]
	\arrow["\lambda"{description, pos=0.7}, from=2-4, to=2-6]
	\arrow["\Pi"{description, pos=0.7}, from=4-4, to=4-6]
	\arrow["\Sigma"{pos=0.7}, from=2-6, to=4-6]
	\arrow["\pi"{description}, from=4-2, to=4-4]
	\arrow["\pi"{description}, from=2-2, to=2-4]
	\arrow["{\overline{\Sigma}}"{pos=0.7}, from=2-2, to=4-2]
	\arrow["\zeta"{description, pos=0.6}, shift right=1, color={rgb,255:red,214;green,92;blue,92}, from=1-1, to=2-2]
	\arrow["\theta"{description, pos=0.6}, shift right=1, from=2-2, to=1-1]
	\arrow["\pi"{description}, curve={height=-18pt}, from=1-1, to=2-6]
	\arrow["{\overline \Sigma}"{description}, curve={height=6pt}, from=1-1, to=4-2]
	\arrow["{\dot{\gamma}}"{description}, color={rgb,255:red,214;green,92;blue,92}, dashed, from=2-2, to=3-5]
	\arrow[dashed, from=3-5, to=5-5]
	\arrow["\gamma"{description}, dashed, from=4-2, to=5-5]
\end{tikzcd}\]

\begin{itemize}
    \item[($\Pi$\textsc{E})] The functor $\dot{\gamma}\zeta$ is the elimination rule, because to each quadruple it matches a term of the correct type. We call $\dot{\gamma}(a,f_B)=f_B \la a \ra=:f(a)$.
    \item[($\Pi\beta$)] Computation $\beta$ amounts to proving that if we apply introduction, followed by elimination, we kind of get to the point we started from. This is a rule with codomain as in \Cref{complexjce}, therefore we show that identity on $\duu.\Sigma\Delta\duu$ equalizes the following pair of arrows,
\[\begin{tikzcd}
	{\duu.\Sigma\Delta\duu} && {(\duu.\Sigma\Delta\uu)\Sigma.\Pi'\duu} && {\duu.\Sigma\Delta\uu} & \duu
	\arrow["id"{description}, shift right=2, curve={height=6pt}, from=1-1, to=1-5]
	\arrow["{\dot{\gamma}}"{description}, from=1-5, to=1-6]
	\arrow["\zeta", from=1-3, to=1-5]
	\arrow["\theta", from=1-1, to=1-3]
\end{tikzcd}\]
On the upper path is computed $(\lambda_A b )(a)$, on the lower we get $b\la a \ra$. The two paths equalize trivially. The desired rule is then
$$id: \duu.\Sigma\Delta\duu \to \duu.\Sigma\Delta\duu=\mathbb{E}(\dot{\gamma}\zeta\theta,\dot{\gamma}).$$
\end{itemize}

\subsubsection{Of congruence rules involving definitional equality.}\label{congruence}
In our dictionary in \ref{dictionnaire} we decided that definitional equality of types and terms should be interpreted as judgemental equality according to $u$ and $\du$, respectively, hence as identity of objects in the respective \quotmarks{universe} categories. This guarantees that rules ($\Pi$\textsc{F}=), ($\Pi$\textsc{I}=), ($\Pi$\textsc{E}=) are automatically verified. Rule ($\Pi$\textsc{I}=), also known as the \textit{$\xi$-rule}, in particular, is not verified by all models, especially those that are more computationally oriented, such as Kleene realizability or game semantics: we are indeed quite extensional in our spirit, but we believe this is more of a choice that we are making than a constraint of judgemental theories, and that it would be interesting to further develop the theory with different, weaker, but still finite-limit stable interpretations of judgemental equality.

\subsubsection{Of $\eta$ and elimination.}\label{damnueta}

The $\eta$-rule accounts for the need to determine what happens in the case that one wants to apply elimination followed by introduction, and at first it looks exactly as the dual of ($\Pi\beta$C). While it is clear that $\beta$ should prescribe equality of two terms, though, there is actually no agreement on the features $\eta$ should present, so that in the literature we find instances of the resulting computation of $\eta$ as being a {\it conversion} (\ie consisting of a definitional equality), interpreted as an {\it expansion}, or a {\it reduction} (meaning a non-symmetric relation whose reflexive, symmetric, and transitive closure defines the conversion). The virtue of each process, and each of its 2-categorical delivery, is the topic of \cite{seely1986modelling}.

In our framework, \cref{pipb} tells us something about which $\eta$-rule we should be looking at, and in fact we have
\vspace{-0.2cm}

\begin{prooftree}\AxiomC{$\Gamma \vdash f:\Pi_A B$}\LeftLabel{($\Pi\eta$)}\UnaryInfC{$\Gamma \vdash f=\lambda_A(f_B):\Pi_A B$}\end{prooftree}

\noindent which is precisely what $\theta\zeta=id$ says. This is only one of the possible expressions for $\eta$, and it differs from that presented in \cite[p.253]{awodey_2018}, which much more swiftly agrees with the tradition of categories with families. This is because, in a sense, we think the notion of elimination presented there, and in \Cref{mlpi} above, is not the correct one: it really is $\zeta$ performing the elimination, and it really is $f_B$ the term witnessing it. It is not in the computation through $\dot{\gamma}$ that a term of type $\Pi_A B$ turns into a term involving $A,B$. This argument, together with the possibility of excluding the $\eta$ rule entirely, will be made much more clear in \Cref{phity}.

\subsection{Dependent type theories with (extensional) Id-types} \label{idty}

For identity types we need to be able to consider pairs of terms of the same type, this is why we begin by pulling back $\Sigma$ against itself. Call $\pi_1$, $\pi_2$ the corresponding projections and $diag:\duu\to\duu\times\duu$ the unique map such that $\pi_1 \circ diag = id = \pi_2 \circ diag$.

\begin{defn}[Extensional $\mathsf{Id}$-types]\label{widtypes}
A plain dependent type theory {\it with extensional $\mathsf{Id}$-types} is a pDTT as in \Cref{jdtt} having two additional rules $\mathsf{Id}$, $\mathsf{i}$ such that the diagram below is commutative and the upper square is a pullback.
\[\begin{tikzcd}
	\duu && {\dot{\mathbb{U}}} \\
	\duu\times\duu && {\mathbb{U}} \\
	& {\ctx}
	\arrow["\Sigma", from=1-3, to=2-3]
	\arrow[from=2-3, to=3-2]
	\arrow["{\mathsf{Id}}"{description}, color={rgb,255:red,214;green,92;blue,92}, from=2-1, to=2-3]
	\arrow[from=2-1, to=3-2]
	\arrow["diag"', from=1-1, to=2-1]
	\arrow["{\mathsf{i}}"{description}, color={rgb,255:red,214;green,92;blue,92}, from=1-1, to=1-3]
\end{tikzcd}\]
\end{defn}

Again, the rest of the subsection is dedicated to showing that the proof theory generated by such judgemental theory actually meets our intuition for having $\mathsf{Id}$-types. Classically, having $\mathsf{Id}$-types means to implement the following rules

 \begin{minipage}{0.59\textwidth}
     \begin{prooftree}\AxiomC{$\Gamma \vdash A\;{\tt Type} $}\AxiomC{$\Gamma \vdash a:A$}\AxiomC{$\Gamma\vdash b:A$}\LeftLabel{($\mathsf{Id}$\textsc{F})}\TrinaryInfC{$\Gamma \vdash \mathsf{Id}_A(a,b) \;{\tt Type}$}\end{prooftree}
 \end{minipage}\hfill
 \begin{minipage}{0.39\textwidth}
      \begin{prooftree}\AxiomC{$\Gamma \vdash a:A$}\LeftLabel{($\mathsf{Id}$\textsc{I})}\UnaryInfC{$\Gamma \vdash \mathsf{i}(a):\mathsf{Id}_A(a,a)$}\end{prooftree}
 \end{minipage}

\hfill
 \begin{minipage}{0.4\textwidth}
     \begin{prooftree}\AxiomC{$\Gamma \vdash c:\mathsf{Id}_A(a,b)$}\LeftLabel{($\mathsf{Id}$\textsc{E})}\UnaryInfC{$\Gamma \vdash a=b:A$}\end{prooftree}
 \end{minipage}
 \begin{minipage}{0.4\textwidth}
      \begin{prooftree}\AxiomC{$\Gamma \vdash c:\mathsf{Id}_A(a,b)$}\LeftLabel{($\mathsf{Id}\eta$)}\UnaryInfC{$\Gamma \vdash c=\mathsf{i}(a):\mathsf{Id}_A(a,a)$}\end{prooftree}
 \end{minipage}
 \begin{minipage}{0.08\textwidth}
  \phantom{ci}
\end{minipage}
\vspace{.08cm}

\noindent As it was for \Cref{wpitypes}, the first two rules are evident in the very definition of dependent type theory with $\mathsf{Id}$-types.

\begin{itemize}
    \item[($\mathsf{Id}$\textsc{F})] Type formation is precisely the rule $(\mathsf{Id})$ in the sense of \Cref{rule} and \Cref{dicjud}. Clearly $\duu\times\duu$ classifies the premises of ($\mathsf{Id}$\textsc{F}).
    \item[($\mathsf{Id}$\textsc{I})] Similarly, the introduction rule is the rule $(\mathsf{i})$ in the sense of \Cref{rule} and \Cref{dicrules}, where the commutativity of the diagram forces the correct typing for the term.
\end{itemize}
For elimination and conversion we need to pin-point a classifier for judgements of the form
$$ \Gamma \vdash c: \mathsf{Id}_A (a,b),$$
but since the square is a pullback insisting on the cospan $(\mathsf{Id},\Sigma)$, such a feat is achieved by the (upper-left) $\duu$. Then not only do $\mathsf{Id},\mathsf{i}$ compute the appropriate term and type (below on the left), but they also act as projections (below on the right).
\[\begin{tikzcd}
	&& \duu \\
	\duu\times\duu && \duu && \duu
	\arrow["{\mathsf{i}}"{description}, from=2-3, to=2-5]
	\arrow["\pi"{description}, from=1-3, to=2-5]
	\arrow["\sim", from=1-3, to=2-3]
	\arrow["diag"{description}, from=2-3, to=2-1]
	\arrow["\pi"{description}, from=1-3, to=2-1]
	\arrow["\psi"', from=1-3, to=2-3]
\end{tikzcd}\]
\[\begin{tikzcd}
	a && {\mathsf{i}(a)} & {(a,b,c)} && c \\
	{(a,a)} && {\mathsf{Id}_A(a,a)} & {(a,b)} && {\mathsf{Id}_A(a,b)}
	\arrow["\Sigma"{description}, maps to, from=1-3, to=2-3]
	\arrow["{\mathsf{Id}}"{description}, maps to, from=2-1, to=2-3]
	\arrow["diag"{description}, color={rgb,255:red,214;green,92;blue,92}, maps to, from=1-1, to=2-1]
	\arrow["{\mathsf{i}}"{description}, color={rgb,255:red,214;green,92;blue,92}, maps to, from=1-1, to=1-3]
	\arrow["\pi"{description}, color={rgb,255:red,214;green,92;blue,92}, maps to, from=1-4, to=1-6]
	\arrow["\pi"{description}, color={rgb,255:red,214;green,92;blue,92}, maps to, from=1-4, to=2-4]
	\arrow["{\mathsf{Id}}"{description}, maps to, from=2-4, to=2-6]
	\arrow["\Sigma"{description}, maps to, from=1-6, to=2-6]
\end{tikzcd}\]
The object classifying judgements of the form $\Gamma \vdash a=b:A$, instead, is the equalizer $\mathbb{E}(\pi_1,\pi_2)$. By its universal property there must be a unique $\phi$ making the following diagram commute.
\[\begin{tikzcd}
	{\mathbb{E}(\pi_1,\pi_2)} & \duu && \duu \\
	& \duu\times\duu && \uu
	\arrow["\Sigma", from=1-4, to=2-4]
	\arrow["diag", from=1-2, to=2-2]
	\arrow["{\mathsf{i}}"{description}, from=1-2, to=1-4]
	\arrow["{\mathsf{Id}}"{description}, from=2-2, to=2-4]
	\arrow["\phi"', dashed, from=1-2, to=1-1]
	\arrow["e"', from=1-1, to=2-2]
\end{tikzcd}\]
\begin{itemize}
    \item[($\mathsf{Id}$\textsc{E})] The elimination rule, then, is
    $ \phi: \duu \to \mathbb{E}(\pi_1,\pi_2).$
    \item[($\mathsf{Id}\eta$)] The computation rule is computed as
    $$ \duu \to \mathbb{E}(\pi,\psi\mathsf{i}\pi_1 e \phi)=\mathbb{E}(\pi,\pi)=\duu$$
    therefore it is the map $id:\duu\to\duu$.
\end{itemize}

There would be a notion of $\beta$-computation (in the sense of introduction followed by elimination) here, too, but it is not usually written because it is trivial once one has definitional equality. In fact, it takes the following form.

\begin{minipage}{0.4\textwidth}
     \begin{prooftree}\AxiomC{$\Gamma \vdash a:A$}\LeftLabel{($\mathsf{Id}\beta1$)}\UnaryInfC{$\Gamma \vdash a=a:A$}\end{prooftree}
 \end{minipage}
 \begin{minipage}{0.5\textwidth}
      \begin{prooftree}\AxiomC{$\Gamma \vdash a:A$}\LeftLabel{($\mathsf{Id}\beta2$)}\UnaryInfC{$\Gamma \vdash \mathsf{i}(a)=\mathsf{i}(a):\mathsf{Id}_A(a,a)$}\end{prooftree}
 \end{minipage}\hfill
\vspace{.08cm}

\subsection{A categorical definition of extensional type constructor} \label{phity}

\begin{defn}[The extensional type constructor $\Phi$]\label{wphitypes}
A plain dependent type theory {\it with extensional $\Phi$-types} is a pDTT as in \Cref{jdtt} having two additional rules $\Phi$, $\Psi$ such that the diagram below is commutative and the upper square is a pullback.
\[\begin{tikzcd}
	{\mathbb{X}} && {\dot{\mathbb{U}}} \\
	{\mathbb{Y}} & {} & {\mathbb{U}} \\
	& {\ctx}
	\arrow["\Sigma", from=1-3, to=2-3]
	\arrow[from=2-3, to=3-2]
	\arrow["\Phi"{description}, color={rgb,255:red,214;green,92;blue,92}, from=2-1, to=2-3]
	\arrow[from=2-1, to=3-2]
	\arrow["\Lambda"', from=1-1, to=2-1]
	\arrow["\Psi"{description}, color={rgb,255:red,214;green,92;blue,92}, from=1-1, to=1-3]
\end{tikzcd}\]
\end{defn}

\begin{rem}
Notice that the definition is implicitly assuming that $\Lambda$ belongs to the closure of the generators under finite limits. Also, it is evident by the previous sections that $\Pi$-types and $\mathsf{Id}$-types fall under this definition.
\end{rem}

The rest of the subsection is devoted to showing that the proof theory generated by such judgemental theory actually meets our intuition for having extensional $\Phi$-types.

\vspace{.03cm}
 \begin{minipage}{0.5\textwidth}
     \begin{prooftree}\AxiomC{$\Gamma \vdash Y\; \mathbb{Y}$}\LeftLabel{($\Phi$\textsc{F})}\UnaryInfC{$\Gamma \vdash \Phi Y \; {\tt Type}$}\end{prooftree}
 \end{minipage}
 \begin{minipage}{0.5\textwidth}
     \begin{prooftree}\AxiomC{$\Gamma \vdash X\; \mathbb{X}$}\LeftLabel{($\Phi$\textsc{I})}\UnaryInfC{$\Gamma \vdash \Psi X : \Phi\Lambda X$}\end{prooftree}
 \end{minipage}
\vspace{.03cm}
\begin{itemize}
    \item[($\Phi$\textsc{F})] Type formation is precisely the rule $(\Phi)$ in the sense of \Cref{rule} and \Cref{dicjud}.
    \item[($\Phi$\textsc{I})] Similarly, the introduction rule is precisely the rule $(\Psi)$ in the sense of \Cref{rule} and \Cref{dicrules}, where the commutativity of the diagram forces the correct typing for the term.
\end{itemize}
Now, because we have requested that the square in \Cref{wphitypes} is a pullback, we automatically get the dashed functors below.
\[\begin{tikzcd}
	{\mathbb{Y}\Sigma.\Phi\dot{\mathbb{U}}} &&& {\mathbb{X}} \\
	& {\mathbb{X}} & {\dot{\mathbb{U}}} && {\mathbb{Y}\Sigma.\Phi\dot{\mathbb{U}}} & {\dot{\mathbb{U}}} \\
	& {\mathbb{Y}} & {\mathbb{U}} && {\mathbb{Y}} & {\mathbb{U}}
	\arrow["\Sigma", from=2-6, to=3-6]
	\arrow["\Phi"{description}, from=3-5, to=3-6]
	\arrow["\Lambda"', curve={height=18pt}, from=1-4, to=3-5]
	\arrow["\Psi"{description}, curve={height=-18pt}, from=1-4, to=2-6]
	\arrow[from=2-5, to=2-6]
	\arrow[from=2-5, to=3-5]
	\arrow["{\Lambda \star \Psi}"{description}, dashed, from=1-4, to=2-5]
	\arrow["\Lambda"', from=2-2, to=3-2]
	\arrow["\Psi"{description}, from=2-2, to=2-3]
	\arrow["\Sigma", from=2-3, to=3-3]
	\arrow["\Phi"{description}, from=3-2, to=3-3]
	\arrow[curve={height=-12pt}, from=1-1, to=2-3]
	\arrow[curve={height=12pt}, from=1-1, to=3-2]
	\arrow["{(-)\langle-\rangle}"{description}, dashed, from=1-1, to=2-2]
\end{tikzcd}\]

\begin{itemize}
    \item[($\Phi$\textsc{E})]  The rule associated to functor $(-)\langle-\rangle$ gives us the elimination rule on the right. Indeed the pullback category precisely classifies the premises of ($\Phi$\textsc{E}).
     \begin{prooftree}\AxiomC{$\Gamma \vdash a: \Phi Y$}\LeftLabel{($\Phi$\textsc{E})}\UnaryInfC{$\Gamma \vdash \Phi Y\langle a \rangle \; \mathbb{X}$}\end{prooftree}

\end{itemize}

\noindent  By essential uniqueness of pullbacks, the compositions $((-)\langle-\rangle) \circ (\Lambda \star \Psi)$ and $ (\Lambda \star \Psi) \circ( (-)\langle-\rangle)$ both amount to the identity of the respective object. This observation provided by universal property of the equalizer induces the arrows $\eta$ and $\beta$ in the diagram below.

\[\begin{tikzcd}
	{\mathbb{X}} &&&&&& {\mathbb{Y}\Sigma.\Phi\dot{\mathbb{U}}} \\
	{\mathbb{E}} && {\mathbb{X}} && {\dot{\mathbb{U}}} && {\mathbb{E}} & {\mathbb{Y}\Sigma.\Phi\dot{\mathbb{U}}} && {\dot{\mathbb{U}}}
	\arrow["\Psi", shift left=1, from=2-3, to=2-5]
	\arrow["{\Psi\circ((\Lambda-)\langle \Psi-\rangle)}"', shift right=1, from=2-3, to=2-5]
	\arrow[from=2-1, to=2-3]
	\arrow["id"{description}, from=1-1, to=2-3]
	\arrow["\beta"', dashed, from=1-1, to=2-1]
	\arrow["{\Phi.\Sigma}", shift left=1, from=2-8, to=2-10]
	\arrow["{\Phi.\Sigma\circ\Lambda.\Psi \circ((-)\langle -\rangle)}"', shift right=1, from=2-8, to=2-10]
	\arrow[from=2-7, to=2-8]
	\arrow["id"{description}, from=1-7, to=2-8]
	\arrow["\eta"', dashed, from=1-7, to=2-7]
\end{tikzcd}\]

\begin{itemize}

\item[($\Phi\beta$)]  The rule associated to the functor $\beta$ is our $\beta$-computation. Indeed, if we write down the rule explicitly we get the following.
     \begin{prooftree}\AxiomC{$\Gamma \vdash X \; \mathbb{X}$}\LeftLabel{($\Phi\beta$)}\UnaryInfC{$\Gamma \vdash  \Psi X = \Psi ( (\Lambda X )\langle \Psi X \rangle) :  \Phi \Lambda X$}\end{prooftree}

\item[($\Phi\eta$)]  The rule associated to the functor $\eta$ is our $\eta$-computation. Indeed, if we write down the rule explicitly we get the following.
     \begin{prooftree}\AxiomC{$\Gamma \vdash a: \Phi Y$}\LeftLabel{($\Phi\eta$)}\UnaryInfC{$\Gamma \vdash a = \Psi(\Phi Y \langle a \rangle) : \Phi Y$}\end{prooftree}
\end{itemize}

Additionally, and as in the case of dependent products in \Cref{congruence}, we have rules guaranteeing that definitional equality of terms and types is \quotmarks{preserved} through formation, introduction, and elimination. See thereof for a discussion on possible variations.

\begin{rem}[Weaker notions of type constructors]\label{nonsolopb}
Our definition of type constructor is very modular: for example, if we request that the square in \Cref{wphitypes} is a weak pullback (as opposed to a pullback) with a distinguished section, we can still construct the functors $(-)\langle - \rangle$ and $\Lambda \star \Psi$, and one of the two compositions still amounts to the identity. This ensures both elimination and $\beta$-computation, while we lose $\eta$-computation. This remark generalizes a similar analysis contained in \cite[Cor. 2.5]{awodey_2018}.
\end{rem}

We believe that \Cref{wphitypes} is more proof of both the computational and the expressive power of judgemental theories. We now use the construction above to enrich a pDTT with units and dependent sums. We reverse engineer the theory in order to provide the correct definition, and that will be all that we need because of the calculations above. By the end of this paper, we will have shown that \Cref{wphitypes} captures dependent products, dependent sums, unit types, extensional identity types. In addition, the construction in \cite[$\S$2.4]{awodey_2018} might suggest that it fits intensional identity, too, but we do not discuss this further here.

\begin{rem}[Other type constructors]
We are indeed aware that \Cref{wphitypes} does not capture {\it all} type constructors used in both the theory and the practice of type theory, for example it does not allow for the description of (co)inductive types, but we believe that our categorical theory of judgement has been proved fruitful in coding syntactic data. Clearly finite limits will capture finite constructions, but 2-category theory is much more than finite, nor it is only about limits, therefore we trust that with some effort this work could be extended to different constructors.
\end{rem}

\subsection{Examples: unit types and $\Sigma$-types}
\subsubsection{Dependent type theories with unit types}
Our aim is to describe the premises of introduction and formation, and the relation they are in. Recall that the rules in question are
\begin{minipage}{0.49\textwidth}
	\begin{prooftree}\AxiomC{$\vdash \Gamma \;{\tt ctx}$}\LeftLabel{(uI)}\UnaryInfC{$\Gamma\vdash \mathsf{1}_{\Gamma}\;{\tt Type}$}\end{prooftree}
\end{minipage}
\begin{minipage}{0.49\textwidth}
	\begin{prooftree}\AxiomC{$\vdash \Gamma \;{\tt ctx}$}\LeftLabel{(uF)}\UnaryInfC{$\Gamma\vdash \ast_\Gamma :  \mathsf{1}_{\Gamma}$}\end{prooftree}
\end{minipage}
\vspace{.08cm}

\noindent so that both only take in input a context, and the premises are identical, hence our motivation to give the following definition.

\begin{defn}[Unit-types]\label{unit}
	A plain dependent type theory {\it with unit-types} is a pDTT having two additional functors $\mathsf{1}$, $\ast$ such that the diagram below is commutative and the upper square is a pullback.
	\[\begin{tikzcd}
		\ctx && \duu \\
		\ctx && \uu \\
		& {\ctx}
		\arrow["\id"', from=1-1, to=2-1]
		\arrow["\ast"{description}, color={rgb,255:red,167;green,42;blue,42}, from=1-1, to=1-3]
		\arrow["{\mathsf{1}}"{description}, color={rgb,255:red,167;green,42;blue,42}, from=2-1, to=2-3]
		\arrow["\Sigma", from=1-3, to=2-3]
		\arrow["\id"', from=2-1, to=3-2]
		\arrow["u", from=2-3, to=3-2]
	\end{tikzcd}\]
\end{defn}
We now show that the judgemental theory generated by diagrams in \Cref{unit} contains codes for formation, introduction, elimination, and computation of unit types. Introduction and formation in fact read as follows

\begin{minipage}{0.49\textwidth}
	\begin{prooftree}\AxiomC{$\Gamma\vdash \Gamma \;\id$}\LeftLabel{($\mathsf{1}$)}\UnaryInfC{$\Gamma\vdash \mathsf{1}_{\Gamma}\; u$}\end{prooftree}
\end{minipage}
\begin{minipage}{0.49\textwidth}
	\begin{prooftree}\AxiomC{$\Gamma\vdash \Gamma \;\id$}\LeftLabel{($\ast$)}\UnaryInfC{$\Gamma\vdash (\ast_\Gamma, \mathsf{1}_{\Gamma})\; \dot{u}$}\end{prooftree}
\end{minipage}
\vspace{.09cm}

\noindent or, in our more familiar writing

\begin{minipage}{0.49\textwidth}
	\begin{prooftree}\AxiomC{$\vdash \Gamma \;{\tt ctx}$}\LeftLabel{(uI)}\UnaryInfC{$\Gamma\vdash \mathsf{1}_{\Gamma}\;{\tt Type}$}\end{prooftree}
\end{minipage}
\begin{minipage}{0.49\textwidth}
	\begin{prooftree}\AxiomC{$\vdash \Gamma \;{\tt ctx}$}\LeftLabel{(uF)}\UnaryInfC{$\Gamma\vdash \ast_\Gamma :  \mathsf{1}_{\Gamma}$}\end{prooftree}
\end{minipage}
\vspace{.08cm}

\noindent moreover, the elimination rule is captured by the unique map $\psi\colon \ctx.\mathsf{1}\duu \to \ctx$ and it translates to the syntactic writing on the left, while postcomposed with $\ast$ it translates as the more familiar rule on the right

\begin{minipage}{0.49\textwidth}
	\begin{prooftree}\AxiomC{$\Gamma \vdash t:\mathsf{1}_\Gamma$}\LeftLabel{($\psi$)}\UnaryInfC{$\vdash \Gamma \;{\tt ctx}$}\end{prooftree}
\end{minipage}
\begin{minipage}{0.49\textwidth}
	\begin{prooftree}\AxiomC{$\Gamma \vdash t:\mathsf{1}_\Gamma$}\LeftLabel{($\ast \circ \psi$)}\UnaryInfC{$\Gamma\vdash \ast_\Gamma : \mathsf{1}_{\Gamma}$}\end{prooftree}
\end{minipage}
\vspace{.08cm}

\noindent which is also denoted (uE). Finally, computation $\beta$ and $\eta$ can be decoded from the two following diagrams
\[\begin{tikzcd}
	\ctx &&&&& {\ctx.\mathsf{1}\duu} \\
	{\mathsf{Eq}} && \ctx && \duu & {\mathsf{Eq}} & {\ctx.\mathsf{1}\duu} && \duu
	\arrow["\ast", shift left=1, from=2-3, to=2-5]
	\arrow["{\ast \circ \psi\phi}"', shift right=1, from=2-3, to=2-5]
	\arrow[from=2-1, to=2-3]
	\arrow["\id"{description}, from=1-1, to=2-3]
	\arrow["\beta"', dashed, from=1-1, to=2-1]
	\arrow["{\mathsf{1}.\Sigma}", shift left=1, from=2-7, to=2-9]
	\arrow["{\mathsf{1}.\Sigma \circ\phi\psi}"', shift right=1, from=2-7, to=2-9]
	\arrow[from=2-6, to=2-7]
	\arrow["\id"{description}, from=1-6, to=2-7]
	\arrow["\eta"', dashed, from=1-6, to=2-6]
\end{tikzcd}\]
with $\phi$ the inverse to $\psi$, which read, respectively, as follows.

\begin{minipage}{0.49\textwidth}
	\begin{prooftree}\AxiomC{$\vdash \Gamma \;{\tt ctx}$}\LeftLabel{(u$\beta$)}\UnaryInfC{$\Gamma\vdash \ast_\Gamma =_{\mathsf{1}_\Gamma} \ast_\Gamma$}\end{prooftree}
\end{minipage}
\begin{minipage}{0.49\textwidth}
	\begin{prooftree}\AxiomC{$\Gamma \vdash t:\mathsf{1}_\Gamma$}\LeftLabel{(u$\eta$)}\UnaryInfC{$\Gamma\vdash t=_{\mathsf{1}_{\Gamma}}\ast_\Gamma $}\end{prooftree}
\end{minipage}

\subsubsection{Dependent type theories with $\Sigma$-types}\label{dttwithsigma}
We hope the reader will forgive us if to avoid confusion we adopt the unusual notation of $\rbag$ instead of $\Sigma$. We then start to look at rules for formation and introduction, which for sum types are usually the following.
\vspace{-0.3cm}

     \begin{prooftree}\AxiomC{$\Gamma \vdash A\; {\tt Type}$}\AxiomC{$\Gamma.A \vdash B\; {\tt Type}$}\LeftLabel{($\rbag$\textsc{F})}\BinaryInfC{$\Gamma \vdash \rbag_A B \; {\tt Type}$}\end{prooftree}

     \begin{prooftree}\AxiomC{$\Gamma \vdash A\; {\tt Type}$}\AxiomC{$\Gamma.A \vdash B\; {\tt Type}$}\AxiomC{$\Gamma\vdash a:A$}\AxiomC{$\Gamma\vdash b:B \la a \ra$}\LeftLabel{($\rbag$\textsc{I})}\QuaternaryInfC{$\Gamma \vdash \langle a,b\rangle:\rbag_A B$}\end{prooftree}

\vspace{.03cm}

\noindent In order to classify the premise of ($\rbag$F) we simply use $\uu.\Delta\uu$ from \Cref{typedep}. The premise of ($\rbag$I), instead, can be coded via the following nested judgement classifier
\[\begin{tikzcd}
	{(\duu.\Sigma\Delta\uu)\Sigma.\gamma\duu} &&& \duu \\
	{\duu.\Sigma\Delta\uu} & {\mathbb{U}.\Delta{\mathbb{U}}} && {\mathbb{U}} & {} \\
	\duu & {\mathbb{U}} & {\dot{\mathbb{U}}} & {\ctx}
	\arrow["\Delta"', from=3-2, to=3-3]
	\arrow["{\dot{u}}"', from=3-3, to=3-4]
	\arrow["u", from=2-4, to=3-4]
	\arrow[from=2-2, to=2-4]
	\arrow[from=2-2, to=3-2]
	\arrow[""{name=0, anchor=center, inner sep=0}, "\Sigma"', from=3-1, to=3-2]
	\arrow[from=2-1, to=3-1]
	\arrow[from=2-1, to=2-2]
	\arrow["\ulcorner"{anchor=center, pos=0.125}, draw=none, from=2-2, to=3-3]
	\arrow[""{name=1, anchor=center, inner sep=0}, "\gamma", color={rgb,255:red,214;green,92;blue,92}, curve={height=-24pt}, from=2-1, to=2-4]
	\arrow["\Sigma", color={rgb,255:red,214;green,92;blue,92}, from=1-4, to=2-4]
	\arrow[dashed, from=1-1, to=2-1]
	\arrow[curve={height=-24pt}, dashed, from=1-1, to=1-4]
	\arrow["\ulcorner"{anchor=center, pos=0.125}, draw=none, from=1-1, to=2-2]
	\arrow["\ulcorner"{anchor=center, pos=0.125}, draw=none, from=2-1, to=0]
	\arrow[shorten <=1pt, shorten >=1pt, Rightarrow, from=1, to=2-2]
\end{tikzcd}\]
with $\gamma = \pi \circ u^* id$ from \Cref{typedep}. The desired rule $\Lambda$, then, is the functor $(\duu.\Sigma\Delta\uu)\Sigma.\gamma\duu\to \uu.\Delta\uu $ appearing above.

\begin{defn}\label{wsumtypes}
A plain dependent type theory {\it with $\rbag$-types} is a pDTT as in \Cref{jdtt} having two additional rules $\rbag$, $\mathsf{pair}$ such that the diagram below is commutative and the upper square is a pullback.
\[\begin{tikzcd}
	{(\duu.\Sigma\Delta\uu)\Sigma.\gamma\duu} && {\dot{\mathbb{U}}} \\
	{\uu.\Delta\uu} & {} & {\mathbb{U}} \\
	& {\ctx}
	\arrow["\Sigma", from=1-3, to=2-3]
	\arrow[from=2-3, to=3-2]
    \arrow["\rbag"{description}, color={rgb,255:red,214;green,92;blue,92}, from=2-1, to=2-3]
	\arrow[from=2-1, to=3-2]
	\arrow["{\Sigma.(u.\dot{u}\Delta)\circ\Sigma.\gamma}"', from=1-1, to=2-1]
	\arrow["{\mathsf{pair}}"{description, pos=0.4}, color={rgb,255:red,214;green,92;blue,92}, from=1-1, to=1-3]
\end{tikzcd}\]
\end{defn}

\noindent A pDTT with $\rbag$-types immediately has formation and introduction (with $\langle a,b \rangle = \mathsf{pair}(a,b)$) and, as follows from the content of \Cref{wphitypes}, the three (admittedly hard to look at) rules below. We write $\Lambda$ for $\Sigma.(u.\dot{u}\Delta)\circ\Sigma.\gamma$.

\vspace{.03cm}

     \begin{prooftree}\AxiomC{$\Gamma \vdash c: \rbag_A B$}\LeftLabel{($\rbag$\textsc{E})}\UnaryInfC{$\Gamma \vdash \big(\rbag_A B \big)\langle c \rangle \; (\duu.\Sigma\Delta\uu)\Sigma.\gamma\duu$}\end{prooftree}

      \begin{prooftree}\AxiomC{$\Gamma \vdash c: \rbag_A B$}\LeftLabel{($\rbag\eta$)}\UnaryInfC{$\Gamma \vdash c = \mathsf{pair}(\big(\rbag_A B \big)\langle c \rangle ) : \rbag_A B$}\end{prooftree}

  \begin{prooftree}\AxiomC{$\Gamma \vdash X \; (\duu.\Sigma\Delta\uu)\Sigma.\gamma\duu$}\LeftLabel{($\rbag\beta$)}\UnaryInfC{$\Gamma \vdash  \mathsf{pair} (X) = \mathsf{pair} (\;( \Lambda\, X )\langle \mathsf{pair} (X) \rangle\;) : \rbag \Lambda X$}\end{prooftree}

If we break down the job of the classifier $(\duu.\Sigma\Delta\uu)\Sigma.\gamma\duu$ we recover the familiar following ones.




\vspace{.03cm}
 \begin{minipage}{0.5\textwidth}
     \begin{prooftree}\AxiomC{$\Gamma \vdash c: \rbag_A B$}\LeftLabel{($\rbag$\textsc{E1})}\UnaryInfC{$\Gamma \vdash \pi_1 c:A$}\end{prooftree}
 \end{minipage}
  \begin{minipage}{0.5\textwidth}
     \begin{prooftree}\AxiomC{$\Gamma \vdash c: \rbag_A B$}\LeftLabel{($\rbag$\textsc{E2})}\UnaryInfC{$\Gamma \vdash \pi_2 c : B[\pi_1 c]   $}\end{prooftree}
 \end{minipage}

      \begin{prooftree}\AxiomC{$\Gamma \vdash c: \rbag_A B$}\LeftLabel{($\rbag\eta$)}\UnaryInfC{$\Gamma \vdash c = \mathsf{pair}( \pi_1 c, \pi_2 c ) : \rbag_A B$}\end{prooftree}

\begin{minipage}{0.5\textwidth}
     \begin{prooftree}\AxiomC{$\Gamma\vdash a:A \quad \Gamma\vdash b:B \la a \ra$}\LeftLabel{($\rbag\beta1$)}\UnaryInfC{$\Gamma \vdash a = \pi_1 (\mathsf{pair}(a,b)):A$}\end{prooftree}
 \end{minipage}
  \begin{minipage}{0.5\textwidth}
     \begin{prooftree}\AxiomC{$\Gamma\vdash a:A \quad \Gamma\vdash b:B \la a \ra$}\LeftLabel{($\rbag\beta2$)}\UnaryInfC{$\Gamma \vdash b = \pi_2 (\mathsf{pair}(a,b)):B\la a \ra$}\end{prooftree}
 \end{minipage}

\section{First-order logic} \label{secnat}
In this section we design the judgemental theory that performs the calculus of natural deduction. As for \Cref{secdtt}, we introduce the basic judgements and rules and show how they generate the desired structure, then we add more rules to perform additional computations. Though we follow the path of the well-known fibrational approach to first order logic, we spend some time in re-developing it in the context of judgemental theories: this is meant to present the benefits of the judgemental approach, to compare the resulting structure with that of dependent types, and to give a pedagogical example of how one might want to implement a judgemental theory starting from notions which are known to be fibrational in nature.

\begin{disclaim}[Why we do not start from dependent type theory]\label{thisdischere}
As we discussed in \Cref{calculusproof}, one could very well follow \cite{MARTINLOF197573} and use \Cref{secdtt} as a starting point for this analysis by simply restricting it to the proof-irrelevant case. This is what is really happening in \Cref{jgen} - and an explicit construction is actually provided in \Cref{rem:nd_as_tt} - but we choose to recover the whole theory from scratch for two reasons: on one hand, we hope that it makes the present work accessible to the non-(type theorist), or to someone who is more familiar with traditional first-order logic; on the other we aim to more swiftly align to the tradition of doctrines \cite{equalityinhd,Pitts1983AnAO,Makkai1993TheFF,mr_quotientcompl13}.

Again, as explained in \ref{calculusproof}, our distinction is mathematically artificial, and we will remark that throughout our discussion, see for example \ref{dtycut}.
\end{disclaim}

\begin{disclaim}[Why we do not do Gentzen's sequent calculus]
On the other hand, we could have chosen to present first-order logic in the formalism of sequent calculus in \cite{Gentzen1935UntersuchungenD}. Though our framework allows us for it -- and in fact many of the categorical constructions in the following section do so, already, starting from \Cref{jgen} --  we have chosen to take the perspective of natural deduction because on one hand we believe that, being closer to how logic is used makes it easier to follow what each categorical operation is doing and, secondly, dealing with connective and quantifiers with pairs of introduction/elimination rules, as opposed to right/left introduction rules, helps to keep the connection with dependent types (\ref{thisdischere}) in the back of the reader's mind.
\end{disclaim}

\begin{defn}[Natural deduction theory]\label{jgen}
A {\it natural deduction theory} is a substitutional (\ref{substitutional}) judgemental theory $(\ctx, \classof{J}, \classof{R}, \classof{P})$
 such that
\begin{itemize}
    \item $\ctx$ is $\catof{Fin}$, the category of finite sets;
    \item $\classof{J}$ can be presented by \textit{one} judgement classifier $p: \mathbb{P} \to \ctx$, which is a faithful fibration, has fibered products and implication, and has fibered initial objects.
\end{itemize}
We think of $\ctx$ as the category of {\it variables and terms} and of $\ff$ as the category of {\it well-formed formulae} fibered over variables. We call this NDT for short.
\end{defn}

\begin{rem}[On cardinality]
We can define $\lambda$-ary theories but we would need to close judgemental theories under $\lambda$-small limits, and we would have to replace $\mathsf{Fin}$ with the category of $\lambda$-small sets.
\end{rem}

\begin{rem}[From doctrines to classifiers] \label{doctrines}
Let \hbox{$P: \ctx\opp \to \catof{Pos}$} be a \textit{doctrine}, intended in the most non-committal sense. Consider $\square: P^I \to P$ any operational property/structure on $P$, e.g.:
\begin{itemize}
    \item having (finite) fibered meets $\wedge: P^I \to P$;
    \item having (finite) fibered joins $\vee: P^I \to P$;
    \item having a negation operator $\neg : P \to P$.
\end{itemize}
then, by the Grothendieck construction, we obtain some corresponding diagram of fibrations,
\vspace{-0.3cm}
\[\begin{tikzcd}
	{\ff^I} && {\ff} \\
	\\
	& {\ctx}
	\arrow["{p^I}"{description}, from=1-1, to=3-2]
	\arrow["p"{description}, from=1-3, to=3-2]
	\arrow["\square"{description}, from=1-1, to=1-3]
\end{tikzcd}\]

\noindent This produces a pre-judgemental theory obtained by $\ff$, together with all its structural operators. For example, if $P$ in an Heyting algebra fiber-wise, we have operators $\bot, \top, \wedge, \vee, \Rightarrow$ of the proper arities on $\ff$.
\end{rem}

\begin{rem}[The arrow category] \label{adessoalvolo}
Consider a NDT. Because it is closed under finite powers (\Cref{powers}) we can compute
\[\begin{tikzcd}
	{\ff\due} && {\ff}
	\arrow[from=1-1, to=1-3]
\end{tikzcd}\]
and we will show how the operations defined on $\ff$ lift to (a suitable subcategory of) $\ff\due$ thanks to the closure under finite limits. We will come back to this in \Cref{ciaone}.
\end{rem}

\begin{rem}[Weakening]\label{weak}
Since $p$ has fibered products we can compute the following nested judgement (on the left)

\hfill\begin{minipage}{0.48\textwidth}
\[\begin{tikzcd}
	{\ff^\times} && \ff \\
	{\ctx^2} && \ctx
	\arrow["p", from=1-3, to=2-3]
	\arrow["{-\times -}"{description}, from=2-1, to=2-3]
	\arrow[dashed, from=1-1, to=2-1]
	\arrow[dashed, from=1-1, to=1-3]
	\arrow["\lrcorner"{anchor=center, pos=0.125}, draw=none, from=1-1, to=2-3]
\end{tikzcd}\]
 \end{minipage}
 \begin{minipage}{0.42\textwidth}
\[\begin{tikzcd}
	{\ctx^2} && \ctx \\
	{\ctx^2} && {\ctx^2}
	\arrow[""{name=0, anchor=center, inner sep=0}, "{-\times -}"{description}, curve={height=-12pt}, from=1-1, to=1-3]
	\arrow[""{name=1, anchor=center, inner sep=0}, "diag"{description}, curve={height=-12pt}, from=1-3, to=1-1]
	\arrow[""{name=2, anchor=center, inner sep=0}, "{diag_{-\times -}}"{description}, shift left=4, from=2-1, to=2-3]
	\arrow[""{name=3, anchor=center, inner sep=0}, "id"{description}, shift right=4, from=2-1, to=2-3]
	\arrow["\dashv"{anchor=center, rotate=90}, draw=none, from=1, to=0]
	\arrow["\epsilon", shorten <=2pt, shorten >=2pt, Rightarrow, from=2, to=3]
\end{tikzcd}\]
 \end{minipage}
with an adjunction $diag \dashv - \times -$ whose counit computes projections. These are well known in the literature and perform what is usually called {\it weakening}:
$$ x\times y \vdash x \;\ctx\,.$$

\noindent We can now define a span $(p^2, p^\times)$ out of $\ff^2$ (as a category, not as the fibration $\ff \times \ff$) with $p^\times:(\phi,\psi)\mapsto (\phi[\pr_1]\land \psi[\pr_2])$ the product of the respective cartesian lifts of $\phi,\psi$ along $\pr_1,\pr_2$,
\[\begin{tikzcd}
	& {\phi[\pr_1]\land\psi[\pr_2]} \\
	\phi & {p\phi\times p\psi} & \psi \\
	{p \phi} && {p \psi}
	\arrow["{\pr_1}"{description}, from=2-2, to=3-1]
	\arrow["{\pr_2}"{description}, from=2-2, to=3-3]
	\arrow[dashed, from=1-2, to=2-1]
	\arrow[dashed, from=1-2, to=2-3]
\end{tikzcd}\]
and such a span makes the diagram involving $\ff^\times$ commute, therefore we have a unique rule
$$ w: \ff^2 \to \ff^\times$$
over $\ctx^2$. If $p(\phi,\psi)=(x,y)$ we might denote $w(\phi,\psi)=w_y \phi \land w_x \psi$.
\end{rem}

\begin{defn}[NDT with weakening] 
A NDT is said to {\it have weakening} if for each $y\in\ctx$, $-\times y$ is in $\classof{J}$.

\end{defn}




%
%

In this section we will show that, in fact, a NDT produces the calculus of natural deduction. 

\subsection{Dictionary}\label{dict_gentzen}
As we did in \Cref{secdtt}, we declare a local dictionary, both to make the paper more comprehensible and to account for classical notation.

\begin{notat}[Stratified contexts] \label{doublectx}
Notice that already in \Cref{weak} we follow the intuition and use $x,y,\dots$ to name objects of $\ctx$. In fact, we here want to give a way to present judgements that are traditionally of the form
$$ x;\Gamma \vdash \psi$$
so that they read as having two contexts: the free variables in the formulae {\it and} the formula(e) in the premise of the sequent. In fact, we will ``stack up'' two fibrations so that the objects living on top ($\Gamma\Rightarrow\psi$) are both fibered on those in the middle ($\Gamma$) and those on the bottom ($x$). We hope to make it all clearer in the table that will follow.
\end{notat}

\begin{constr}[Entailment] \label{ciaone}
We wish to represent entailment between two formulae in the same context. In order to do that we pick in $\ff\due$ (see \Cref{adessoalvolo}) all objects belonging to the same fiber. Call $I:\ctx\to\ctx\due$ the functor mapping $\Gamma\mapsto id_\Gamma$ and compute the following (dashed) limit.
\[\begin{tikzcd}[ampersand replacement=\&]
	{\ff\due I.p\due \ctx} \&\& \ctx \\
	{\ff\due} \&\& {\ctx\due} \\
	\ff \&\& \ctx
	\arrow["p"{description}, from=3-1, to=3-3]
	\arrow["{\dom}", shift left=1, from=2-1, to=3-1]
	\arrow["{\cod}"', shift right=1, from=2-1, to=3-1]
	\arrow["{\cod}"', shift right=1, from=2-3, to=3-3]
	\arrow["{\dom}", shift left=1, from=2-3, to=3-3]
	\arrow["I", from=1-3, to=2-3]
	\arrow[dashed, from=1-1, to=2-1]
	\arrow[""{name=0, anchor=center, inner sep=0}, "{p\due}"{description}, from=2-1, to=2-3]
	\arrow[dashed, from=1-1, to=1-3]
	\arrow["\ulcorner"{anchor=center, pos=0.125}, draw=none, from=1-1, to=0]
\end{tikzcd}\]
Both the pullback and all universal arrows belong to the judgemental theory. The classifier we are interested in is the composition $ \ff\due I.p\due \ctx \to \ctx $, and will simply denote it with $e:\ee\to\ctx$.
\end{constr}

\begin{rem}[Natural deduction as a type theory]\label{rem:nd_as_tt}
One can check that $\ee$ and $\ff$ fit into a plain dependent type theory as follows
\[\begin{tikzcd}[ampersand replacement=\&]
	\ee \&\& \ff \\
	\& {\mathsf{ctx}}
	\arrow["e"', from=1-1, to=2-2]
	\arrow["p", from=1-3, to=2-2]
	\arrow[""{name=0, anchor=center, inner sep=0}, "\Sigma"{description}, from=1-1, to=1-3]
	\arrow[""{name=1, anchor=center, inner sep=0}, "\Delta"{description}, curve={height=12pt}, from=1-3, to=1-1]
	\arrow["\dashv"{anchor=center, rotate=90}, draw=none, from=0, to=1]
\end{tikzcd}\]
with $\Sigma:(!,\phi_\Gamma,\psi)\mapsto \psi$ and $\Delta:\phi\mapsto (\id,\phi,\phi)$. They clearly form an adjoint pair, with $\Sigma$ cartesian. The counit is the identity, while the unit at each entailment is the entailment itself, hence both have cartesian components.

When no connectives nor quantifiers are involved, then, one can see the case for first order logic as a particular instance of dependent type theory with faithful type fibration.
\end{rem}

\begin{rem}
The $\Gamma$ appearing in \Cref{doublectx} indicates a finite set of formulae in context $x$. We can see it as a product in $\ff$ and, when we want to do so, we will write $\phi_\Gamma$.
\end{rem}

We are finally ready to declare our local dictionary according to the notation of \Cref{judgement}, and that is the following.

\begin{center}
\begin{tabular}{ c | c }
\hline
  $ x \vdash e \; \ee \quad x\vdash(\dom\circ I.p\due)(e) =_{\ff} \phi_\Gamma \quad x\vdash(\cod\circ I.p\due)(e) =_{\ff} \psi$& $ x;\Gamma \vdash \psi $\\
  $ x \vdash e \; \ee \quad x\vdash(\dom\circ I.p\due)(e) =_{\ff} \phi_\Gamma\land\phi \quad x\vdash(\cod\circ I.p\due)(e) =_{\ff} \psi$& $ x;\Gamma,\phi \vdash \psi $\\
\hline
\end{tabular}
\end{center}


\begin{rem}\label{sameold}
Consider that in the case that $(\; x;\bot\vdash \phi \; \ff \quad\text{and}\quad x \vdash \phi\to\psi \; \ee\;)$ then $x \vdash \phi \Rightarrow \psi \; \ff,$
therefore our framework accounts for the classical correspondence for all $x,\phi,\psi$,
$$x\vdash \phi \Rightarrow \psi \quad\text{iff}\quad x;\phi \vdash \psi.  $$
\end{rem}

\begin{rem}
Since each $p$-fiber is thin, there is at most one $e\in\ee$ between each pair of objects $(\phi,\psi)\in\ff\times\ff$.
\end{rem}

\begin{notat}
In order to make our calculations more readable, we pin-point a specific notation for  $\cod,\dom$ in the case that they follow the inclusion of $\ee$ into $\ff\due$. We creatively write $c$ and $d$, respectively.
\end{notat}

\subsection{From properties to rules}\label{proptorules}
Before we begin our analysis of rules of natural deduction, we show how certain properties lift from $\ff$ (the category) to $\ee$ (the judgement classifier). These will be instrumental in building up rules from $p$. In a sense, this subsection shows how to turn {\it internal properties of $p$} into {\it external rules about $p$}, which is precisely what we did for contexts in \Cref{toytt}.

\begin{rem}[The domain-codomain policy] \label{triv}
Since we will frequently use either $c$ or $d$ to select the consequent or the antecedent of a sequent, it will be useful to have a way to relate the two. The proof of \Cref{trans} is clear evidence in this sense. There is a trivial policy
\[\begin{tikzcd}
	{\ff\due} && {\ff\due} \\
	\\
	& \ff
	\arrow[""{name=0, anchor=center, inner sep=0}, "{\dom}"{description}, from=1-1, to=3-2]
	\arrow["{\cod}"{description}, from=1-3, to=3-2]
	\arrow["\id"{description}, from=1-1, to=1-3]
	\arrow["\alpha"{description}, shorten <=8pt, shorten >=8pt, Rightarrow, from=0, to=1-3]
\end{tikzcd}\]
where $\alpha_{\psi\to\phi} = (\psi\to\phi)$. Note that if useful we might bravely invert the direction of $\id$. We call $\alpha$, too, the obvious whiskering $d\Rightarrow c$.
\end{rem}

\begin{lem}[A special instance of cut] \label{trans}
The relation captured by $e:\ee\to\ctx$ is transitive in the sense that the rule below is in the NDT.
 \begin{prooftree}\AxiomC{$x;\psi\vdash \phi$}\AxiomC{$x;\phi\vdash \chi$}\LeftLabel{(T)}\BinaryInfC{$x;\psi \vdash \chi$}\end{prooftree}

\begin{proof}
Consider the following $\sharp$-lifting of the triangle in \Cref{triv} along $d$.
\[\begin{tikzcd}
	{\ee d.c\ee} &&& \ee & {(x;\psi\vdash\phi,\,x;\phi\vdash\chi)}  \\
	& {\ee d.d \ee} &&& {(x;\psi\vdash\phi,\,x;\psi\vdash\chi)} \\
	\ee &&& \ff &  \\
	& \ee
	\arrow["d"{description}, from=4-2, to=3-4]
	\arrow[""{name=0, anchor=center, inner sep=0}, "c"{description}, from=3-1, to=3-4]
	\arrow["\id"', curve={height=6pt}, from=3-1, to=4-2]
	\arrow["{d.d}"{description}, from=2-2, to=1-4]
	\arrow[""{name=1, anchor=center, inner sep=0}, "{c.d}"{description}, from=1-1, to=1-4]
	\arrow["{d^* \id}"'{pos=0.6}, curve={height=6pt}, dashed, from=1-1, to=2-2]
	\arrow["d", from=1-4, to=3-4]
	\arrow[from=2-2, to=4-2]
	\arrow[from=1-1, to=3-1]
	\arrow[shift right=1, dashed, maps to, from=1-5, to=2-5]
	\arrow["\alpha"'{pos=0.7}, shorten <=2pt, shorten >=4pt, Rightarrow, from=4-2, to=0]
	\arrow["{d^*\alpha}"{pos=0.4}, shorten <=2pt, shorten >=4pt, Rightarrow, dashed, from=2-2, to=1]
\end{tikzcd}\]
A little computation shows that the upper triangle reads as on the right, producing the desired rule
$$ t:=d.d \circ d^* \id : \ee d.c\ee \to \ee.$$
\end{proof}
\end{lem}

\begin{lem}[Preservation through product as a rule] \label{prodpres}
Interaction of arrows and products in $\ff$ (the category) lifts to $\ee$ (the judgement classifier) in the sense that the rule below is in the NDT. \begin{prooftree}\AxiomC{$x;\psi\vdash \phi$}\AxiomC{$x;\bot\vdash \chi$}\LeftLabel{(F)}\BinaryInfC{$x;\psi\land\chi \vdash \phi\land\chi$}\end{prooftree}

\begin{proof}
It is coded by a functor $f:\ee\times\ff\to\ee$ which to pairs $(\psi\to\phi,\chi)$ over some $x$ assigns the unique map $\psi\land\chi \to \phi\land\chi$ defined via the universal property of the product in the fiber over $x$.
\end{proof}
\end{lem}

\subsection{Formal structural rules} \label{structrules}
Here we show that an NDT generates the following formal structural rules.

\begin{minipage}{0.38\textwidth}
     \begin{prooftree}\AxiomC{}\LeftLabel{(H)}\UnaryInfC{$x;\Gamma,\phi \vdash \phi$}\end{prooftree}
 \end{minipage}
 \begin{minipage}{0.3\textwidth}
     \begin{prooftree}\AxiomC{$x;\Gamma,\Delta \vdash \phi$}\LeftLabel{(Sw)}\UnaryInfC{$x;\Delta,\Gamma \vdash \phi$}\end{prooftree}
 \end{minipage}
 \begin{minipage}{0.28\textwidth}
     \begin{prooftree}\AxiomC{$x;\Gamma,\psi,\psi \vdash \phi$}\LeftLabel{(C)}\UnaryInfC{$x;\Gamma,\psi \vdash \phi$}\end{prooftree}
 \end{minipage}\hfill

 \hspace{1.3cm}
 \begin{minipage}{0.3\textwidth}
     \begin{prooftree}\AxiomC{$x;\Gamma \vdash \phi$}\LeftLabel{(W)}\UnaryInfC{$x;\Gamma,\psi \vdash \phi$}\end{prooftree}
 \end{minipage}
 \begin{minipage}{0.5\textwidth}
     \begin{prooftree}\AxiomC{$x;\Gamma \vdash \phi$}\AxiomC{$x;\Gamma,\phi \vdash \psi$}\LeftLabel{(Cut)}\BinaryInfC{$x;\Gamma \vdash \psi$}\end{prooftree}
 \end{minipage}
\vspace{.08cm}

\noindent We break the discussion into three parts.

\subsubsection{ Hypothesis and the simple fibration}\label{ruleH}
Clearly for each pair $(\Gamma,\phi)$ over the same context, we have that $\pr_2: \phi_{\Gamma}\land\phi\to \phi$, therefore $\ee$ classifies $x;\Gamma,\phi\vdash\phi$.
\begin{itemize}
    \item[(H)] The Hypothesis rule, then, is coded into the existence of $e$ itself.
\end{itemize}
We would be content with this already, but it is worth noticing that the association performing the projection $\pr_2$
$$ (\Gamma,\phi) \mapsto (\phi_{\Gamma}\land \phi \to \phi)$$
can be described functorially, and it contains some profound information. Such functor, in fact, provides an insight into possible developments of the present work, plus it (almost) allows for a presentation of the {\it simple fibration} from \cite{jacobs1999categorical}, which has a meaningful logical interpretation: it constitutes the ``least informative'' type theory one can observe over a category with finite products. Therefore we say a little more about that.


\begin{defn}[The simple fibration]\label{simplefib}
Define on the category $\ee$ the monad comprised of the following data:
\begin{itemize}
    \item the functor $S:\ee \to \ee$ acting as follows
    $$ \psi \to \phi \quad\mapsto\quad \psi\land\phi \to \psi\to \phi;$$
    \item the $2$-cell $\eta:Id\Rightarrow S$ defined via the universal property of products;
    \item the $2$-cell $\mu: S\circ S \Rightarrow S$ acting as $(\pr_1,\id)$.
\end{itemize}
\end{defn}
\begin{rem}
$(S,\eta,\mu)$ is idempotent. This is because $\mu$ acts as follows
\[\begin{tikzcd}
	{(\psi\land\phi)\land\phi} & \psi\land\phi & \psi & \phi \\
	\psi\land\phi && \psi & \phi
	\arrow["{\pr_1}"{description}, dashed, from=1-1, to=2-1]
	\arrow[from=1-1, to=1-2]
	\arrow[from=1-2, to=1-3]
	\arrow[from=1-3, to=1-4]
	\arrow["\id"{description}, dashed, from=1-4, to=2-4]
	\arrow[from=2-1, to=2-3]
	\arrow[from=2-3, to=2-4]
\end{tikzcd}\]
and $(\psi\land\phi)\land\phi=\psi\land\phi$ because $p$ is thin and its products are fibered, and in fact the forgetful functor from algebras over $S$ into $\ee$ is fully faithful. All $S$-algebras are free.
\end{rem}
The Kleisli category of $S$ is equivalent to (the total category) of what in \cite{jacobs1999categorical} is called the {\it simple fibration} associated to $p$. That is $\overline{p}:s(\ff)\to\ff$ where $s(\ff)$ has for objects pairs $(\phi,\phi')$ in the same $p$-fiber and maps $a=(a_1,a_2):(\phi,\phi')\to(\psi,\psi')$ such that $a_1:\phi\to\psi$, $a_2:\phi\land\phi'\to\psi'$, and $p(a_1)=p(a_2)$. The functor $\overline{p}$ acts as the first projection. If we call $1:\ctx\to\ff$ the (fibered) terminal object functor, one checks that $ \overline{p}.1 \cong p$. Moreover, the functor
$$ q:\;(\phi,\phi')\;\mapsto\; (\phi\land\phi'\to\phi)$$
induces a comprehension category (as in \Cref{tojacobs}) $s(\ff)\to\ff^{\to}$. The type theory associated to such a functor is (that equivalent to) untyped lambda calculus.
\noindent The functor $S$ induces the following rule.

 \begin{minipage}{0.4\textwidth}
\[\begin{tikzcd}
	{\ee} && {\ee} \\
	& {\ff} \\
	& {\ctx}
	\arrow["p"{description}, from=2-2, to=3-2]
	\arrow["{c}"{description}, curve={height=6pt}, from=1-1, to=2-2]
	\arrow["c"{description}, curve={height=-6pt}, from=1-3, to=2-2]
	\arrow["{S}"{description}, from=1-1, to=1-3]
\end{tikzcd}\]
 \end{minipage}
 \begin{minipage}{0.4\textwidth}
     \begin{prooftree}\AxiomC{$x;\Gamma\vdash\phi$}\doubleLine\LeftLabel{Dictionary in \ref{dict_gentzen}}\UnaryInfC{$x\; \vdash \phi_{\Gamma}\to\phi\; \ee$}\LeftLabel{(S)}\UnaryInfC{$x \vdash S(\phi_{\Gamma}\to\phi)  \; \ee$}\doubleLine\LeftLabel{\Cref{simplefib}}\UnaryInfC{$x \vdash \phi_{\Gamma}\land\phi\to \phi  \;\ee$}
     \doubleLine\LeftLabel{Dictionary in \ref{dict_gentzen}}\UnaryInfC{$x;\Gamma,\phi \vdash \phi $}\end{prooftree}
 \end{minipage}\hfill

\subsubsection{Swap and Contraction: fibered products everywhere}
\begin{itemize}
    \item[(Sw)] The Swap rule holds because the fibered product is symmetric and this too is expressed via a commutative triangle: consider the following composition
    \[\begin{tikzcd}
    	\ff\times\ff & \ff\times\ff && \ff
    	\arrow["\land"{description}, from=1-2, to=1-4]
    	\arrow["\sim", draw=none, from=1-1, to=1-2]
    	\arrow["s"{description}, from=1-1, to=1-2]
    \end{tikzcd}\]
    where the map $s$ computes the permutation. The desired rule is computed as the (iso)morphism
    $$ s.(d.\land): (\ff\times\ff).\ee \to (\ff\times\ff).\ee.$$
    \item[(C)] Contraction is supported by the following dashed map
        \[\begin{tikzcd}
        	{(\ff\times\ff)d.(\land\!\circ \id\times\Delta)\ee} &&&& \ee \\
        	&&& {(\ff\times\ff)d.\!\land\!\ee} \\
        	\ff\times\ff && \ff\times\ff\times\ff && \ff \\
        	&&& \ff\times\ff
        	\arrow["\land"{description}, from=4-4, to=3-5]
        	\arrow["\id"{description}, from=3-1, to=4-4]
        	\arrow["\id\times\Delta"{description}, from=3-1, to=3-3]
        	\arrow["{\id\times \pr_1}"{description}, from=3-3, to=4-4]
        	\arrow["\land"{description, pos=0.6}, from=3-3, to=3-5]
        	\arrow["d", from=1-5, to=3-5]
        	\arrow[from=2-4, to=4-4]
        	\arrow[from=2-4, to=1-5]
        	\arrow["\lrcorner"{anchor=center, pos=0.125}, draw=none, from=2-4, to=3-5]
        	\arrow[from=1-1, to=3-1]
        	\arrow[from=1-1, to=1-5]
        	\arrow["\lrcorner"{anchor=center, pos=0.125}, draw=none, from=1-1, to=3-3]
        	\arrow[dashed, from=1-1, to=2-4]
        \end{tikzcd}\]
    where we write $\land$ for the obvious product $\ff\times\ff\times\ff\to\ff$. On the bottom we have the triangle on the left commuting, and $\id\times\Delta$ equalizing $\land$ and $\id\times \pr_1 \circ \land$. The two nested judgements on the top classify, respectively, the antecedent and the consequent of (C), and the dashed map exists by the universal property of the ``smaller'' pullback.
\end{itemize}

\subsubsection{Weakening and Cut: more transitivity}

We will see that to provide both Weakening and Cut it is sufficient to apply (T) from \Cref{trans} to appropriate triples. Let us start with (W), first: notice that, as it happened in \Cref{secdtt} and is evident from \Cref{dict_gentzen}, the consequent in (W) is actually silent of (at least) one judgement, that is $x;\bot\vdash \psi$. The procedure we follow for (W) is that of
\vspace{-0.5cm}

     \begin{prooftree}\AxiomC{$x;\Gamma \vdash \phi$}\AxiomC{$(\,x;\bot\vdash \psi\,)$}\LeftLabel{(F)}\BinaryInfC{$x;\Gamma,\psi \vdash \phi\land\psi$}

     \AxiomC{$(x;\Gamma \vdash \phi$}\AxiomC{$\,x;\bot\vdash \psi\,)$}\BinaryInfC{
     $x;\phi\land\psi \vdash \phi$}

      \LeftLabel{(T)}\BinaryInfC{$x;\Gamma,\psi \vdash \phi $}
     \end{prooftree}
therefore we need to pre-process the premise of $t$ in order to apply it to triples of the form $(\phi_{\Gamma}\land\psi,\phi\land\psi,\phi)$. This is achieved via the following diagram
\[\begin{tikzcd}
	\ee\times\ff & \ff\times\ff \\
	& {\ee d.c\ee} & \ee \\
	& \ee & \ff
	\arrow["c"{description}, from=3-2, to=3-3]
	\arrow["d"{description}, from=2-3, to=3-3]
	\arrow[from=2-2, to=3-2]
	\arrow[from=2-2, to=2-3]
	\arrow["\lrcorner"{anchor=center, pos=0.125}, draw=none, from=2-2, to=3-3]
	\arrow["f"{description}, curve={height=12pt}, from=1-1, to=3-2]
	\arrow["{c\times \id}", curve={height=-6pt}, from=1-1, to=1-2]
	\arrow[dashed, from=1-1, to=2-2]
	\arrow["{q_1}"{description}, curve={height=-6pt}, from=1-2, to=2-3]
\end{tikzcd}\]
with $q_1$ the map $(\phi,\psi)\mapsto (\phi\land\psi\to \phi)$ acting on pairs in the same $p$-fiber. Notice how this is related to $q$ in \Cref{simplefib}.

\begin{rem}[Cones and branches]\label{lovebranches}
Here branches in the tree of a deduction correspond to cones over limit diagrams. We could make this statement more precise, but we hope the following discussion speaks for itself.
\end{rem}
\begin{itemize}
    \item[(W)] Weakening is computed by the dashed arrow above followed by $t$ from \Cref{trans}.
\end{itemize}
For (Cut) we again apply \Cref{trans}, this time to the triple $(\phi_{\Gamma},\phi_{\Gamma}\land\phi,\psi)$, that is we will build the diagram corresponding to the following composing rules.

 \begin{prooftree}
    \AxiomC{$x;\Gamma \vdash \phi\quad x;\Gamma,\phi \vdash \psi$}
    \UnaryInfC{$x;\Gamma \vdash \phi$}\UnaryInfC{$x;\Gamma \vdash \phi_\Gamma\land\phi$}

    \AxiomC{$x;\Gamma \vdash \phi\quad x;\Gamma,\phi \vdash \psi$}
    \UnaryInfC{$x;\Gamma,\phi \vdash \psi$}\doubleLine\UnaryInfC{
     $x;\phi_\Gamma\land\phi \vdash \psi$}

      \LeftLabel{(T)}\BinaryInfC{$x;\Gamma \vdash \psi $}
     \end{prooftree}

\noindent Therefore we want a map from $\ee d. dS\ee$, classifying the premise of Cut, into $\ee d.c \ee$ so that we then can apply (T). That is achieved as follows.
\[\begin{tikzcd}
	{\ee d. dS\ee} \\
	\ee & {\ee d.c\ee} & \ee \\
	\ee\times\ff & \ee & \ff
	\arrow["c"{description}, from=3-2, to=3-3]
	\arrow["d"{description}, from=2-3, to=3-3]
	\arrow[from=2-2, to=3-2]
	\arrow[from=2-2, to=2-3]
	\arrow["\lrcorner"{anchor=center, pos=0.125}, draw=none, from=2-2, to=3-3]
	\arrow[dashed, from=1-1, to=2-2]
	\arrow["{dS.d}"{description}, curve={height=-12pt}, from=1-1, to=2-3]
	\arrow["{d.dS}"', from=1-1, to=2-1]
	\arrow["{(\id,d)}"', from=2-1, to=3-1]
	\arrow["f"{description}, from=3-1, to=3-2]
\end{tikzcd}\]
\begin{itemize}
    \item[(Cut)] Cut is computed by the dashed arrow above followed by $t$ from \Cref{trans}.
\end{itemize}


\begin{rem}[Cut is a policy]
While perhaps not evident, the Cut rule is in fact a policy in the sense of \Cref{jt}: it just preprocesses data going into the policy (T) from \Cref{trans}.
\[\begin{tikzcd}
	{\ee d. dS\ee} & {\ee d.c\ee} & {\ee d.d \ee} \\
	& \ff
	\arrow[""{name=0, anchor=center, inner sep=0}, from=1-2, to=2-2]
	\arrow[draw={rgb,255:red,214;green,92;blue,92}, from=1-2, to=1-3]
	\arrow[draw={rgb,255:red,214;green,92;blue,92}, from=1-1, to=1-2]
	\arrow[curve={height=-6pt}, from=1-3, to=2-2]
	\arrow[draw={rgb,255:red,214;green,92;blue,92}, shorten >=6pt, Rightarrow, from=1-3, to=0]
\end{tikzcd}\]
\end{rem}

\begin{rem}[Relationship between DTy and Cut] \label{dtycut}
As we have seen, at the very core of Cut sits the policy (T) from \Cref{trans}. The reader will notice the incredible similarity between the process that constructs (T) and the process that constructs (DTy).
\vspace{-0.1cm}
\[\begin{tikzcd}
	{\ee d. dS\ee} & {} & {\ee d.d \ee} & {\dot{\mathbb{U}}.\Delta\Sigma\uu} && {\duu \times \uu} \\
	& \ff &&& \ctx
	\arrow[""{name=0, anchor=center, inner sep=0}, curve={height=6pt}, from=1-1, to=2-2]
	\arrow[from=1-1, to=1-3]
	\arrow[curve={height=-6pt}, from=1-3, to=2-2]
	\arrow[""{name=1, anchor=center, inner sep=0}, curve={height=6pt}, from=1-4, to=2-5]
	\arrow[curve={height=-6pt}, from=1-6, to=2-5]
	\arrow[from=1-4, to=1-6]
	\arrow[shorten >=9pt, Rightarrow, from=1-3, to=0]
	\arrow[shorten >=9pt, Rightarrow, from=1-6, to=1]
\end{tikzcd}\]
Not only do the diagrams in \Cref{trans} and in \Cref{ctxextdavero} look very similar, but even their ingredients have affine logical meaning. Indeed, in both cases the hypothesis of the policy is a nested judgement where a modality appears: in the case of natural deduction, this is the monad $S$, in the case of dependent type theory it is the monad $\Delta\Sigma$. Of course, some delicate differences appear too\footnote{For example the {\it height} at which one performs the action of the monad, that is $\duu$ and $\uu$ are manipulated over contexts while both instances of $\ee$ are bounded to formulae.}. This kind of thoughts could lead to a general notion of \textit{cut}, a special family of policies, but we leave such a task for a possible future work.
\end{rem}

\subsection{Formal rules for connectives}\label{connrules}
We here show that the moment we ask that connectives are closed under $p$-fibers, with $p$ a NDT, we automatically get the expected rules. Since \Cref{jgen} already contains the requirement that $p$ has fibered products, we here show how to provide in a NDT rules for $\land$, and need to ask nothing more of it. If the reader inspects the constructions below, they will see that such a procedure could be repeated for NDTs having $p$ {\it additionally} equipped with $\lor,\neg$.

The ones which are usually required for $\land$ are the following.

\vspace{.03cm}
\begin{minipage}{0.38\textwidth}
     \begin{prooftree}\AxiomC{$x;\Gamma \vdash \phi$}\AxiomC{$x;\Gamma\vdash\psi$}\LeftLabel{($\land$\textsc{I})}\BinaryInfC{$x;\Gamma \vdash \phi\land\psi$}\end{prooftree}
 \end{minipage}
 \begin{minipage}{0.3\textwidth}
     \begin{prooftree}\AxiomC{$x;\Gamma \vdash \phi\land \psi$}\LeftLabel{($\land$\textsc{E1})}\UnaryInfC{$x;\Gamma \vdash \phi$}\end{prooftree}
 \end{minipage}
 \begin{minipage}{0.28\textwidth}
     \begin{prooftree}\AxiomC{$x;\Gamma \vdash \phi\land \psi$}\LeftLabel{($\land$\textsc{E2})}\UnaryInfC{$x;\Gamma \vdash \psi$}\end{prooftree}
 \end{minipage}
\vspace{.01cm}

\begin{itemize}
    \item[($\land$\textsc{I})] Introduction is represented by the functor $conj:\ee d.d\ee\to \ee$ induced by the universal property of the fibered product.
    \item[($\land$\textsc{E1})] In order to represent its domain, we compute the equalizer
\[\begin{tikzcd}
	{\mathbb{E}(d.d,conj)} & {\ee d.d\ee} & \ee
	\arrow["{d.d}", shift left=1, from=1-2, to=1-3]
	\arrow[dashed, from=1-1, to=1-2]
	\arrow["conj"', shift right=1, from=1-2, to=1-3]
\end{tikzcd}\]
    where by $d.d$ (sadly ambiguous, in this case) we wish to express the first projection of the pullback. The desired rule is then induced by the first product projection and it assumes the following form.
    $$ \mathbb{E}(d.d,conj)\to \ee \to \ee$$
    \item[($\land$\textsc{E2})] Dually, we compose the equalizer with the functor induced by the second projection.
\end{itemize}

\begin{defn}[Heyting and Boolean NDTs]\label{fojndt}
A NDT is said to
\begin{itemize}
    \item be {\it Heyting} if we have operators $\bot, \top, \wedge, \vee, \Rightarrow$ of the proper arities on $p$;
    \item be {\it Boolean} if it is Heyting and, being $\neg:= (\text{-}) \Rightarrow \bot $, the morphism of fibrations $\neg\neg$ is equivalent to $\id$.
\end{itemize}
\end{defn}

\subsection{Substitution}\label{termsassub}
While in \Cref{secdtt} we thought of morphisms of contexts as substitutions, in the setting of proof theory we regard them as terms. When we write a map
    $$ y \to x$$
we see it as a list of terms and denote it as such:
    $$ [t/x] : y \to x .$$
In particular, if $x=x_1 \times\dots\times x_k$, each term $t_i = \pr_i \circ t $ is a term built up from $y$ and in context $x_i$, with $i=1,\dots,k$. Then we can identify $\ctx_{/x}$ with the classifier collecting {\it all} terms in context $x$.

All of this belongs to the intuition and in fact there is nothing more to $\ctx$ than what described in \Cref{jgen}, but it is with this perspective that we now look at how substitution behaves in NDTs. Recall from \Cref{syntax_policies} that substitutionality allows us to compute rules and policies as the following

\begin{minipage}{0.63\textwidth}
\[\begin{tikzcd}[ampersand replacement=\&]
	{\ctx\due.\cod\ff} \&\&\& \ff \\
	\& {\ctx\due.\dom\ff} \\
	\ctx\due \&\&\& \ctx \\
	\& \ctx\due
	\arrow[""{name=0, anchor=center, inner sep=0}, "\cod", from=3-1, to=3-4]
	\arrow["\id"', curve={height=6pt}, from=3-1, to=4-2]
	\arrow["{\dom.p}"', from=2-2, to=1-4]
	\arrow[""{name=1, anchor=center, inner sep=0}, "{\cod.p}", from=1-1, to=1-4]
	\arrow["{p^*\id}"'{pos=0.6}, curve={height=6pt}, from=1-1, to=2-2]
	\arrow["p", from=1-4, to=3-4]
	\arrow["\dom"', from=4-2, to=3-4]
	\arrow["\alpha"'{pos=0.6}, shorten <=2pt, shorten >=3pt, Rightarrow, from=4-2, to=0]
	\arrow["{p^*\alpha}"{pos=0.2}, shorten <=2pt, shorten >=4pt, Rightarrow, from=2-2, to=1]
\end{tikzcd}\]
\end{minipage}
\begin{minipage}{0.35\textwidth}
\centering
 \begin{prooftree}\AxiomC{$x \vdash \phi\; \ff$}\UnaryInfC{$y \vdash \phi[t/x] \; \ff$}\end{prooftree}
\end{minipage}
\vspace{.1cm}

\noindent and here we have only blindly expanded the information contained in the diagram on the left by following the discussion in \ref{syntax_policies}. 
%
%

\subsection{Formal rules for quantifiers} \label{quantrules}
Finally, we wish to give an account of quantifiers, hence we introduce more structure on $p$ and on the judgemental theory it generates.

\begin{defn}[First order NDTs]
A NDT is said to
\begin{itemize}
    \item be {\it intuitionistic first order} if it has weakening (\Cref{weak}), is Heyting and $w$ from \Cref{weak} has left and right adjoints,
    $$ \exists \dashv w \dashv \forall,$$
    where $\exists, \forall$ are morphisms of fibrations and belong to $\classof{J}$; we call such theories IcFOTs, for short;
    \item is {\it classical first order} if it is intuitionistic first order and also Boolean; we call these cFOTs for short.
\end{itemize}
\end{defn}

\noindent We believe that the request of being morphisms of fibrations (i.e. preserve cartesian squares) is related to the more traditional properties required for $\forall,\exists$, namely Frobenius reciprocity and Beck-Chevalley.

\begin{rem}
Consider a (intuitionistic) first order theory in the traditional sense.  Then it induces a (I)cFOT: (the fibration associated to) the hyperdoctrine of Lindenbaum-Tarski algebras of well-formed formulae, as for example in \cite{mr_quotientcompl13}.
\end{rem}

We only provide explicit representation of the rules involving $\forall$ in the IcFOT, $\exists$ could be worked out in a similar fashion. First of all, notice that the pair of adjoint functors $w\dashv\forall$ induces (via the hom-set isomorphism) the following rule (on the left)

\begin{minipage}{0.45\textwidth}
     \begin{prooftree}\AxiomC{$x\times y \vdash w(\phi,\psi) \leq \chi \; \ff$}\doubleLine\LeftLabel{(FA)}\UnaryInfC{$(x,y)\vdash (\phi,\psi) \leq \forall\chi \;\ff\times\ff$}\end{prooftree}
 \end{minipage}
 \begin{minipage}{0.55\textwidth}
     \begin{prooftree}\AxiomC{$x\times y \vdash w_y\phi\land w_x\psi \leq \chi \; \ff$}\doubleLine\UnaryInfC{$x \vdash \phi \leq \forall_y\chi \;\ff\qquad y \vdash \psi \leq \forall_x\chi \;\ff$}\end{prooftree}
 \end{minipage}
\vspace{.1cm}

which, if we denote $\forall_y = \pr_1 \circ \forall$ and $\forall_x = \pr_2 \circ \forall$, amounts to the rule on the right. The two rules we need to produce are the following.

\vspace{.03cm}
\hfill
\begin{minipage}{0.45\textwidth}
     \begin{prooftree}\AxiomC{$x\times y; w_y \Gamma \vdash \phi$}\LeftLabel{($\forall$\textsc{I})}\UnaryInfC{$x;\Gamma \vdash \forall_y \phi$}\end{prooftree}
 \end{minipage}
 \begin{minipage}{0.28\textwidth}
     \begin{prooftree}\AxiomC{$x;\Gamma \vdash \forall_y \phi $}\LeftLabel{($\forall$\textsc{E})}\UnaryInfC{$x;\Gamma\vdash \phi[t/y]$}\end{prooftree}
 \end{minipage}
  \begin{minipage}{0.16\textwidth}
    \phantom{y}
 \end{minipage}
\vspace{.1cm}

\noindent Notice that we included the writing $w_y\Gamma$ (with $w_y$ of the kind described in \Cref{weak}) to express the desired dependency, since in this case we wish to say that there is no $y$ free in $\Gamma$. Also, writing $\phi[t/y]$ is a bit improper in the sense that, since $p(\phi)=p(\phi_\Gamma) \times y = x\times y$, each substitution in $\phi$ should have codomain $x \times y$. It is clear what happens here, but we will go into detail when the time comes.

We begin with Introduction. It does actually pretty much read as the fact that $\forall$ is right adjoint to $w$ ``at'' the triple $(\, (\phi_\Gamma,\phi_\Gamma),\,\phi \,)$, but if we wish to write a rule in the sense of \Cref{jt}, we shall start computing the premise, which we do via the following pullback
\[\begin{tikzcd}
	{\ff (Pr_1\,p.\times).p\ff^\times} & {\ff^\times} \\
	& {\ctx^2} \\
	\ff & \ctx
	\arrow["p"{description}, from=3-1, to=3-2]
	\arrow["{Pr_1}", from=2-2, to=3-2]
	\arrow["{p.\times}", from=1-2, to=2-2]
	\arrow[from=1-1, to=3-1]
	\arrow[from=1-1, to=1-2]
	\arrow["\lrcorner"{anchor=center, pos=0.125}, draw=none, from=1-1, to=3-2]
\end{tikzcd}\]
which classifies pairs $(x\vdash \Gamma,\, x\times y \vdash \phi)$. But now we exploit the fact that
$$  w_y\Gamma \leq \phi  \quad\text{iff}\quad w_y\Gamma \land \phi = w_y\Gamma$$
so we ask of the equalizer of the maps
\[\begin{tikzcd}
	{\ff (Pr_1\,p.\times).p\ff^\times} & {\ff^2} && {\ff^2} & \ff
	\arrow[hook, from=1-1, to=1-2]
	\arrow["{p^\times}"', from=1-4, to=1-5]
	\arrow["{(Pr_1,Pr_1)}"', shift right=1, from=1-2, to=1-4]
	\arrow["\id", shift left=1, from=1-2, to=1-4]
\end{tikzcd}\]
with the top one computing $w_y\Gamma \land \phi$ and the bottom one $w_y\Gamma$. We denote $\mathbb{E}(\,p^\times,p^\times (Pr_1,Pr_1)\,)$ with $\mathbb{A}$.


\begin{itemize}
    \item[($\forall$\textsc{I})] The introduction rule is the functor $ \mathbb{A} \to \mathbb{A}$ which follows from the hom-set isomorphism discussed above. Its inverse implies that actually it is the following.
    \begin{prooftree}\AxiomC{$x\times y; w_y \Gamma \vdash \phi$}\doubleLine\UnaryInfC{$x;\Gamma \vdash \forall_y \phi$}\end{prooftree}
\end{itemize}
 With Elimination, we (implicitly) use the isomorphism above and write
$ [t/y]$ for $([x/x],[t/y]):x \to x\times y  $
exploiting $ \ctx_{/x\times y} \cong \ctx_{/x} \times \ctx_{/y}$. Using substitution again as in \Cref{syntax_policies}, we get
\begin{prooftree}\AxiomC{$x\times y; w_y \Gamma \vdash \phi$}\UnaryInfC{$x;(w_y\Gamma)[t/y]\vdash \phi[t/y]$}\end{prooftree}

But recall that $w_y\Gamma =\phi_\Gamma [\pr_1]$, and since
\[
[\pr_1][t/y]: x\to x\times y \to x
\]
is $\id=[x/x]$, given that also $p$ is faithful, we automatically get $(w_y\Gamma)[t/y]=\Gamma$ concluding the proof.

\subsection{Cut elimination}\label{cutelim}
In pointing out necessary features of a judgemental analogue of natural deduction, we see that no instance of Cut is (explicitly) mentioned and, instead, in \Cref{structrules} Cut is shown to {\it automatically} be in the IcFOT generated by $p:\ff\to \ctx$. We regard this as an instance of what in sequent calculus is called ``cut elimination'' (and is shown to be quite hard to prove \cite{Hauptsatz}), or of ``normalization'' in natural deduction (which, in turn, follows almost instantly from admissibility).

In a very precise sense, such rule is a tool that we already have encoded in the theory the moment we require that it satisfies some properties that we deem fundamental. In fact, curiously, the main reason it works is the existence of the domain-codomain policy (\Cref{triv}) and {\it not} (only) composition of arrows. More on this peculiarity was discussed in \Cref{dtycut}.

\section{Ceci n'est pas un topos} \label{topos}
The definitions developed in this work allow for a discussion about the internal logic of a topos, intended in the most unbiased sense. Indeed this section will touch on several variations of the concept:
\begin{itemize}
    \item elementary topoi à la Lawvere-Tierney \cite{lawvere1971quantifiers};
    \item pretopoi and predicative approaches in the spirit of Maietti \cite{maietti2005modular};
    \item $2$-topoi à la Weber \cite{weber2007yoneda} and cosmoi à la Street \cite{street1974elementary,street1980cosmoi}.
\end{itemize}

We will see that all these notions of \textit{topos} support a plain dependent type theory in the sense of \Cref{secdtt}. Such a dtt recovers, among other things, the Mitchell-Bénabou language of the topos and nicely interacts with its Kripke-Joyal semantics. Most importantly, though, our treatment frames the main feature of a topos-like category in a clear way. The discussion is set in such a way that at each step the level of conceptual complexity gets higher and higher. The discussion about predicative foundations, in particular, contains a key point of view to understand our treatment of $2$-topoi, which is an improved version of \cite{weber2007yoneda}.

\subsection{Elementary topoi} \label{elementary}
\begin{disc}[A bit of history]
The internal logic of a topos has been discussed by several authors. After \cite{sheavesingeometry}, this collective humus has been crystallized in the Mitchell-Bénabou language and its \textit{tautological} interpretation, the Kripke-Joyal semantics. These attributions are somewhat symbolic. For what concerns the Mitchell-Bénabou language, the best historical account is given, to our knowledge, by Johnstone \cite{johnstone1977topos}. After Mitchell's original contribution \cite{mitchell1972boolean}, Johnstone refers to the unfindable \cite{coste1972langage} for Bénabou's contribution, but the paper is actually authored by Coste. \cite{osius1975logical} and others were definitely part of the intellectual debate on the topic. For what concerns the Kripke-Joyal semantics the situation is much more cloudy, Osius \cite{osius1975note} tells us that the original ideas from Joyal were never published, while a footprint of Joyal's contribution to the topic only emerges (in French) in \cite{boileau1981logique}. These ideas were later conveyed in several texts with slight variations, like \cite{lambek1988introduction} and \cite{borceux_1994}.
Both in the case of the language and its semantics, we will refer to the presentation in \cite[VI, Sec. 5 and 6]{sheavesingeometry} which is in a sense the most informal and essential. Our main objective is to demonstrate that our formalism can reboot the core ideas behind the Mitchell-Bénabou language. We will not discuss in detail Kripke-Joyal semantics, even though the connection could be drawn, as exemplified by the recent \cite{awodey2021kripke}.
\end{disc}

\begin{defn}[The dtt of an elementary topos] \label{dttelementary}
For an an elementary topos $\ctg{E}$, we can construct a dependent type theory in the sense of \Cref{jdtt} as follows.
\[\begin{tikzcd}
	{\ctg{E}_{/1}} && {\ctg{E}_{/\Omega}} \\
	\\
	& {\ctg{E}}
	\arrow["{\id}"{description}, from=1-1, to=3-2]
	\arrow["\omega"{description}, from=1-3, to=3-2]
	\arrow[""{name=0, anchor=center, inner sep=0}, "{\Sigma_\top}"{description}, from=1-1, to=1-3]
	\arrow[""{name=1, anchor=center, inner sep=0}, "{\Delta_\top}"{description}, curve={height=18pt}, dashed, from=1-3, to=1-1]
	\arrow["\dashv"{anchor=center, rotate=90}, draw=none, from=0, to=1]
\end{tikzcd}\]
The map $\Sigma_\top$ is induced (via precomposition) by the map $\top : 1 \to \Omega$ which picks the top-element of $\Omega$. $\Delta_\top$ is given by pullback, and of course the whole discussion fits perfectly with \Cref{awodeytous}, with the technical advantage that the presheaves in this case are internally represented by objects in the topos, thus there is no need to use the Yoneda embedding.
\vspace{-0.2cm}
\[\begin{tikzcd}
	{\Sigma\Delta \phi} && 1 \\
	\\
	X && \Omega
	\arrow["\top"{description}, from=1-3, to=3-3]
	\arrow["\phi"{description}, from=3-1, to=3-3]
	\arrow[color={rgb,255:red,214;green,92;blue,92}, dashed, from=1-1, to=3-1]
	\arrow[dashed, from=1-1, to=1-3]
\end{tikzcd}\]
\end{defn}

\begin{rem}[Comprehension category, display maps, monomorphisms]
Of course, at this point the whole content of \Cref{secdtt} applies, and thus we can load a whole judgement calculus for this dependent type theory. For example, the rule
\begin{prooftree}\AxiomC{$\Gamma\; \vdash \phi \; \ctg{E}_{/\Omega}$}\LeftLabel{$(\Delta)$}\UnaryInfC{$\Delta \phi \vdash \Delta \phi \; \ctg{E}_{/1}$}\end{prooftree}
is telling us that to each proposition $\phi: X \to \Omega$, corresponds an object $\Delta \phi$, which is precisely the object supporting the subobject of $X$ classified by $\phi$.
Similarly, following \Cref{tojacobs}, we obtain a representation of the internal logic of the topos in terms of a comprehension category, \[\disp: \ctg{E}_{/\Omega} \to \ctg{E}\due.\]
Such correspondence maps a formula $\phi$ to the dashed colored arrow in the construction above. It follows that the correspondence maps a proposition to its zero locus, i.e. the monomorphism whose characteristic function is precisely $\phi$. Of course, this idea is not novel and it dates back to Taylor's PhD thesis or his more recent \cite{taylor1999practical}.
\end{rem}

\begin{rem}[Mitchell-Bénabou reloaded]
Following \cite[VI, Sec. 5]{sheavesingeometry} we see that there is a canonical dictionary between our judgements classified by $\ctg{E}_{/\Omega}$ and \textit{formulae}, i.e. terms of type $\Omega$ in the sense of \cite[pag. 299, right after the bulleted list]{sheavesingeometry}. Moreover, and somewhat most importantly, display maps construct subobjects as zero locus of formulae, as explained in \cite[pag. 300, right after the bulleted list]{sheavesingeometry}.

\begin{center}
\begin{tabular}{ c | c }
\hline
 \phantom{LO}$ X \vdash \phi \; \ctg{E}_{/\Omega}$ \phantom{LO}&\phantom{LO} $ \phi(x) $ \phantom{LO}\\
  \phantom{LO}$ X \vdash \Delta_\top \phi \; \ctg{E}_{/\top}$ \phantom{LO}&\phantom{LO} $ \{x | \phi(x) \}$ \phantom{LO}\\
\hline
\end{tabular}
\end{center}

Notice the difference between $\Delta \phi$ and $\disp_\phi$: even though they might seem to be similar things, the first one gives us the support of the subobject, while the second one gives us the subobject itself.
\[\begin{tikzcd}
	{\{x |\phi(x)\}} && 1 \\
	\\
	X && \Omega
	\arrow["\top"{description}, from=1-3, to=3-3]
	\arrow["\phi"{description}, from=3-1, to=3-3]
	\arrow["{\disp_{\phi}}"{description}, dashed, from=1-1, to=3-1]
	\arrow["\Delta\phi"{description}, dashed, from=1-1, to=1-3]
\end{tikzcd}\]
\end{rem}
As a result of this discussion, one can use the judgement calculus produced by this dependent type theory to simulate the internal logic of the topos, and the result will be consistent with the Mitchell-Bénabou language of the topos. 

Let us give a few examples. Notice that we chose topoi as a very strong theory, but in fact \Cref{toposunit}, \Cref{toposid} show the \emph{modularity} of our approach, in a fashion very much affine to \cite{maietti2005modular}.
\begin{lem}\label{toposunit}
The pDTT induced by a topos $\ctg{E}$ has unit types in the sense of \Cref{unit}.
\end{lem}
\begin{proof}
It suffices to show that we have functors $\ast$ and $\mathsf{1}$ making the following diagram commute and the square a pullback.
	\[\begin{tikzcd}
		\ctg{E} && \ctg{E}_{/\top} \\
		\ctg{E} && \ctg{E}_{/\Omega} \\
		& {\ctg{E}}
		\arrow["\id"', from=1-1, to=2-1]
		\arrow["\ast"{description}, color={rgb,255:red,167;green,42;blue,42}, from=1-1, to=1-3]
		\arrow["{\mathsf{1}}"{description}, color={rgb,255:red,167;green,42;blue,42}, from=2-1, to=2-3]
		\arrow["\Sigma", from=1-3, to=2-3]
		\arrow["\id"', from=2-1, to=3-2]
		\arrow["\omega", from=2-3, to=3-2]
	\end{tikzcd}\]
Let us denote $!_X$ the unique map from $X$ to the terminal -- for the moment, elsewhere we have and we will use $X$ both for the object and the map to $1$. One can easily check that defining $\ast\colon X\mapsto !_X$, and in the obvious way on morphisms, and $\mathsf{1}\colon X\mapsto\top\circ\, !_X$, and in the obvious way on morphisms, does the job.
\end{proof}

\begin{lem}\label{toposid}
The pDTT induced by a topos $\ctg{E}$ has extensional identity types in the sense of \Cref{widtypes}.
\end{lem}
This can be proved in similarly as in \Cref{toposunit}, using equalizers. We take a bit of care in proving the following, instead.
\begin{lem}
The pDTT induced by a topos $\ctg{E}$ has dependent product types in the sense of \Cref{wpitypes}.
\end{lem}
\begin{proof}
It suffices to show that we have functors $\lambda$ and $\Pi$ making the following diagram commute and the square a pullback.
\[\begin{tikzcd}
	{\ctg{E}_{/\Omega}.\Delta\ctg{E}_{/\top}} && {\ctg{E}_{/\top}} \\
	{\ctg{E}_{/\Omega}.\Delta{\ctg{E}_{/\Omega}}} && {\ctg{E}_{/\Omega}} \\
	& {\ctg{E}}
	\arrow["{\Sigma.(\id\Delta.\omega)}"', from=1-1, to=2-1]
	\arrow["\lambda"{description}, color={rgb,255:red,167;green,42;blue,42}, from=1-1, to=1-3]
	\arrow["\Pi"{description}, color={rgb,255:red,167;green,42;blue,42}, from=2-1, to=2-3]
	\arrow["\Sigma", from=1-3, to=2-3]
	\arrow["v"{description}, from=2-1, to=3-2]
	\arrow[from=2-3, to=3-2]
\end{tikzcd}\]
Let us first compute the two categories
\[
\ctg{E}_{/\Omega}.\Delta{\ctg{E}_{/\Omega}}\tand\ctg{E}_{/\Omega}.\Delta\ctg{E}_{/\top}\;.
\]
Following the construction in \Cref{typedep}, we can see that they respectively have objects
\[
(\phi,\psi)\tand(\phi,\{x|\phi(x)\})
\]
with $\phi,\psi$ as below.
\[\begin{tikzcd}[ampersand replacement=\&]
	\Omega \& {\{x|\phi(x)\}} \& 1 \\
	\& X \& \Omega
	\arrow["{\disp_\phi}"', from=1-2, to=2-2]
	\arrow["\phi"', from=2-2, to=2-3]
	\arrow["\top", from=1-3, to=2-3]
	\arrow["\Delta\phi", from=1-2, to=1-3]
	\arrow["\psi"', from=1-2, to=1-1]
	\arrow["\lrcorner"{anchor=center, pos=0.125}, draw=none, from=1-2, to=2-3]
\end{tikzcd}\]
The verical map on the left hand side of the square computes the diagonal of the pullback square above, meaning it acts as $(\phi,\{x|\phi(x)\})\mapsto(\phi,\phi\circ\disp_\phi)$.

To provide suitable $\Pi, \lambda$ we of course look at right adjoints to pullback functors. The fact that they reasonably model dependent products has been widely discussed from the publication of \cite{seely1984locally}, with distinguished treatments in \cite{https://doi.org/10.1002/malq.202000069}, where an explicit construction is given, and in \cite{maietti2005modular}, where it is better framed in the context of the different properties of a topos and their logical counterpart.

One can always show that for a given $\phi\colon X\to \Omega$ (and, in fact, for any $f\colon X\to Y$), we have the following equivalence and adjunction,
\[\begin{tikzcd}[ampersand replacement=\&]
	{\ctg{E}_{/X}\cong(\ctg{E}_{/\Omega})_{/\phi}} \&\& {\ctg{E}_{/\Omega}}
	\arrow[""{name=0, anchor=center, inner sep=0}, "{\Pi_\phi}"', shift right=3, from=1-1, to=1-3]
	\arrow[""{name=1, anchor=center, inner sep=0}, "{\phi^*}"', shift right=3, from=1-3, to=1-1]
	\arrow["\dashv"{anchor=center, rotate=-90}, draw=none, from=1, to=0]
\end{tikzcd}\]
see for example \cite[IV.7]{sheavesingeometry}. Given a pair $(\phi,\psi)$ in $\ctg{E}_{/\Omega}.\Delta{\ctg{E}_{/\Omega}}$, then, it is natural to compute $\disp_\psi$,
\[\begin{tikzcd}[ampersand replacement=\&]
	1 \& {\{x,\phi(x)|\psi(x)\}} \\
	\Omega \& {\{x|\phi(x)\}} \& 1 \\
	\& X \& \Omega
	\arrow["{\disp_\phi}"', from=2-2, to=3-2]
	\arrow["\phi"', from=3-2, to=3-3]
	\arrow["\top", from=2-3, to=3-3]
	\arrow["\Delta\phi", from=2-2, to=2-3]
	\arrow["\psi", from=2-2, to=2-1]
	\arrow["\lrcorner"{anchor=center, pos=0.125}, draw=none, from=2-2, to=3-3]
	\arrow["{\disp_\psi}", from=1-2, to=2-2]
	\arrow["\top"', from=1-1, to=2-1]
	\arrow["\Delta\psi"', from=1-2, to=1-1]
	\arrow["\lrcorner"{anchor=center, pos=0.125, rotate=-90}, draw=none, from=1-2, to=2-1]
\end{tikzcd}\]
and define $\Pi(\phi,\psi)=\Pi_\phi(\disp_\phi\circ\disp_\psi)$. As for $\lambda$, we put $\lambda(\phi,\{x|\phi(x)\})=\{x|\phi(x)\}$.

The square involving $\Pi,\lambda$ commutes because $\{x,\phi(x)|\phi(x)\}=\{x|\phi(x)\}$ hence the composition of displays is mapped to the trivial triangle $\disp_\phi: \phi\circ\disp_\phi\to\phi$. The universal property of $\Pi_\phi$ is what guarantees that the domain of $\Pi(\phi,\psi)$ is, in fact $\{p|\Pi(\phi,\psi)(p)\}$. From this remark, one can immediately show that the desired square is a pullback.
\end{proof}

This is nothing new, but we believe it provides a different perspective on the internal logic of a topos (or any category, really). It ends up being really close to the following intuition.
\hyphenblockcquote{english}{maietti2005modular}{
We can then conclude that describing the internal
dependent type theory of a category means to capture the type-theoretic properties of the
codomain fibration, while describing the internal many-sorted logic of a category – considering the sorts as types – means to capture the properties of the subobject fibration together
with the one-dimensional structure of the category under consideration.
}
In a sense, our work is about extending this process to more than just the codomain fibration.

\subsection{Predicative topoi} \label{predicative}
Under the name of \textit{predicative mathematics} goes a gradient of foundations that, at its extreme, rejects the assumption of function spaces and powersets. In this sense, the category of sets we are used to work with, and on which the whole program of ETCS \cite{lawvere1964elementary,lawvere2005elementary} is built on, is inherently impredicative. As Awodey pointed out in his talk at the CT2021 in Genova \cite{awoct}, this bit of impredicativity is the trade off for a very algebraic approach to set theory, so that its main features can be encoded in few axioms, as those in the definition of elementary topos. Yet, for a sufficiently topos-like predicative foundation, we can still reason in a way that is very similar to the case of an elementary topos, and provide a dtt whose judgement calculus is the internal logic of the \textit{predicative topos}.

\begin{defn}[Virtual object]
A presheaf $P: \ctg{C}\opp \to \catof{Set}$ is virtually representable, or more simply a \textit{virtual object} if it preserves all limits that exist. A subobject $P \to \hirayo c$ of a representable that is a virtual object is called a \textit{virtual subobject} of $c$.
\end{defn}

\begin{rem}[Freyd dust]
Virtual objects will play a crucial role in our definition of predicative topos. Before we give it, though, we feel the need to put a bit of context around our virtual objects. While the name itself, and in a sense the intuition that we have on them, is somewhat original, the general idea has been known to category theorists since forever. If we ignore the solution set condition in the Adjoint Functor Theorem, then the Yoneda embedding yields an equivalence of categories \[\hirayo : \ctg{C} \to \catof{Cont}(\ctg{C}\opp , \catof{Set}).\]
Thus, virtual objects are a kind of \textit{Freyd dust} covering the image of the Yoneda embedding. These presheaves have almost indistinguishable properties with respect to a representable, and  -- up to a size issue -- they are \textit{just} the image of the Yoneda embedding. This intuition sits at the core of the very recent \cite{brandenburg2021large}, and was already used from a technical point of view in \cite[6.4]{Makkaipare}.
\end{rem}

\begin{defn}[Predicative topos]
A predicative topos $\ctg{C}$ is a category with finite limits that
\begin{itemize}
    \item is \textit{virtually} cartesian closed, i.e. $\ctg{C}(- \times b, c)$ is a virtual object for all $b, c$;
    \item has \textit{specification}, i.e. virtual subobjects are representable;
    \item has \textit{virtual} subobject classifier, i.e. the subobject doctrine $\mathrm{Sub}: \ctg{C}\opp \to \catof{Set}$ is a virtual object.
\end{itemize}
\end{defn}

\begin{rem}[Descent, Descent, Descent]
This definition captures a key feature of Grothendieck topoi. Indeed, if one inspects the reason for which a Grothendieck topos has a subobject classifier, one discovers that the exactness properties of the category force the subobject functor to be continuous, thus \textit{descent} implies that $\mathrm{Sub}$ is a virtual object.
Because descent is the defining feature of infinitary pretopoi, their subobject doctrine is a virtual object too. It follows that an infinitary pretopos with specification is a predicative topos too.
If we want these exactness property to be witnessed by an object in the category (i.e. if we want  $\mathrm{Sub}$ to be representable) we trade its existence with predicativity. This very geometric point of view is implicitely claiming that some form of \textit{descent} is the key feature of a topos, which is impredicatively forced in the definition of elementary topos via its subobject classifier. Let us isolate the main observation of this remark in the corollary below.
\end{rem}

\begin{cor}
    An infinitary pretopos with specification is a predicative topos.
\end{cor}

\begin{defn}[The dtt of a predicative topos] \label{dttpredicative}
Let $\ctg{C}$ be a predicative topos. Consider the following pullback diagram in the category of prestacks over $\ctg{C}$,
\[\begin{tikzcd}
	P && 1 \\
	\\
	{\hirayo \Gamma} && {\mathrm{Sub}}
	\arrow["\top", from=1-3, to=3-3]
	\arrow["\phi"{description}, from=3-1, to=3-3]
	\arrow[color={rgb,255:red,214;green,92;blue,92}, dashed, from=1-1, to=3-1]
	\arrow[color={rgb,255:red,214;green,92;blue,92}, dashed, from=1-1, to=1-3]
	\arrow["\ulcorner"{anchor=center, pos=0.125}, draw=none, from=1-1, to=3-3]
\end{tikzcd}\]
Because all the prestacks involved in the cospan are virtual objects, and virtual objects are trivially closed under limits, $P$ is virtual. Since $\top$ is a mono, and monos are pullback stable, $P$ is a virtual subobject, and thus it is represented by assumption via some object $\Gamma.\phi \in \ctg{C}$. It follows as in the proof of \Cref{awodeytous}, that in the diagram below involving the subobject fibration, the functor $\Sigma_\top$ has a right adjoint, which thus provides a plain dtt in our sense.

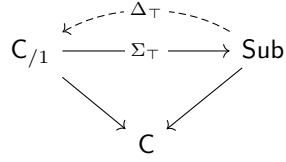
\[\begin{tikzcd}
	{\ctg{C}_{/1}} && {\mathrm{Sub}} \\
	& {\ctg{C}}
	\arrow[from=1-1, to=2-2]
	\arrow[from=1-3, to=2-2]
	\arrow["{\Sigma_\top}"{description}, from=1-1, to=1-3]
	\arrow["{\Delta_\top}"{description}, curve={height=18pt}, dashed, from=1-3, to=1-1]
\end{tikzcd}\]

\end{defn}

\subsection{Elementary $2$-topoi} \label{2topos}

Elementary $2$-topoi were introduce by Weber in \cite{weber2007yoneda}, with Yoneda structures and cosmoi \cite{street1980cosmoi} in mind. The analogy with elementary topoi is exemplified by the prototypical example of elementary $2$-topos.

\begin{exa}[The $2$-topos of categories]
Consider the $2$-category $\catof{Cat}$, with some flexibility on size. To be more precise, $\catof{cat}$ is the $2$-category of (essentially) small categories, $\catof{Cat}$ is the $2$-category of locally small, but possibly large categories, $\catof{CAT}$ is the $2$-category of locally large categories. Then
\vspace{-0.2cm}
\[\begin{tikzcd}
	{\mathrm{Elts}(\phi)} && {\catof{Set}_{\bullet}} \\
	\\
	{\ctg{C}} && {\catof{Set}}
	\arrow[from=1-3, to=3-3]
	\arrow["\phi"{description}, from=3-1, to=3-3]
	\arrow[dashed, from=1-1, to=3-1]
	\arrow[dashed, from=1-1, to=1-3]
	\arrow["\ulcorner"{anchor=center, pos=0.125}, draw=none, from=1-1, to=3-3]
\end{tikzcd}\]

\noindent there is an equivalence of categories - established by the Grothendieck construction - between discrete fibrations over $\ctg{C}$ and copresheaves $\phi$ as in the diagram above. This is telling us that $\catof{Cat}$ has a classifier of discrete fibrations, given by the copresheaf construction. So, in a $2$-topos, discrete opfibrations play the analog of monomorphisms, and their associated prestack is representable.

\[\catof{Fib}_{\text{dsc}}(\ctg{C}) \simeq \catof{Cat}(\ctg{C}, \catof{Set}).\]

\end{exa}

\begin{defn}[Elementary $2$-topos, very similar to {\cite[Def. 4.10]{weber2007yoneda}}]
An elementary $2$-topos is a cartesian closed $2$-category with finite limits and a classifier of discrete fibrations.
\end{defn}

\begin{rem}[Not exactly Weber]
Weber's original definition allows for a more humble notion of classifier, indeed it can be a classifier of \textit{some} discrete fibrations. Also, it is based on the notion of opfibration, but this choice does not lead to any conceptual difference in our treatment. Among the examples, he gives $\catof{Set}_{\bullet, \lambda} \to \catof{Set}_\lambda$ as the classifier of the fibrations with $\lambda$-small fibers.
Our definition, which is in some sense more ambitious but also closer to that of elementary topos rules out all our desired examples.
\begin{itemize}
    \item $\catof{cat}$ has finite limits and is cartesian closed, but it does not have a classifier.
    \item $\catof{Cat}$ has finite limits and a subobject classifier, but is not cartesian closed.
    \item $\catof{CAT}$ does not have the classifier, again.
\end{itemize}

\noindent This is probably the reason behind Weber's flexibility, indeed the classifiers of $\lambda$-\textit{small} fibrations are by many considered a sufficiently expressive alternative that successfully eludes size issues. We do not see it that way. Here we see that we have a problem that is very similar to the predicativity case.
\end{rem}

Luckily, there is a very consistent way to fix Weber's definition of elementary $2$-topos in such a way that all the listed desiderata are indeed examples. The situation is very similar to that of predicativity. Indeed, the prestack of fibrations \[P: \catof{cat}\opp \to \catof{Cat}\]
is a virtual object (because it is classified by hom-ing into $\catof{Set}$), despite not being representable. This is witnessing the fact that $\catof{cat}$ has a $2$-dimensional version of \textit{descent}, and indeed it is a $2$-topos in the sense of Street \cite{street1982characterization}. Of course, on a technical level, it just follows from the fact that such a prestack is almost representable, and thus of course it is a virtual object.

\begin{defn}[The fibration of discrete opfibrations]
Let $\ctg{K}$ be a $2$-category, and consider the prestack mapping an object $k$ to the category of discrete fibrations over $k$, \[ k \mapsto \catof{Fib}_{\text{dsc}}(k).\] Via the Grothendieck construction, this prestack comes with an associated fibration, for which we will use the same name. Moreover, because the identity of $k$ is always a discrete fibration, we can construct the following commutative triangle.

\[\begin{tikzcd}
	{\ctg{K}_{/1}} && {\catof{Fib}_{\text{dsc}}} \\
	& {\ctg{K}}
	\arrow[from=1-1, to=2-2]
	\arrow[from=1-3, to=2-2]
	\arrow["{\Sigma_\top}"{description}, from=1-1, to=1-3]
\end{tikzcd}\]
\end{defn}

\begin{defn}[Similar to {\cite[Def. 4.10]{weber2007yoneda}}] \label{elementary 2top}
An elementary $2$-topos is a $2$-category $\ctg{K}$ that
\begin{enumerate}
    \item has finite $2$-limits,
    \item is cartesian closed,
    \item the prestack of discrete fibrations is a $2$-virtual object, i.e. it preserves all $2$-limits,
    \item $\Sigma_\top$ above has a right adjoint.
\end{enumerate}
\end{defn}

\begin{rem}[The dtt of an elementary $2$-topos]
As in the case of \Cref{dttelementary} and \Cref{dttpredicative} the existence of the right adjoint for $\Sigma_\top$ provides us with a dtt in the sense of \Cref{awodeytous}, expressing the internal logic of the elementary $2$-topos.
\end{rem}
\begin{exa}
Now, let us show that $\catof{cat}$ is an elementary $2$-topos in our sense. Given the discussion above, it is enough to verify the condition (4) in the definition above. In the spirit of \Cref{awodeytous}, this follows from the observation that if $\ctg{C}$ is a small category, the category of elements of a copresheaf $\mathrm{Elts}(\phi)$ is always small, and thus we can construct the right adjoint $\Delta_\top$.
\end{exa}

\section{Future developments}\label{futurejt}

There are two kinds of future developments for this project. To begin with, the new language that we propose allows us to compare, analyze, and design deductive systems.

One one hand, as we have specified in the introduction to this paper, we here only see a couple of possible applications of the framework of judgemental theories, but their versatility suggests many more are possible, for example to modal or linear logic. A taste of the first is already contained in \cite{biequivCoEm}. Moreover, as any other calculus, questions of compactness and normalization arise. We believe trying to answer them would lead to interesting insights into both the logic and the category theory.

On the other hand, in \Cref{dtycut} a well-known link between the cut rule and substitution of terms in expressed in our framework. There we suggested many common features of the two, and a comodality seems to appear. We hope to find more examples of these \emph{cut-like} phenomena, and study their intrinsic properties. Moreover, it feels like our treatment of substitution might intercept some concepts in \cite{MCCUSKER2021106689}, where a calculus of substitution is introduced by means of composition of certain dinatural transformations. This is a relation that we wish to investigate in future work.

In a different direction, the general theory of judgemental theories shows some possible tweaks and adjustments that may lead to a crisper and sharper presentation.

Firstly, the attentive reader might have noticed that the choice of fixing a given category for contexts is a mere formality, and it actually makes the definitions less smooth that we wished, see for example the discussion pertaining \Cref{mainjt}: if anything, this work has convinced us that the notion of \emph{context} in a logical theory is simply a relative one. We believe that this line of thought and work should be explored further. Nevertheless, we decided to keep the exposition closer to classical presentations as not to make an already cryptic theory appear even more strenuous to follow.    

Finally, a recent work by the second author and Osmond \cite{di2022bi} shows that $2$-categories with finite bilimits can be used to specify many fragments of first-order logic, in such a way that their functorial semantics recovers precisely their theories. We believe there is a possible unification of the theory of judgemental theories (of a certain shape) with the theory introduced in \cite{di2022bi}, but we shall defer such speculations to future work.

\bibliographystyle{alpha}
\normalsize
\bibliography{thebib}

\end{document}